\documentclass[11pt,a4paper,oneside,reqno]{amsart}
\usepackage{amsmath, amsthm, amsopn, amssymb}
\usepackage{mathtools, thmtools, thm-restate}
\usepackage{bbm}
\usepackage{enumitem}

\usepackage{multirow, makecell}
\usepackage{caption} 
\usepackage{tikz}
\usetikzlibrary{shapes, positioning, calc, backgrounds, fit}

\usepackage{stmaryrd} 
\usepackage{xcolor}

\usepackage{setspace}

\usepackage{geometry}
\usepackage{needspace}

\usepackage{hyperref}
\declaretheorem[name=Theorem,within=section]{thm}
\declaretheorem[name=Lemma,sibling=thm]{lemma}
\declaretheorem[name=Question,sibling=thm]{quest}

\newtheorem{cor}[thm]{Corollary}
\newtheorem{fact}[thm]{Fact}

\declaretheorem[name=Claim,sibling=thm]{claim}

\declaretheorem[name=Definition,sibling=thm,style=definition]{dfn}

\newcommand{\Gnp}{G_{n, p}^{(3)}}

\newcommand{\bG}{\mathbf{G}}
\newcommand{\bGs}{\mathbf{G^*}}

\newcommand{\bR}{\mathbf{R}}
\newcommand{\bRs}{\mathbf{R^*}}

\newcommand{\NN}{\mathbb{N}}
\newcommand{\RR}{\mathbb{R}}

\newcommand{\cA}{\mathcal{A}}
\newcommand{\cB}{\mathcal{B}}
\newcommand{\cC}{\mathcal{C}}
\newcommand{\cD}{\mathcal{D}}
\newcommand{\cE}{\mathcal{E}}
\newcommand{\cF}{\mathcal{F}}
\newcommand{\cG}{\mathcal{G}}

\newcommand{\cI}{\mathcal{I}}

\newcommand{\cM}{\mathcal{M}}

\newcommand{\cP}{\mathcal{P}}
\newcommand{\cQ}{\mathcal{Q}}
\newcommand{\cR}{\mathcal{R}}
\newcommand{\cS}{\mathcal{S}}
\newcommand{\cT}{\mathcal{T}}

\newcommand{\cY}{\mathcal{Y}}

\newcommand{\fC}{\mathfrak{C}}

\renewcommand{\Pr}{\mathbb{P}}
\newcommand{\Ex}{\mathbb{E}}

\newcommand{\1}{\mathbbm{1}} 

\newcommand{\eps}{\varepsilon}

\newcommand{\br}[1]{\llbracket{#1}\rrbracket}

\renewcommand{\le}{\leqslant}
\renewcommand{\ge}{\geqslant}

\newcommand{\cFQL}{\cF_Q^{\mathrm{low}}}

\newcommand{\cFQH}{\cF_Q^{\mathrm{high}}}

\newcommand{\cTFQH}{\Tilde{\cF}_Q^{\mathrm{high}}}

\newcommand{\deficit}{\mathrm{def}}
\newcommand{\ext}{\mathrm{ext}}
\renewcommand{\int}{\mathrm{int}}
\newcommand{\crit}{\mathrm{crit}}
\newcommand{\core}{\mathrm{core}}
\newcommand{\Core}{\mathrm{CORE}}

\newcommand{\eq}{\mathrm{eq}}

\thanks{This research was supported by the Israel Science Foundation grant 2110/22 and the ERC Consolidator Grant 101044123 (RandomHypGra).}

\author{Ilay Hoshen}
\title{Random Tur\'an Theorem for the Fano Plane}

\begin{document}
\begin{abstract}
    Let $F$ denote the Fano plane, the $3$-uniform hypergraph with $7$ vertices and $7$ edges. Frankl and F\"uredi, and independently Keevash and Sudakov, proved that the largest $F$-free subhypergraph of $K_n^{(3)}$ is bipartite. In this paper, we determine the sharp threshold for this property in the random setting. We show that for $\hat{p} = \Theta_F \cdot n^{-2/3} \left(\log n\right)^{1/6}$, where $\Theta_F$ is an explicit constant depending on $F$, we have
    \[
        \lim_{n \to\infty} \mathbb{P}\left(\text{Every largest $F$-free subhypergraph of $\Gnp$ is bipartite}\right) = \begin{cases}
            1, \quad (1+\epsilon) \hat{p} \le p = o(1), \\
            0, \quad \frac{1}{n^2} \ll p \le (1-\epsilon) \hat{p}.
        \end{cases}
    \]
    To the best of our knowledge, this work provides the first sharp threshold result obtained for a Turán-type problem in random hypergraphs.
\end{abstract}

\maketitle

\setcounter{tocdepth}{1}
\tableofcontents

\section{Introduction}
A central question in extremal combinatorics is to determine the structure of a largest $H$-free subgraph of a given host graph $G$. 
When $G$ is the complete graph $K_n$, the field is well-established, rooted in the classical theorems of Mantel and Tur\'an. 
In contrast, the case where the host graph is the complete $\ell$-uniform hypergraph $K_n^{(\ell)}$ is significantly more challanging; the extremal structure is currently known only for a limited number of hypergraphs $H$. 
In this paper, we focus on the random version of this question, specifically when $H$ is the Fano plane and $G$ is distributed as the random $3$-uniform hypergraph $G_{n, p}^{(3)}$.

The remainder of this introduction is organised as follows. 
We first present classical results for the deterministic case when the host graph is $K_n$. 
We then move to the random setting and provide a brief history of the problem for the binomial random graph $G_{n, p}$, the random graph on $n$ vertices where each edge is included independently with probability $p$. 
Turning to hypergraphs, we discuss known results for the deterministic case when the host hypergraph is $K_n^{(\ell)}$.
Finally, we address the random version for hypergraphs, presenting prior work and the main results of this paper.

\subsection*{The Deterministic and Random Graph Settings}
Mantel's theorem \cite{mantel1907} famously stated that the largest triangle-free subgraph of $K_n$ is bipartite. Tur\'an \cite{Tur41} generalised this by showing that the largest $K_r$-free subgraph of $K_n$ is $(r-1)$-partite. 
Further extending these results, Simonovits considered \textit{edge-critical} graphs, defined as those where there exists an edge $e \in H$ such that $\chi(H \setminus e) < \chi(H)$, where $\chi(H)$ denotes the chromatic number of $H$. In \cite{Sim68}, he showed that for any such non-bipartite graph $H$, the largest $F$-free subgraph of $K_n$ is $(\chi(H)-1)$-partite provided that $n$ is sufficiently large.

In 1990, Babai, Simonovits and Spencer \cite{BabSimSpe90} asked the following question.
\begin{quest}
    Suppose that $H$ is a non-bipartite and edge-critical graph. For what values of $p$ is it true that whp the largest $H$-free subgraph of $G_{n, p}$ is $(\chi(H)-1)$-partite?
\end{quest}

The primary result of \cite{BabSimSpe90} was that, for every integer $r \ge 2$, whp every largest $C_{2r + 1}$-free subgraph of $G_{n, p}$ is bipartite as long as $p \ge \frac{1}{2} - \epsilon_r$ for some small constant $\epsilon_r$ that depends on $r$. Brightwell, Panagiotou and Steger \cite{BriPanSte12} answered a challenging question in \cite{BabSimSpe90} and proved that, for every integer $r \ge 3$, whp every largest $K_r$-free subgraph of $G_{n, p}$ is $(r-1)$-partite already when $p \ge n^{-{c_r}}$ for some positive small constant $c_r$. 

To state results approaching the correct order of magnitude for this threshold, we first define the $\ell$-density of an $\ell$-uniform hypergraph $H$. Let $v(G)$ and $e(G)$ denote the number of vertices and edges of a graph $G$, respectively. The $\ell$-density of an $\ell$-uniform hypergraph $H$ is defined as
\[
    m_\ell(H) \coloneqq \max\left\{\frac{e(H') - 1}{v(H') - \ell} \colon H' \subseteq H,\ v(H') > \ell\right\}.
\]
A hypergraph $H$ is called \textit{$\ell$-balanced} if the maximum is achieved at $H'=H$. We say that $H$ is \textit{strictly $\ell$-balanced} if this maximum is uniquely achieved by setting $H'=H$.

In 2015, DeMarco and Kahn identified the correct threshold for cliques in two landmark papers \cite{DeMKah15Man, DeMKah15Tur}—addressing the case for $K_3$ and all larger cliques, respectively. They established that, for every integer $r \ge 3$, whp every largest $K_r$-free subgraph of $G_{n, p}$ is $(r-1)$-partite provided that $p \ge C_r n^{-1/m_2(K_r)} \left(\log n\right)^{\frac{1}{e(K_r)-1}}$ for a large enough constant $C_r$. The author and Samotij \cite{hoshen2023simonovits} extended this and showed that, for every non-bipartite, edge-crictical and strictly $2$-balanced graph $H$, there exists a constant $\Theta_H$ such that whp every largest $H$-free subgraph of $G_{n, p}$ is $(\chi(H)-1)$-partite provided that $p \ge (1+\epsilon) \Theta_H n^{-1/m_2(H)} \left(\log n\right)^{\frac{1}{e(H)-1}}$.
Notably, the constant $\Theta_H$ is explicitly presented in \cite{hoshen2023simonovits} and, at that time, the authors believed that this constant is optimal. 

One strategy for demonstrating that the property fails for smaller values of $p$, that is some largest $H$-free subgraph is not $(\chi(H)-1)$-partite, is the following: find a largest $(\chi(H)-1)$-partite subgraph $G' \subseteq G_{n, p}$ that can be extended to a larger $H$-free subgraph of $G_{n, p}$. 
In \cite{hoshen2024stabilitylargecutsrandom}, the author, Samotij, and Zhukovskii established that the constant $\Theta_H$ described in \cite{hoshen2023simonovits} is indeed optimal by proving a sharp threshold result at $p = \Theta_H n^{-1/m_2(H)} \left(\log n\right)^{\frac{1}{e(H)-1}}$ for the property that every largest $H$-free subgraph of $G_{n, p}$ is $(\chi(H)-1)$-partite. This work introduced a novel technique that, broadly speaking, provides a method for proving the existence of typical properties within max-cuts, or near-max-cuts, in $G_{n, p}$. 

Together, the results of \cite{hoshen2023simonovits} and \cite{hoshen2024stabilitylargecutsrandom} establish the sharp threshold for the aforementioned property. The full statement is given below.
\begin{thm}\cite{hoshen2023simonovits,hoshen2024stabilitylargecutsrandom}\label{thm:Simonovits-threshold}
    For every non-bipartite, strictly $2$-balanced and edge-critical graph $H$, there exists a constant $\Theta_H$\footnote{The constant $\Theta_H$ is explicitly defined in \cite{hoshen2023simonovits}.} such that the following holds for every $\epsilon > 0$.

    \begin{enumerate}[label=(\roman*), ref=\thethm(\roman*), itemsep=0pt, topsep=4pt]
        \item If
        \begin{align*}
            (1+\epsilon) \Theta_H \cdot n^{-1/m_2(H)} \left(\log n\right)^{\frac{1}{e(H)-1}} \le p \le 1,
        \end{align*} 
        then whp every largest $H$-free subgraph of $G_{n, p}$ is $(\chi(H)-1)$-partite.

        \item If
        \begin{align*}
            \frac{1}{n} \ll p \le (1-\epsilon) \Theta_H \cdot n^{-1/m_2(H)} \left(\log n\right)^{\frac{1}{e(H)-1}},
        \end{align*}     
        then whp every largest $H$-free subgraph of $G_{n, p}$ is not $(\chi(H)-1)$-partite.
        \end{enumerate}   
\end{thm}


\subsection*{Hypergraphs: Deterministic and Random}
We now turn our attention to hypergraphs. Two notable examples are the ``generalised triangle" $F_5$ and the Fano plane. The $F_5$ is a $3$-uniform hypergraph on five vertices $\{v_1, \dots, v_5\}$ with edges $\{v_1, v_2, v_3\}, \{v_1, v_4, v_5\}$ and $\{v_2, v_3, v_5\}$. The Fano plane is a $3$-uniform hypergraph on seven vertices where every pair of vertices is contained in exactly one edge (which is unique, up to isomorphism), see Figure \ref{figure:hypergraphs}.

\begin{figure}[h]
\centering
\begin{tikzpicture}
    \node (fig1) at (2.5,0) {
    \begin{tikzpicture}[scale=1.1, 
                         every node/.style={circle, fill=black, inner sep=1.2pt}] 
  \def\step{1}

  \coordinate (a) at (0,0);
  \coordinate (b) at (0,-\step);
  \coordinate (c) at (0,-2*\step);
  \coordinate (d) at (\step,0);
  \coordinate (e) at (2*\step,0);

  \node at (a) {};
  \node at (b) {};
  \node at (c) {};
  \node at (d) {};
  \node at (e) {};

  \node[draw=none, fill=none, left=2.5pt, font=\small\itshape] at (a) {};
  \node[draw=none, fill=none, left=2.5pt, font=\small\itshape] at (b) {};
  \node[draw=none, fill=none, left=2.5pt, font=\small\itshape] at (c) {};
  \node[draw=none, fill=none, above=2.5pt, font=\small\itshape] at (d) {};
  \node[draw=none, fill=none, above=2.5pt, font=\small\itshape] at (e) {};

  \draw[black, thick] (0,-\step) ellipse (0.32 and 1.55);

  \draw[black, thick] (\step,0) ellipse (1.55 and 0.32);

  \draw[black, thick, rounded corners=2pt]
    ($(b)+(0.28,-0.18)$) -- 
    ($(b)+(-0.28,-0.18)$) -- 
    ($(b)+(-0.28,0.18)$) --     
    ($(d)+(-0.23,0.28)$) -- 
    ($(e)+(0.23,0.28)$) -- 
    ($(e)+(0.23,-0.28)$) -- 
    cycle;
    \end{tikzpicture}
    };
    \node[above=6pt of fig1, font=\normalsize] {$F_5$};

    \node (fig2) at (7.5,0) {
    \begin{tikzpicture}[thick, scale=2.2, 
                         every node/.style={circle, fill=black, inner sep=1.4pt}] 
\coordinate (A) at (90:1);
\coordinate (B) at (210:1);
\coordinate (C) at (330:1);

\coordinate (D) at ($(A)!0.5!(B)$);
\coordinate (E) at ($(B)!0.5!(C)$);
\coordinate (F) at ($(C)!0.5!(A)$);

\coordinate (G) at (0,0);

\draw (A) -- (B) -- (C) -- cycle;

\draw (A) -- (E);
\draw (B) -- (F);
\draw (C) -- (D);

\draw (G) circle [radius=0.5];

\node (Apt) at (A) {};
\node (Bpt) at (B) {};
\node (Cpt) at (C) {};
\node (Dpt) at (D) {};
\node (Ept) at (E) {};
\node (Fpt) at (F) {};
\node (Gpt) at (G) {};

    \end{tikzpicture}
    };
    \node[above=6pt of fig2, font=\normalsize] {The Fano plane};
\end{tikzpicture}
\caption{The $3$-uniform hypergraphs $F_5$ and the Fano plane.}
\label{figure:hypergraphs}
\end{figure}
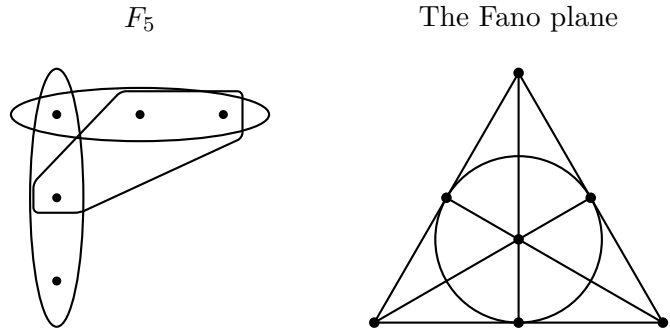

A hypergraph $G$ is said to be \textit{$r$-partite} if its vertices can be partitioned into $r$ sets such that there are no edges fully contained in one of the sets. Further, $G$ is said to be \textit{strongly $r$-partite} if its vertices can be partitioned into $r$ sets such that the intersection of every edge with each of the sets is of size at most one. Frankl and F\"uredi \cite{MR753720} established that the largest $F_5$-free subgraph of $K_n^{(3)}$ is strongly $3$-partite. For the Fano plane, F\"uredi and Simonovits \cite{MR2160414}, and independently Keevash and Sudakov \cite{MR2176425}, proved that the largest Fano-plane-free subgraph of $K_n^{(3)}$ is bipartite (that is, $2$-partite).

As extremal questions are still not well understood even when the host graph is the complete hypergraph, there are relatively few results for the random hypergraph setting. Balogh, Butterfield, Hu and Lenz \cite{MR3508721} proved that, whenever $p \ge C \frac{\log n}{n}$ for some large enough constant $C$, typically every largest $F_5$-free subhypergraph of $\Gnp$ is $3$-partite. This was later improved by Araujo, Balogh and Luo \cite{MR4663370} to the range $p \ge C \frac{\sqrt{\log n}}{n}$ which is optimal up to the value of the constant $C$.

\clearpage
In this paper, building on the approaches developed in \cite{hoshen2023simonovits} and \cite{hoshen2024stabilitylargecutsrandom}, we determine the sharp threshold for the property that every largest Fano-plane-free subhypergraph of $G_{n,p}^{(3)}$ is bipartite. Furthermore, we expect that the techniques introduced here are sufficiently robust to be applied to a wider class of hypergraphs beyond the Fano plane. From now on, let $F$ denote the Fano plane.

To describe the constant in the sharp threshold, let $K_2(m)$ be the complete bipartite $3$-uniform hypergraph with parts of size $a$, and let $K_2^+(m)$ be $K_2(m)$ plus a single edge. Let $N(F, G)$ the number of copies of $F$ in $G$, and define
\begin{align}\label{align:pi_F}
    \pi_F \coloneqq \lim_{m \to \infty} \frac{N\left(F, K_2^+\left(m\right)\right)}{m^{v(F)-3}}.
\end{align}
Finally, let $\Theta_F$ be the constant satisfying
\begin{align}\label{align:pi_F equality}
    \frac{\pi_F}{2^{v(F)-3}} \cdot \Theta_F^{e(F)-1} = 3 - \frac{1}{m_3(F)}.
\end{align}
Let us remark that in the case of the Fano plane, we have $m_3(F) = \frac{3}{2}$.

We are now ready to state the main result of this paper, which establishes the sharp threshold for the aforementioned property.
\begin{thm}\label{thm:main-theorem}
    Let $F$ be the Fano plane, let $\epsilon > 0$, and let $p \in (0, 1)$. 
    \begin{enumerate}[label=(\roman*), ref=\thethm(\roman*), itemsep=0pt, topsep=4pt]
        \item\label{thm:1-statement} If
        \begin{align}\label{align:p definition}
            (1+\epsilon) \Theta_F \cdot n^{-\frac{1}{m_3(F)}} \left(\log n\right)^{\frac{1}{e(F)-1}}  \le p \ll 1,
        \end{align} 
        then whp every largest $F$-free subhypergraph of $\Gnp$ is bipartite.

        \item\label{thm:0-statement} If
        \begin{align}\label{align:p definition 0}
            \frac{1}{n^2} \ll p \le (1-\epsilon) \Theta_F \cdot n^{-\frac{1}{m_3(F)}} \left(\log n\right)^{\frac{1}{e(F)-1}},
        \end{align}     
        then whp every largest $F$-free subhypergraph of $\Gnp$ is not bipartite. 
        \end{enumerate}  
\end{thm}

To the best of our knowledge, this work provides the first sharp threshold result for a Turán-type problem in random hypergraphs. As hypergraph extremal theory is challenging, even in the deterministic setting, the Turán densities and extremal structures for most hypergraphs remain unknown. While our primary focus is on the Fano plane, we believe that for any hypergraph whose deterministic extremal question is well-understood, the methods presented in this paper may be applicable. 

To provide some intuition for the constant $\Theta_F$, consider a fixed partition $\Pi$ of $\br{n} \coloneqq \{1, \dots, n\}$ into two parts and the corresponding bipartite subhypergraph $G' \subseteq G_{n,p}^{(3)}$ consisting of all edges in $G_{n,p}^{(3)}$ that have non-empty intersection with each of the parts of the partition $\Pi$. Since the Fano plane $F$ is not bipartite, $G'$ is necessarily $F$-free. Now, suppose there exists an edge $e \in G_{n,p}^{(3)}$ induced by one of the parts of $\Pi$. If there are no copies of $F$ in $G' \cup \{e\}$ that contain $e$, then $G' \cup \{e\}$ remains $F$-free, implying that the largest $F$-free subhypergraph of $G_{n,p}^{(3)}$ is strictly larger than the bipartite subhypergraph $G'$. 

Consequently, if $\Pi$ is a max-cut of $\Gnp$ (that is, a bipartite subgraph of largest size), then to prove Theorem \ref{thm:1-statement} we must require that every such internal edge $e$ participates in at least one copy of $F$ whose remaining $e(F)-1$ edges belong to $G'$. The constant $\Theta_F$ is precisely the value at which a union bound (and treating $\Pi$ as fixed) over all such potential internal edges succeeds, ensuring that whp no such edge can be added without creating a copy of the Fano plane.

It is also worth noting that the majority of this paper is dedicated to establishing the 1-statement, namely Theorem \ref{thm:1-statement}. As we shall see, the starting point for the proof is a stability result by Conlon and Gowers \cite[Theorem 10.24]{MR3548529}. Consequently, whenever one attempts to extend Theorem \ref{thm:1-statement} to other hypergraphs, this kind of stability result is essential for initiating the proof. It states that, typically, every largest $F$-free subhypergraph of $\Gnp$ is close to being bipartite.
\begin{thm}\label{thm:stability}\cite[Theorem 10.24]{MR3548529}
    Let $F$ be the Fano plane. For every $\beta > 0$, there exists $C > 0$ such that the following holds for every $p \ge C n^{-2/3}$. Whp, every largest $F$-free subhypergraph of $\Gnp$ can be made bipartite by deleting less than $\beta n^3 p$ edges.
\end{thm}

Last but not least, for the 0-statement, namely Theorem \ref{thm:0-statement}, much of the foundational framework was established in \cite{hoshen2024stabilitylargecutsrandom}. Here, we provide the two critical missing components required to complete the proof specifically for the Fano plane. First, we prove that $\Gnp$ typically admits a \textit{core}, that is two disjoint large sets of vertices that respect the configuration of every (nearly) max-cut of $\Gnp$. Second, we provide precise estimates for the concentration of the number of copies of $F \setminus \{e\}$ (the Fano plane missing a single edge) contained within a fixed cut in $\Gnp$.

\subsection{Notation}
As this paper deals exclusively with hypergraphs, we refer to them simply as graphs. For an integer $n > 0$, define $\br{n} \coloneqq \{1, \dots, n\}$.
Given a graph $G$, we denote its set of vertices and edges by $V(G)$ and $E(G)$, and their respective sizes by $v(G)$ and $e(G)$.
We treat graphs as sets of edges, so the \textit{size} of a graph is $e(G) = |E(G)|$. For a set of vertices $A \subseteq V(G)$, we let $G[A]$ denote the subgraph of $G$ induced by $A$, that is $G[A] = \{e \in G \colon e \subseteq A\}$.

For any set of vertices $S$, the \textit{degree} $d_G(S)$ is the number of edges in $G$ that contain $S$. In the case where $S = \{v\}$, we write $d_G(v)$. For every integer $i \in \mathbb{N}$, we denote by $\Delta_i(G)$ the maximum degree over all sets of vertices of size $i$ in $G$.

\subsection*{Cuts and Max-cuts} 
A cut $\Pi$ of $K_n^{(3)}$ is a partition of its vertices into two sets. For every collection $\Pi$ of disjoint sets of vertices, in particular a cut, define the following two quantities.
\begin{itemize}
    \item \textit{Internal edges}: $\text{int}(\Pi)$ is the set of edges fully contained within a single part of the partition.
    \item \textit{External edges}: $\text{ext}(\Pi) \coloneqq K_n^{(3)} \setminus \text{int}(\Pi)$ is the set of edges that \textit{cross} the cut $\Pi$.
\end{itemize}
We say that an edge $e$ \textit{crosses} a cut $\Pi$ if $e \in \ext(\Pi)$. The \textit{size} of a cut $\Pi$ in a graph $G$ is the number of edges of $G$ that cross $\Pi$, that is $e(G \cap \ext(\Pi))$. 

Given a collection of cuts $\cC$, we define the following. The maximum size of a cut in $G$ among $\mathcal{C}$ is defined as
\[
    b_{\cC}(G) \coloneqq \max_{\Pi \in \cC} e\left(G \cap \ext(\Pi)\right).
\] 
The set of \textit{max-cuts} of $G$ among $\cC$ is the set of cuts $\Pi \in \cC$ that achieves the maximum above.
The \textit{deficit} of a cut $\Pi \in \cC$ in $G$ is the difference between $b_{\cC}(G)$ and the size of $\Pi$ in $G$. That is,
\[
    \deficit_{\cC}(\Pi; G) \coloneqq b_{\cC}(G) - e(G \cap \ext(\Pi)).
\]
If the subscript of $\cC$ is omitted, then it is assumed that $\cC$ is the collecting of all possible cuts (unless stated otherwise).
Finally, for every $\delta > 0$, we say that a cut $\Pi = (A_1, A_2)$ is \textit{$\delta$-balanced} if $|A_i| \le (0.5 + \delta)n$ for each $i \in \br{2}$.

\subsection*{Colourings and Compatibility}
A $\br{2}$-coloured graph $Q$ is a graph where every vertex is assigned with a colour in $\br{2}$. For each $i \in \br{2}$, we let $V^i(Q)$ denote the set of vertices in $Q$  with colour $i$. We say that a $2$-tuple $\Pi = (A_1, A_2)$ of disjoint sets of vertices is \textit{compatible} with $Q$ if $V^i(Q) \subseteq A_i$ for each $i \in \br{2}$.

\section{Outline of The Proof of Theorem \normalfont{\ref{thm:1-statement}}}
In this section, we prove Theorem \ref{thm:1-statement}, relying on a key lemma --- Lemma \ref{lemma:main-lemma} --- whose proof is provided in the remainder of the paper.

Let $\epsilon > 0$ be a small constant, let $p$ satisfy \eqref{align:p definition} and let $G \sim \Gnp$. Our main goal is to show that whp every largest $F$-free subgraph of $G$ is bipartite. Let $H$ be a largest $F$-free subgraph of $G$ and let $\Pi = (A_1, A_2)$ be a max-cut of $H$. Note that if $e(H[A_1]) = e(H[A_2]) = 0$, then $H$ is bipartite and the proof is complete. Assume, without loss of generality, that $e(H[A_1]) \ge e(H[A_2])$. Observe that if $e(H[A_1]) > 0$ and at least one of the following holds, then we obtain a contradiction to the maximality of $H$ (as a largest $F$-free subgraph of $G$):
\begin{enumerate}[label=(C\arabic*)]
    \item There exists a bipartite subgraph of $G$ that is strictly larger than $H$.\label{item:graph bigger than H}
    \item There are more than $e(H[A_1]) + e(H[A_2])$ edges crossing $\Pi$ that belong to $G \setminus H$.\label{item:many crossing edges}
\end{enumerate}
Indeed, if \ref{item:graph bigger than H} holds, then it contradicts the maximality of $H$ since any bipartite graph is $F$-free. If \ref{item:many crossing edges} holds, then the size of the cut $\Pi$ in $G$ equals 
\begin{align*}
    e(H \cap \ext(\Pi)) + e((G \setminus H) \cap \ext(\Pi)) &>  e(H \cap \ext(\Pi))  + e(H[A_1]) + e(H[A_2]) = e(H).
\end{align*}
This again contradicts the maximality of $H$, as $G \cap \ext(\Pi)$ is a bipartite subgraph of $G$ larger than $H$.

The proof proceeds by contradiction. We show that whp, if $e(H[A_1]) > 0$, one of the two conditions above must occur. Consequently, we conclude that whp $e(H[A_1]) = 0$, implying that $H$ is bipartite. A natural approach would be to fix a cut $\Pi$ and a graph $H$ with a given number of edges, and then estimate the probability that neither \ref{item:graph bigger than H} nor \ref{item:many crossing edges} holds. If this probability were sufficiently small to allow for a union bound over all possible choices of $\Pi$ and $H$, the proof would be complete. However, since this probability depends on the size and structure of $H$, when $H$ contains just a few edges, the number of choices of $\Pi$ is far too large for a standard union bound. Thus this approach seems hopeless at first glance.

DeMarco and Kahn~\cite{DeMKah15Tur} introduced a clever argument that effectively overcomes the obstacle of the `huge' union bound over the cuts. This approach was further refined by the author and Samotij in~\cite{hoshen2023simonovits}. In practice, these ideas enable us to consider only the number of choices of $H$, rather than the number of possible cuts, as the primary contribution to the union bound. 

In what follows, we aim to estimate the probability that, for a fixed cut $\Pi = (A_1, A_2)$ and a graph $H$ with $e(H[A_1])>0$, no contradiction arises regarding the maximality of $H$. The probability bound that we obtain will be sufficiently small to overcome the union bound over all possible graphs $H$. In fact, we perform this estimation only for graphs $H$ that exhibit one of several suitable structures. The following two claims describe these structures and show that every such $H$ contains one of them as a subgraph. Since the proof of the first claim is a straightforward application of the probabilitic method, we defer it to the preliminaries.

\begin{claim}\label{claim: subgraph-of-Q-with-low-deg}
    For every two constants $\tilde{\epsilon}, \eta >0$ and every $3$-uniform graph $Q$ with $\Delta_1(Q) \le \eta n^2 p$ and $e(Q) \ge \log^2 n$, there exists a subgraph $\tilde{Q} \subseteq Q$ with
    \[
        \Delta_1(\tilde{Q}) \le \tilde{\epsilon} \frac{n^2 p}{\log n} \quad \text{and} \quad e(\tilde{Q}) \ge \frac{\tilde{\epsilon}}{4 \eta \log n} e(Q).
    \]
\end{claim}

Let $\tilde{\epsilon} > 0$ be a small constant, and let $\eta = \eta(\epsilon, \tilde{\epsilon}) > 0$ be a constant sufficiently smaller than $\tilde{\epsilon}$. The following claim identifies a subgraph of $H$ that is particulary convenient for our analysis. We note that although the first two items in the claim appear quite similar, maintaining them as distinct cases will be useful during the proof of Lemma \ref{lemma:main-lemma}. 
\begin{claim}\label{claim:subgraphs-of-I}
    If $e(H[A_1]) = o(n^3 p)$, $\Delta_1(H) \le 2n^2 p$, and $\Delta_2(H) \le 2np$, then there exists a $\br{2}$-coloured graph $Q \subseteq H$ satisfying one of the following.
    \begin{enumerate}
        \item $Q \subseteq H[A_1]$, all vertices of $Q$ are coloured $1$ and 
        \[
            \Delta_1(Q) \le \tilde{\epsilon} \frac{n^2 p}{\log n} \quad \text{and} \quad e(Q) \ge \frac{e(H[A_1])}{2}.
        \]               
        \item $Q \subseteq H[A_1]$, all vertices of $Q$ are coloured $1$ and 
        \[
            \Delta_1(Q) \le \tilde{\epsilon} \frac{n^2 p}{\log n} \quad \text{and} \quad e(Q) \ge \max\left\{\frac{\tilde{\epsilon}}{16\eta \log n} e(H[A_1]),\ \frac{n^2 p}{\log^3 n}\right\}.
        \]        
        \item $Q \subseteq H$ is a union of $k = o(n)$ stars with $e(Q[A_1]) \ge \frac{\eta}{16} \cdot e(H[A_1])$, constructed as follows. The star centres $v_1, \dots, v_k$ are in $A_1$, and the edges of $Q$ defined such that for $i \in \br{k}$ and $j \in \br{2}$, there are exactly $\eta n^2 p$ pairs of vertices $\{u_1, u_2\} \subseteq A_j \setminus \{v_1, \dots, v_k\}$ such that $\{v_i, u_1, u_2\} \in Q$. Each vertex in $Q$ is assigned colour $j$ if and only if it belongs to $A_j$. (See Figure \ref{figure:stars} for an illustration of $Q[A_1]$).
    \end{enumerate}
\end{claim}
\begin{figure}[h]
    \begin{center}
    \includegraphics[width=10cm]{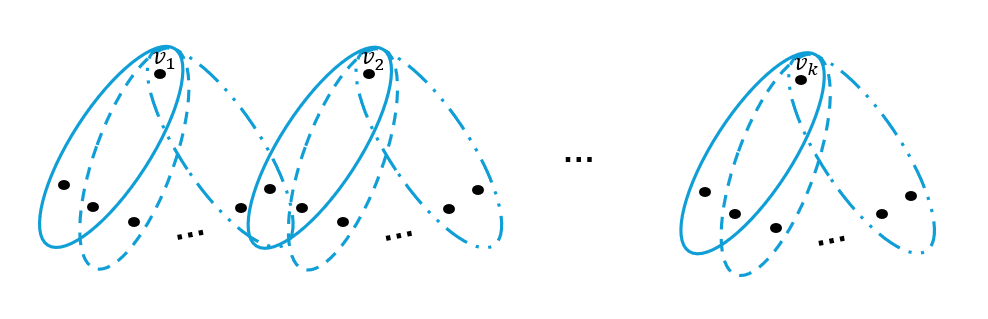}        
    \end{center}
    \caption{The graph $Q[A_1]$ in case $Q$ looks like stars with centres $v_1, \dots, v_k$. Also, there are similar stars with the same set of centres but the other vertices are in $A_2$.}
    \label{figure:stars}
\end{figure}
\begin{proof}
    First, if $\Delta_1(H[A_1]) \le \tilde{\epsilon} \frac{n^2 p}{\log n}$, then we may take $Q = H[A_1]$, which satisfies the first item of the claim. For the remainder of the proof, we assume this is not the case. Let $v_1, \dots, v_k \in A_1$ be the vertices of degree at least $2\eta n^2 p$ in $H[A_1]$. We split the proof into two cases depending on whether 
    \begin{align}\label{align:sum of large degrees}
        \sum_{i=1}^{k} d_{H[A_1]}(v_i) \le e(H[A_1])/8.
    \end{align}

    Assume that \eqref{align:sum of large degrees} holds. Note that for every $i \in \br{k}$, we must have $$d_{H[A_2 \cup \{v_i\}]}(v_i) \ge d_{H[A_1 \cup \{v_i\}]}(v_i) \ge 2\eta n^2 p.$$ Indeed, if this were not the case, moving $v_i$ to $A_2$ would result a cut of $H$ strictly larger than $\Pi$, contradicting its maximality. Since $\Delta_2(H) \le 2np$, then for each pair of distinct centres, there are at most $2np$ edges in $H$ containing them. Further, the assumption $e(H[A_1]) = o(n^3 p)$ implies that $k = o(n)$. Thus, for each $i \in \br{k}$, there are $o(n^2 p)$ edges in $H$ containing $v_i$ and any other centre $v_j$. 
    
    We may therefore take $Q$ to be a union of stars with centres $v_1, \dots, v_k$ such that for each centre and each $i \in \br{2}$, there are exactly $\eta n^2 p$ edges containing that centre and two additional vertices from $A_i \setminus \{v_1, \dots, v_k\}$. In this case, we have
    \[
        e(Q[A_1]) = k \cdot \eta n^2 p \ge \sum_{i=1}^{k} \frac{\eta d_{H[A_1]}(v_i)}{2} > \frac{\eta}{16} \cdot e(H[A_1]),
    \]    
    where the first inequality follows from the assumption that $d_{H[A_1]}(v) \le 2n^2 p$ for every $v \in A_1$. Thus, $Q$ satisfies the third item of the claim.

    Assume now that \eqref{align:sum of large degrees} does not hold. Let $Q \subseteq H[A_1]$ be a largest subgraph that satisfy $\Delta_1(Q) \le 2 \eta n^2 p$. Since every edge in $H[A_1] \setminus Q$ must contain at least one of the vertices $v_1, \dots, v_k$, we have: 
    \[
        e(Q) \ge e(H[A_1]) - e(H[A_1])/8 \ge e(H[A_1]) / 2 \quad \text{and} \quad \Delta_1(Q) \le 2\eta n^2 p.
    \]
    Moreover, there must exist at least one vertex in $A_1$ with degree at least $\tilde{\epsilon} \frac{n^2 p}{\log n}$ in $H[A_1]$; otherwise, the first item in the claim would already be satisfied. Hence, $e(Q) \ge \tilde{\epsilon} \frac{n^2 p}{\log n} \gg \log^2 n$. Let $\tilde{Q} \subseteq Q$ be the  subgraph guarenteed by Claim \ref{claim: subgraph-of-Q-with-low-deg} (applied with $\tilde{\epsilon}, 2\eta$). We have 
    \[
        \Delta_1(\tilde{Q}) \le \tilde{\epsilon} \frac{n^2 p}{\log n} \quad \text{and} \quad e(\tilde{Q}) \ge \frac{\tilde{\epsilon}}{8 \eta \log n} e(Q) \ge \frac{\tilde{\epsilon}}{16 \eta \log n} e(H[A_1]).
    \]
    Thus, $\tilde{Q}$ satisfies the second item in the claim. 
\end{proof}
For every $i \in \br{3}$, let $\cQ_i$ be the family of all possible subgraphs $Q \subseteq K_n^{(3)}$ that satisfy item ($i$) in Claim \ref{claim:subgraphs-of-I}. For convenience, set $\cQ \coloneqq \cQ_1 \cup \cQ_2 \cup \cQ_3$. 

Let $\delta > 0$ be a small constant. For every $Q \in \cQ$, let $\cC_Q$ be the family of all $\delta$-balanced cuts $\Pi' = (V_1, V_2)$ that are compatible with $Q$ (meaning $V^i(Q) \subseteq V_i$ for each $i \in \br{2}$). The following claim establishes that whp $\Pi$ is $\delta$-balanced. This claim will be a direct consequence of Lemma \ref{lemma:max_cut_balanced} which we cover in Section \ref{section:typical-props}.
\begin{claim}\label{claim:Pi is balanced}
    Whp $\Pi$ is $\delta$-balanced.
\end{claim}



Let us continue to the last ingredient of the proof of Theorem \ref{thm:1-statement}. We will need one additional definition before this.
Let $\beta = \beta(n) \ll 1$ the function that is implicit in the statement of Theorem \ref{thm:stability} invoked with $p \gg n^{-1/m_3(F)}$. Then, whp we have 
\begin{align}\label{align:bound on int Pi_H}
    e(H[A_1]) + e(H[A_2]) < \beta n^3 p.
\end{align}
Let $C > 0$ be a sufficiently large constant and, for every $Q \in \cQ$, define
\begin{align}\label{align:d_Q-def}
    d_Q \coloneqq \begin{cases}
        30 e(Q), &Q \in \cQ_1\ \text{and } p > C n^{-2/3} \left(\log n\right)^{1/6}, \\
        \min\left\{\beta n^3 p, \sqrt{\eta} e(Q) \log n\right\}, &Q \in \cQ_2\ \text{or } \left(Q \in \cQ_1\ \text{and }p \le C n^{-2/3} \left(\log n\right)^{1/6}\right), \\
        33 e(Q) / \eta, &Q \in \cQ_3.
    \end{cases}
\end{align}
As we shall see soon, we will need the following lower bound on $d_Q$ whose proof is defered to the end of this section.
\begin{claim}\label{claim:d_Q bound}
    Whp the following holds for every $Q' \in \cQ$. If $e(Q') > 0$, then $d_Q > e(H[A_1]) + e(H[A_2])$.
\end{claim}

For every $Q \in \cQ$ and a cut $\Pi \in \cC_Q$, we say that a copy of $F$ is \textit{$Q$-supported} and \textit{$\Pi$-crossing} if it contains exactly one edge in $Q \cap \int(\Pi)$ and all the other six edges belong to $\ext(\Pi)$. Furthermore,  two such $Q$-supported and $\Pi$-crossing copies of $F$ are said to be edge-disjoint if their intersection is contained within the edges of $Q$. That is, for any two such copies $F_1$ and $F_2$, we have $(F_1 \cap F_2) \setminus Q = \emptyset$.

The following lemma is the central component in the proof of Theorem \ref{thm:1-statement}. 
\begin{lemma}\label{lemma:main-lemma}
    Whp, for every $Q \in \cQ$ satisfying $Q \subseteq \Gnp$ and a cut $\Pi' \in \cC_Q$, at least one of the following holds.
    \begin{enumerate}
        \item $\deficit_{\cC_Q}\left(\Pi'; \Gnp\right) > d_Q$.
        \item There are at least $d_Q$ edge-disjoint, $\Pi'$-crossing and $Q$-supported copies of $F$ in $\Gnp$.
    \end{enumerate}    
\end{lemma}
We conclude proof of Theorem \ref{thm:1-statement}, relying on Lemma \ref{lemma:main-lemma}. First, we choose a suitable subgraph of $H[A_1]$ to which we apply Lemma \ref{lemma:main-lemma} together with $\Pi$. Observe that then $H$ can be made bipartite by deleteing all edges induced by either $A_1$ or $A_2$. Since $\Pi$ is a max-cut of $H$, we cannot make $H$ bipartite by deleting fewer edges. Consequently, by Theorem \ref{thm:stability}, whp $e(H[A_1]) = o(n^3 p)$. Moreover, standard concnetration inequalities imply that whp $\Delta_1(H) \le 2 n^2 p$ and $\Delta_2(H) \le 2np$. Thus, whp $H$ satisfies the assumptions of Claim \ref{claim:subgraphs-of-I}, and we may take $Q \subseteq H$ to be the subgraph guaranteed by that claim. Furthermore, by Claim \ref{claim:Pi is balanced}, whp $\Pi$ is $\delta$-balanced, implying that $\Pi \in \cC_Q$ (note that $\Pi$ is compatible with $Q$ by construction). Hence, the assumptions of Lemma \ref{lemma:main-lemma} are satisfied for $Q$ and $\Pi$.

Our goal is to show that whp $e(H[A_1]) = 0$. Using Lemma \ref{lemma:main-lemma}, we will show that whp, if $e(H[A_1]) > 0$, we arrive at a contradiction to the maximality of $H$ (as a largest $F$-free subgraph of $G$) via either \ref{item:graph bigger than H} or \ref{item:many crossing edges}. Observe that as long as the parameter $d_Q$ increases, we are be able to either indentify a cut larger than $\Pi$ or to find more edge-disjoint, $\Pi$-crossing and $Q$-supported copies of $F$ in $G$. As we shall see soon, both scenarios lead to the desired contradiction. The following claim gives a lower bound for $d_Q$.

Finally, we are ready to apply Lemma \ref{lemma:main-lemma} to $Q$ and $\Pi$. Assume deterministically that the consequence of Claim \ref{claim:d_Q bound} holds. If the first item of the lemma holds, that is $\deficit_{\cC_Q}(\Pi; G) > d_Q$, then there exists a bipartite subgraph of $G$ of size at least
\[
    e(G \cap \ext(\Pi)) + d_Q > e(H \cap \ext(\Pi)) + e(H[A_1]) + e(H[A_2]) = e(H),
\]
where the first inequality follows from Claim \ref{claim:d_Q bound}. Hence, \ref{item:graph bigger than H} is satisfied, contradicting to the maximality of $H$.

Alternatively, if the second item of Lemma \ref{lemma:main-lemma} holds, there exist at least $d_Q$ edge-disjoint, $\Pi$-crossing and $Q$-supported copies of $F$ in $G$. Since $H$ is $F$-free and $Q \subseteq H$, each of these copies must contain at least one edge crossing $\Pi$ that belongs to $G \setminus H$. Because these copies are edge-disjoint outside of $Q$, each copy provies a distinct edge in $G \setminus H$. By Claim \ref{claim:d_Q bound}, $d_Q > 2e(H[A_1])$. Therefore, \ref{item:many crossing edges} is satisfied, which again contradicts to the maximality of $H$. The proof of Theorem \ref{thm:1-statement} is now complete.

We finish with the proof of Claim \ref{claim:d_Q bound}.
\needspace{6\baselineskip}
\begin{proof}[Proof of Claim \ref{claim:d_Q bound}]
    First, observe that if $d_Q = \beta n^3 p$, then by \eqref{align:bound on int Pi_H}, whp $d_Q > e(H[A_1]) + e(H[A_2])$. Assume then $d_Q < \beta n^3 p$.

    Since $e(H[A_1]) \ge e(H[A_2])$, it suffices to show that $d_Q > 2e(H[A_1])$. If $Q \in \cQ_1$, then $e(Q) \ge \frac{e(H[A_1])}{2}$ and thus
\[
    d_Q \ge \min\{\sqrt{\eta} e(Q) \log n, 30 e(Q)\} \ge 2e(H[A_1]).
\]
 If $Q \in \cQ_2$, then $e(Q) \ge \frac{\tilde{\epsilon}}{16 \eta \log n} e(H[A_1])$ and thus
\[
    d_Q = \sqrt{\eta} e(Q) \log n \ge \frac{\Tilde{\epsilon}}{16 \sqrt{\eta}} \cdot e(H[A_1]) > 2e(H[A_1]),
\]
where the last inequality is true since $\sqrt{\eta}$ is sufficiently smaller than $\tilde{\epsilon}$. If $Q \in \cQ_3$, then $e(Q) \ge \frac{\eta}{16} e(H[A_1])$ and thus
\[
    d_Q = \frac{33}{\eta} e(Q) \ge \frac{33}{16} e(H[A_1]) > 2e(H[A_1]). \qedhere
\]
\end{proof}

\subsection{Constants}\label{section:constants}
Throughout the paper, we utilize several constants. Thus far, we have introduced $\epsilon$, which appears in the definition of $p$ in \eqref{align:p definition}; the constants $\eta$ and $\tilde{\epsilon}$, which bound the degrees of various subgraphs in $\cQ$; and $\delta$, which defines the family of $\delta$-balanced cuts. 

As we will introduce further constants, it is helpful to outline the hierarchy if their dependencies. We organise these constants into stages, where those defined in a given stage may depend on any constants defined in preceding stages.

Stage one:
\begin{itemize}
    \item $\epsilon$ --- appears in Theorem \ref{thm:main-theorem}.    
    \item $\hat{C}$ --- appears in the sparsification part in Lemma \ref{lemma:sparsification}.
    \item $C_{Low}$ --- guaranteed by Lemma \ref{lemma:mu-Delta-bounds-low-degree}.    
\end{itemize}

Stage two:
\begin{itemize}
\item $C$ --- determines what $p$ is considered being close to the threshold.        
\end{itemize}

Stage three:
\begin{itemize}        
    \item $\kappa$ --- appears in the parameters chosen for $Q \in \cQ$.
    \item $\alpha$ --- in the definition of rigidity. 
\end{itemize}

Stage four:
\begin{itemize}
    \item $\tilde{\epsilon}$ --- appears in the upper bound of the degree in $\cQ_1 \cup \cQ_2$.
\end{itemize}

Stage five:
\begin{itemize}
    \item $Z$ --- appears in Lemma \ref{lemma:main-switching-lemma}. 
\end{itemize}

Stage six:
\begin{itemize}
    \item $\eta$ --- appears in the degrees in $\cQ_3$.
    \item $\delta$ --- appears in the definition of $\delta$-balanced cuts.
\end{itemize}

Finally, we will also denote by $\beta = \beta(n) \ll 1$ the function that is implicit in the statement of Theorem \ref{thm:stability} invoked with $p \gg n^{-1/m_3(F)} = n^{-2/3}$, where $F$ denotes the Fano plane. That is, whp for every largest $F$-free subgraph $H \subseteq \Gnp$ , we have $e(H \cap \int(\Pi_H)) \le \beta n^3 p$ for every max-cut $\Pi_H$ of $H$.


\section{Tools and Preliminaries}\label{section:prel}

Given a set $V$ and a real $p \in [0,1]$, we denote by $V_p$ the random subset of $V$ obtained by independently retaining each element of $V$ with probability $p$.  Further, given a graph (recalling that we treat hypergraphs simply as graphs) $G$ with vertex set $V$, we define the following two quantities:
\[
  \mu_p(G) \coloneqq \sum_{A \in G} p^{|A|}
  \qquad
  \text{and}
  \qquad
  \Delta_p(G) \coloneqq \sum_{\substack{A, B \in G \\ A \neq B, A \cap B \neq \emptyset}} p^{|A \cup B|},
\]
where the second sum is over \emph{unordered} pairs of edges;  in other words, $\mu_p(G)$ is just the expected number of edges of $G[V_p]$ and $\Delta_p(G)$ is the expected number of pairs of distinct edges of $G[V_p]$ that intersect.

\begin{thm}\label{theorem:Janson}[{Janson's inequality~\cite{Jan90}}]
  Let $G$ be a graph on a finite vertex set $V$.  For all $p \in [0,1]$,
  \[
    \Pr(G[V_p] = \emptyset) \le \exp\big(-\mu_p(G)+\Delta_p(G)\big).
  \]
\end{thm}

Given a graph $G$, denote by $\nu(G)$ the matching number of $G$, that is, the largest size of a collection of pairwise-disjoint edges of $G$. We will use the two following versions of Janson's inequality which are proved in \cite{hoshen2023simonovits}.

\begin{cor}
  \label{cor:Janson-matchings}
  Let $G$ be a hypergraph on a finite vertex set $V$ and let $p \in [0,1]$.
  For every $\gamma \le 1/10$, writing $\mu$ and $\Delta$ for $\mu_p(G)$ and $\Delta_p(G)$, we have
  \[
    \Pr\big(\nu(G[V_p]) \le \gamma^2 \mu\big) \le \exp\big(-(1 - \gamma)\mu + 2 \Delta\big).
  \]
\end{cor}

\begin{cor}
  \label{cor:extended-Janson-matchings}  
  Let $G$ be a hypergraph on a finite vertex set $V$ and let $p \in [0,1]$.
  Writing $\mu$ and $\Delta$ for $\mu_p(G)$ and $\Delta_p(G)$ and letting $\Lambda \coloneqq \min\{\mu, \mu^2 / \Delta\}$, we have
  \[
    \Pr\big(\nu(G[V_p]) \le \Lambda/1000\big) \le \exp(-\Lambda/10).
  \]
\end{cor}

We also give here the proof of Claim \ref{claim: subgraph-of-Q-with-low-deg}.
\begin{proof}[Proof of Claim \ref{claim: subgraph-of-Q-with-low-deg}]
    We use the probabilistic method in order to show that such $\tilde{Q} \subseteq Q$ exists. Set $q \coloneqq \frac{\tilde{\epsilon}}{2\eta \log n}$. Let $\tilde{Q} \subseteq Q$ be a random subgraph such that every edge of $Q$ is kept with probability $q$, independently from all other edges.

    For every vertex $v \in V(Q)$, we have
    \begin{align*}
        \Pr\left(d_{\tilde{Q}}\left(v\right) \ge \tilde{\epsilon} \frac{n^2 p}{\log n}\right) &\le \Pr\left(\text{Bin}\left(\eta n^2 p, q\right) \ge \tilde{\epsilon} \frac{n^2 p}{\log n}\right) \le e^{-\Theta(\tilde{\epsilon} n^2 p / \log n)} \ll \frac{1}{n},
    \end{align*}
    where the last inequality is true by our assumption that $p$ satisfies \eqref{align:p definition}. There are at most $n$ possible choices for $v \in V(Q)$. Thus, by the union bound, whp there are no vertices of degree larger than $\tilde{\epsilon} \frac{n^2 p}{\log n}$ in $\tilde{Q}$.
    
    Furthermore, we have
    \begin{align*}
        \Pr\left(e(\tilde{Q}) \le 0.5 q \cdot e(Q)\right) = \Pr\left(\text{Bin}(e(Q), q) \le 0.5 q \cdot e(Q)\right) \le e^{-\Theta(e(Q) \cdot q)} = o(1),
    \end{align*}
    where the last inequality is true since $e(Q) \cdot q = \Omega(\log n)$ by our assumption that $e(Q) \ge \log^2 n$.

    Therefore, with positive probability $\tilde{Q}$ satisfies 
    \[
        \Delta_1(\tilde{Q}) \le \tilde{\epsilon} \frac{n^2 p}{\log n} \quad \text{and} \quad e(\tilde{Q}) \ge 0.5 q \cdot e(Q) = \frac{\tilde{\epsilon}}{4 \eta \log n} e(Q),
    \]
    and thus there exists $\tilde{Q} \subseteq Q$ satisfying the desired properties of the claim.
\end{proof}

\section{Typical Properties of $\Gnp$}\label{section:typical-props}
In this section, we establish several typical properties of $\Gnp$ and $G_{n, m}^{(3)}$ that will be utilised throughout the paper. The first subsection focuses on the concentration of the number of edges crossing a cut. In the second subsection, we address the concentration of the number of copies of various subgraphs containing a specific fixed edge.

\subsection{Number of edges crossing a cut}

\begin{lemma}\label{lemma:Chernoff_subsets}
  Suppose that $G \sim G_{n, m}^{(3)}$ for some $m \in \{0, \dotsc, \binom{n}{3}\}$. Then, letting $p = m/\binom{n}{3}$,
  \[
    \Pr\left(\exists U\,\,\left|e(G[U]) - \binom{|U|}{3}p \right| > |U|^{1.5}\sqrt{np}\right) \le e^{-n}.
  \]
\end{lemma}

\begin{proof}
  We may clearly assume that $m \ge 1$.
  Fix a nonempty $U\subseteq\br{n}$.  Standard estimates on lower tail probabilities of the binomial distribution yield\footnote{See \cite[Section~6]{Hoe63}, which argues that the hypergeometric distribution is at least as concentrated as the binomial distribution with the same parameters.}
  \[
    \Pr\left(\left|e(G[U]) - \binom{|U|}{3} p\right| > |U|^{1.5} \sqrt{np}\right)\leq 2\exp\left(-\frac{|U|^3np}{|U|^3 p/3! }\right)\le e^{-4n}.
  \]
  The union bound over all $U$ finishes the proof.
\end{proof}

In fact, we will only use the following corollary which is an immediate consequence of \ref{lemma:Chernoff_subsets} and the fact that, from convexity, we have $\sum_{i} x_i^{1.5} \le (\sum_i x_i)^{1.5}$ for every choice of positive integers $x_i$.

\begin{cor}\label{cor:equiv_classes_edges_count}
  Suppose that $G \sim G_{n, m}^{(3)}$ for some $m \in \{0, \dotsc, \binom{n}{3}\}$.
  With probability at least $1 - e^{-n}$, for every $2$-partition $\Pi$ of $\br{n}$,
  \[
    \left|e\left(G \cap \int(\Pi)\right) - |\int(\Pi)| \cdot p \right| \le n^{1.5}\sqrt{np} \quad \text{and} \quad \left|e\left(G \cap \ext(\Pi)\right) - |\ext(\Pi)| \cdot p \right| \le n^{1.5}\sqrt{np}.
  \]
\end{cor}

One important consequence of \ref{cor:equiv_classes_edges_count} is that every cut of $G_{n. m}$ with small deficit must be balanced.

\begin{lemma}\label{lemma:max_cut_balanced}
  Suppose that $G \sim G_{n, m}$ for some $m \in \{0, \dotsc, \binom{n}{3}\}$.
  The following holds with probability at least $1-e^{-n}$. Each part of every $2$-cut of $G$ with deficit at most $d$ has at most $\frac{n}{2}+4\max\left\{\left(\frac{d}{m}\right)^{1/2}, \left(\frac{n}{m}\right)^{1/4}\right\} \cdot n$ vertices.
\end{lemma}

\begin{proof}[Proof of~\ref{lemma:max_cut_balanced}]
  Assume the assertion of \ref{cor:equiv_classes_edges_count}, which holds with probability at least $1-e^{-n}$.
  Let $\Pi = (A_1, A_2)$ be a cut with deficit at most $d$ and let $\eps$ be the number satisfying $|A_1| = n/2+\eps n$. We have\begin{align*}
    |\ext(\Pi)| &= \frac{(n/2 + \epsilon n)(n/2 + \epsilon n - 1)}{2}(n/2 - \epsilon n) + \frac{(n/2 - \epsilon n)(n/2 - \epsilon n - 1)}{2}(n/2 + \epsilon n) \\
    &= \frac{n^3}{2^3} - \frac{\epsilon^2 n^3}{2} - \frac{n^2}{4} + \epsilon^2 n^2 \le \frac{3}{4} \binom{n}{3} - 3 \epsilon^2 \binom{n}{3} + n^2.
  \end{align*}  

  On the one hand, since every graph $G$ with $m$ edges satisfies $b(G) \ge \frac{3}{4} m$ (to see this, consider a uniformly random cut), the size of $\Pi$ in $G$ is at least $\frac{3}{4} m - d$.
  On the other hand, by the assumed assertion of \ref{cor:equiv_classes_edges_count}, letting $p \coloneqq m / \binom{n}{3}$,
  \[
    e(G \cap \ext(\Pi)) \le |\ext(\Pi)| \cdot p + n^{1.5}\sqrt{np} \le \left(\frac{3}{4} - 3 \epsilon^2 \right) m + n^2 p + n^{1.5}\sqrt{np}.
  \]
  This yields
  \[
    3\eps^2m \le d + n^2 p + n^{1.5}\sqrt{np} \le d + 2 n^{1.5}\sqrt{np} \le d + 6\sqrt{nm},
  \]
  which means that
  \[
    \max_i |A_i| - \frac{n}{2} = \eps n \le 4\max\left\{\sqrt{\frac{d}{m}}, \left(\frac{n}{m}\right)^{1/4}\right\} \cdot n,
  \]
  as desired.
\end{proof}


\subsection{Number of subgraphs lying on an edge}

Suppose that $\cG$ is a (hyper)graph on vertex set $V$. Given a vertex $e \in V$, the link graph of $e$ is the graph
\[
    \partial_e(\cG) \coloneqq \{A \setminus \{e\} \colon e \in A \in \cG\}.
\]
We further define $\partial(\cG) \coloneqq \bigcup_{e \in V} \partial_e(\cG)$. 
Moreover, recall that for a given subset $W \subseteq V$, we defined $\cG[W]$ to be the graph induced by $W$. Since we will often consider induced subgraph in various link graphs, we use the convention that the operators $\partial_e$ and $\partial$ bind stronger than the operation of taking induced graphs, that is, $\partial \cG[W] = (\partial \cG)[W]$.

Denote by $\cF$ the set of all copies of $F$ in $K_n^{(3)}$. Note that $\cF$ is a $7$-uniform hypergraph on vertex set $\binom{\br{n}}{3}$. Then, for a graph $G$, we have that $\cF[G]$ is the set of all copies of $F$ in $G$. Further, $\partial \cF[G]$ is the set of all copies in the graph $G$ isomorphic to $F$ minus some edge. Last example is that, for every $e \in K_n^{(3)}$, we have that $\partial \partial_e \cF[G]$ is the set of all copies in the graph $G$ isomorphic to $F$ minus some two edges such that one of them is $e$. 

We start with the following technical claim.
\begin{claim}\label{claim:strictly-balanced}
    Let $\ell \ge 2$ be an integer and let $H$ be an $\ell$-uniform $\ell$-balanced graph. Suppose that $p \ge C n^{-1/m_\ell(H)}$ for some $C > 0$. Then,
    \[
        n^{v(\tilde{H})} p^{e(\tilde{H})} \ge C^{e(\tilde{H}) - 1} n^{\ell} p,
    \]
    for every $\emptyset \neq \tilde{H} \subseteq H$. Moreover, if $H$ is also stricty $\ell$-balanced, there exists $\lambda > 0$ such that 
    \[
        n^{v(\tilde{H})} p^{e(\tilde{H})} \ge C^{e(\tilde{H}) - 1} n^{\ell + \lambda} p,
    \]
    for every $\emptyset \neq \tilde{H} \subseteq H$ with $\ell < v(\tilde{H}) < v(H)$.
\end{claim}
\begin{proof}
    Let $\emptyset \neq \tilde{H} \subseteq H$. Of course, the claim holds whenever $e(\tilde{H}) = 1$ so we may assume that $e(\tilde{H}) > 1$. By our assumption on $p$, we have
    \[
        n^{v(\tilde{H}) - \ell} p^{e(\tilde{H}) - 1} \ge n^{v(\tilde{H}) - \ell} \left(C n^{-1/m_\ell(F)}\right)^{e(\tilde{H}) - 1} = C^{e(\tilde{H}) - 1} n^{v(\tilde{H}) - \ell - \frac{e(\tilde{H}) - 1}{m_\ell(H)}}.
    \]
    By the definition of being an $\ell$-balanced, we have
    \[
        v(\tilde{H}) - \ell \ge \frac{e(\tilde{H}) - 1}{m_\ell(H)}.
    \]
    Moreover, if $H$ is strictly $\ell$-balanced, the inequality is strict whenever $\tilde{H} \neq H$.
\end{proof}

The following claim gives probability bound on the upper tail of sizes of $\partial\partial_e \cF\left[\Gnp\right]$ and $\partial \cF_v \left[\Gnp\right]$, where $\cF_v$ denotes all copies of $F$ in $K_n^{(3)}$ that contain the vertex $v$.
\needspace{6\baselineskip}
\begin{claim}\label{claim:Fano minus edge}
    Let $p \ge n^{-2/3}$. Then,
    \begin{enumerate}
        \item For every edge $e \in K_n^{(3)}$, 
        \[
            \Pr\left(\left|\partial\partial_e \cF\left[\Gnp\right]\right| \ge 2n^4 p^5\right) \le \frac{1}{n^8}.
        \]        
        \item For every vertex $v \in \br{n}$, 
        \[
            \Pr\left(\left|\partial \cF_v \left[\Gnp\right]\right| \ge 2n^6 p^6\right) \le \frac{1}{n^8}.
        \]
    \end{enumerate}
\end{claim}
\begin{proof}
    Let $p \ge n^{-2/3}$. Both parts of the claim are proved using high moments method (see \cite{JanOleRuc04}). For every set of edges $A \subseteq K_n^{(3)}$, denote by $Y_A$ the indicator random variable of the event that $A \subseteq \Gnp$.
    
    For every edge $e \in K_n^{(3)}$ and every vertex $v \in \br{n}$, let
    \[
        X_e \coloneqq \left|\partial \partial_e\cF\left[\Gnp\right]\right| \quad \text{and} \quad X_v \coloneqq \left|\partial \cF_v\left[\Gnp\right]\right|.
    \]
    We have
    \[  
        X_e = \sum_{A \in \partial \partial_e(\cF)} Y_A \quad \text{and} \quad X_v = \sum_{A \in \partial \cF_v} Y_A.
    \]

    Let $C > 0$ be a large constant and set $L = C \log n$. We have
    \begin{align*}
        \mathbb{E}\left[X_e^L\right] = \sum_{A_1, \dots, A_{L-1} \in \partial \partial_e(\cF)} \mathbb{E}\left[\prod_{i=1}^{L-1} Y_{A_i}\right] \sum_{A_L \in \partial \partial_e(\cF)} \mathbb{E}\left[Y_{A_L} \mid \prod_{i=1}^{L-1} Y_{A_i} = 1\right].
    \end{align*}
    The condition $\prod_{i=1}^{L-1} Y_{A_i} = 1$ implies that $\bigcup_{i=1}^{L-1} A_i \subseteq \Gnp$. Note that $e\left( \bigcup_{i=1}^{L-1} A_i \right) \le 7L$. We claim that, for every subgraph $H \subseteq K_n^{(3)}$ with $e(H) \le 7L$, the following holds:
    \begin{align*}
        \sum_{A_L \in \partial \partial_e\cF} \mathbb{E}\left[Y_{A_L} \mid H \subseteq \Gnp\right] \le \mathbb{E}\left[X_e\right] + (7L) n^2 p^4 + (7L)^3 n p^2 + (7L)^6. 
    \end{align*}
    Indeed, the terms in the right-hand side account for the possible intersections between $A_L$ and $H$ as follows.
    \begin{itemize}
        \item \textbf{Zero edges in common:} The term $\mathbb{E}\left[X_e\right]$ provides an upper bound on the expected number of copies $A_L \in \partial \partial_e \cF$ in $\Gnp$ that share no edges with $H$.
        \item \textbf{Exactly one edge in common:} The second term bounds the expected number of copies $A_L$ sharing exactly one edge with $H$.        
        There are at most $7L$ choices for this shared edge. As there are no vertex-disjoint edges in $F$ and moreover, every pair of vertices belong to exactly one edge, this shared edge from $H$ must intersect $e$ at exactly one vertex. Consequently, there are $n^2$ choices for the remaining vertices of $A_L$. Further, these vertices should induce four additional edges in $\Gnp \setminus H$, leading to the $p^4$ factor.
         
        \item \textbf{Two or three edges in common:} The third term bounds the expected number of copies $A_L$ that share two or three edges from $H$. We have at most $(7L)^3$ ways to choose these edges from $H$. As $A_L \in \partial \partial_e\cF$, these edges must contain at least three vertices not in $e$. Consequently, there are $n$ choices for the remaining vertices of $A_L$. Further, these vertices should induce two more additional edges in $\Gnp \setminus H$, yielding the $p^2$ factor.
        \item \textbf{Four or more edges in common:} The last term bounds the cases where $A_L$ shares at least four edges with $H$. Any six vertices from $F$ induce exactly four edges. As $A_L \in \partial \partial_e\cF$, there are no remaining vertices to choose. Hence, there are at most $(7L)^6$ ways to fix this structure.
    \end{itemize}

    Since $\mathbb{E}\left[X_e\right] = \Theta(n^4 p^5) \gg \log^6 n (n^2 p^4 + np^2 + 1)$, we have
    \[
        \sum_{A_L \in \partial \partial_e \cF} \mathbb{E}\left[Y_{A_L} \mid H \subseteq \Gnp\right] \le (1 + o(1)) \mathbb{E}\left[X_e\right].
    \]
    Hence, by induction, we get
    \[
        \mathbb{E}\left[X_e^L\right] \le \left((1+o(1)) \mathbb{E}\left[X_e\right]\right)^L.
    \]

    Let us upper bound the $L$-th moment of $X_v$,
    \begin{align*}
        \mathbb{E}\left[X_v^L\right] = \sum_{A_1, \dots, A_{L-1} \in \partial \cF_v} \mathbb{E}\left[\prod_{i=1}^{L-1} Y_{A_i}\right] \sum_{A_L \in \partial \cF_v} \mathbb{E}\left[Y_{A_L} \mid \prod_{i=1}^{L-1} Y_{A_i} = 1\right].
    \end{align*}
    Similarly to before, the condition $\prod_{i=1}^{L-1} Y_{A_i} = 1$ implies that $\bigcup_{i=1}^{L-1} A_i \subseteq \Gnp$. Note that $e\left( \bigcup_{i=1}^{L-1} A_i \right) \le 7L$. We claim that, for every subgraph $H \subseteq K_n^{(3)}$ with $e(H) \le 7L$, the following holds:
    \begin{align*}
        \sum_{A_L \in \partial \cF_v} \mathbb{E}\left[Y_{A_L} \mid H \subseteq \Gnp\right] \le \mathbb{E}\left[X_v\right] + O\left(L n^4 p^5 + L^2 n^2 p^4 + L^3 n p + L^4\right). 
    \end{align*}
    Indeed, the terms on the right-hand side can be explained as follows:
    \begin{itemize}
        \item \textbf{Zero shared edges:} The term $\mathbb{E}\left[X_v\right]$ bounds the expected number of copies of $A_L \in \partial \cF_v$ that are edge-disjoint from $H$.
        \item \textbf{Exactly one shared edge:} There are at most $7L$ choices for the edge in $H$. This shared edge must use at least two more vertices other than $v$. To complete the copy of $A_L$, we choose the remaining four vertices in at most $n^4$ ways and require five more edges to be present in $\Gnp \setminus H$.
        \item \textbf{Exactly two shared edges:} There are at most $(7L)^2$ choices for these two shared edges. These edges provide at least four more vertices other than $v$. Thus, it remains to choose at most two more vertices in at most $n^2$ ways and require four more edges to be present in $\Gnp \setminus H$.
        \item \textbf{Three to five shared edges:} There are at most $(7L)^5$ choices for these shared edges. These edges provide at least five additional vertices other than $v$. It is left to choose at most one more vertex in at most $n$ ways and require at most one more edge to be present in $\Gnp \setminus H$.
        \item \textbf{Exactly six shared edges:} We have at most $(7L)^6$ choices for these shared edges. These edges must use another six vertices other than $v$.
    \end{itemize}

    Since $\mathbb{E}\left[X_v\right] = \Theta(n^6 p^6) \gg \log^6 n (n^4 p^5 + n^2p^4 + n p + 1)$, we have
    \[
        \sum_{A_L \in \partial \cF_v} \mathbb{E}\left[Y_{A_L} \mid H \subseteq \Gnp\right] \le (1 + o(1)) \mathbb{E}\left[X_v\right].
    \]
    Hence, by induction, we get
    \[
        \mathbb{E}\left[X_v^L\right] \le \left((1+o(1)) \mathbb{E}\left[X_v\right]\right)^L.
    \]

    Lastly, we also have $\mathbb{E}\left[X_e\right] \le n^4 p^5$ and $\mathbb{E}\left[X_v\right] \le n^6 p^6$. Thus, by Markov's inequality,
    \[
        \Pr\left(X_e \ge 2 n^4 p^5\right) \le \frac{\mathbb{E}\left[X^L_e\right]}{\left(2 n^4 p^5\right)^L} \le \left(\frac{1+o(1)}{2}\right)^L \ll \frac{1}{n^8},
    \]
    and
    \[
        \Pr\left(X_v \ge 2 n^6 p^6\right) \le \frac{\mathbb{E}\left[X^L_v\right]}{\left(2 n^6 p^6\right)^L} \le \left(\frac{1+o(1)}{2}\right)^L \ll \frac{1}{n^8},
    \]
    where the last inequalities are true whenever $C$ is sufficiently large.
\end{proof}



\section{Fixed Cuts}\label{section:fixed-cuts}
This section provides the first ingredient for the proof of Lemma \ref{lemma:main-lemma}. Once we fix a graph $Q \in \cQ$ and a cut $\Pi \in \cC_Q$ which is compatible with $Q$, our primary tool for demonstrating the existence of numerous edge-disjoint, $\Pi$-crossing and $Q$-supported copies of $F$ in $\Gnp$ is Janson's inequality (see Theorem \ref{theorem:Janson} and its derivatives, Corollaries \ref{cor:Janson-matchings} and \ref{cor:extended-Janson-matchings}). To obtain the necessary probability bounds, we must analyse the following two quantities.
\begin{enumerate}
    \item The expected number of $\Pi$-crossing and $Q$-supported copies of $F$ in $\Gnp$.
    \item The expected number of pairs of $\Pi$-crossing and $Q$-supported copies of $F$ that have a non-empty intersection in $\Gnp$.
\end{enumerate}

The expectations above may depend on the degrees of $Q$. While the first expectation remains well-behaved, the second quantity can become problematic if the degrees of $Q$ are too large. To address this, we partition our analysis into two cases based on the structural properties of $Q$ given by Claim \ref{claim:subgraphs-of-I}. The first subsection treats $Q \in \cQ_1 \cup \cQ_2$ --- we call this the low degree case, while the second subsection handles with $Q \in \cQ_3$ --- we call this the high degree case. 

Notably, the estimates derived in this section, equiqqed with Corollaries \ref{cor:Janson-matchings} and \ref{cor:extended-Janson-matchings}, are sufficient to prove Lemma \ref{lemma:main-lemma} for a fixed cut $\Pi$ and all $Q \in \cQ$ such that $\Pi \in \cC_Q$. However, at this point, we cannot hope to overcome the union over all the cuts as well and this will be the purpose of Sections \ref{section:rigidity}-\ref{section:proof-of-switching-lemma}.

\subsection{Low Degree}\label{subsection:low degree}
Recall that $\cF$ is the family of all copies of $F$ in $K_n^{(3)}$. For every $Q \in \cQ_1 \cup \cQ_2$, let $\cFQL$ be the set of all subgraphs $F' \setminus Q$ where $F' \in \cF$ contains exactly one edge of $Q$. For a $2$-tuple $\Pi = (A_1, A_2)$ of disjoint vertex sets and a graph $Q \in \cQ_1 \cup \cQ_2$ with $V(Q) \subseteq A_1$, recall that $\cFQL[\ext(\Pi)]$ denotes the collection of copies $F' \in \cFQL$ whose edges belon entirely to $\ext(\Pi)$ (i.e. $F' \subseteq \ext(\Pi)$). Note that, for every $F' \in \cFQL[\ext(\Pi)]$, there exists an edge $f \in Q$ such that $\{f\} \cup F'$ is a $Q$-supported and $\Pi$-crossing copy of $F$. Define the following probabilistic quantities:
\begin{align*}
    \mu_p\left(\cFQL[\ext(\Pi)]\right) &\coloneqq \sum_{F' \in \cFQL\left[\ext(\Pi)\right]} \Pr\left(F' \subseteq \Gnp\right) \quad \text{and}\\
    \Delta_p\left(\cFQL\right) &\coloneqq \sum_{\substack{F' \neq F'' \in \cFQL \\ F' \cap F'' \neq \emptyset}} \Pr\left(F' \cup F'' \subseteq \Gnp\right).
\end{align*}
The central lemma of this subsection establishes sharp bounds on the these quantities.
\needspace{6\baselineskip}
\begin{lemma}\label{lemma:mu-Delta-bounds-low-degree}
    There exists a positive constant $C_{Low} > 0$ such that, for every sufficiently small $\tilde{\epsilon} > 0$, the following folds for every $p \ge n^{-2/3}$, every $2$-tuple $\Pi = (A_1, A_2)$ of disjoint sets of vertices of size at least $n/4$ each and for every $Q \in \cQ_1 \cup \cQ_2$ with $V(Q) \subseteq A_1$.
    \begin{enumerate}
        \item $\mu_p\left(\cFQL[\ext(\Pi)]\right) \ge \left(\frac{5}{4} - o(1)\right) \cdot e(Q) \cdot \left(\min\left\{|A_1|, |A_2|\right\}\right)^{4} \cdot p^6.$
        \item There exists $\lambda > 0$ such that 
        \begin{align*}
            \frac{\Delta_p\left(\cFQL\right)}{\mu_p\left(\cFQL[\ext(\Pi)]\right)} \le \frac{C_{Low}}{3}\left(\Delta_1(Q) \cdot n^{2} p^{5} + \Delta_2(Q) \cdot n^{3 - \lambda} p^{5} + n^{4 - \lambda} p^{6}\right) \le \frac{C_{Low} \cdot \tilde{\epsilon} \cdot n^4 p^6}{\log n}.
        \end{align*}
    \end{enumerate}
\end{lemma}
\begin{proof}
    For the first item, we have 
    \[
        \mu_p\left(\cFQL[\ext(\Pi)]\right) = \sum_{F' \in \cFQL[\ext(\Pi)]} \Pr\left(F' \subseteq \Gnp\right) = \left|\cFQL[\ext(\Pi)]\right| \cdot p^6,
    \]
    where the last equality follows from the fact that $\cFQL[\ext(\Pi)]$ is a family of subgraphs of $K_n^{(3)} \setminus Q$, each consisting of exactly six edges. To estimate $\left|\cFQL[\ext(\Pi)]\right|$, note that for every $F' \in \cFQL[\ext(\Pi)]$, there exists some $f \in Q$ such that $F' \cup \{f\} \cong F$. There are $e(Q)$ choices for such an edge. Once $f \in Q$ is fixed, we consider the possible choices for the remaining four vertices of $F'$. Observe that $F'$ can take one of the following two forms:
    it either uses one additional vertex from $A_1$ and  three from $A_2$, or it uses four vertices from $A_2$. Furthermore, given an edge of $Q$ and four such additional vertices, there are $3!$ ways to form a copy of $F'$ such that $F' \cup \{f\} \cong F$ using those seven vertices. This is illustrated in Figure \ref{figure:copies of F}.
    
\begin{figure}[h]
    \centering
    \includegraphics[width=10cm, trim=0 35pt 0 35pt, clip]{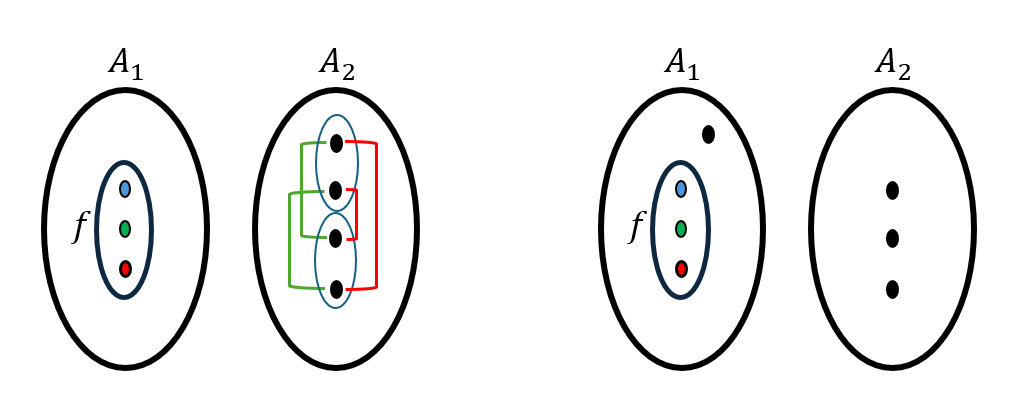}       
    \caption{The figure on the left presents a copy of $F$ lying on an edge $f \in Q$ with four additional vertices in $A_2$. The blue vertex, together with the two vertices in the first blue set in $A_2$, forms an edge, as does the blue vertex with the two vertices in the second blue set. The same structure applies to the green and red vertices with their corresponding sets of size two in $A_2$ of matching colours. Observe that there are $3!$ ways to construct a copy of the Fano plane in such a way.}
    \label{figure:copies of F}
\end{figure}

    
    Once we fix such a copy $F' \in \cFQL[\ext(\Pi)]$, the probability that it is contained in $\Gnp$ is $p^6$.   
    Hence,
    \begin{align*}
        \mu_p\left(\cFQL[\ext(\Pi)]\right) &= (1+o(1)) e(Q) \left[\binom{|A_2|}{4} \cdot 3! \cdot p^6 + \binom{|A_2|}{3} \cdot (|A_1| - 3) \cdot 3! \cdot p^6\right] \\
        &\ge \left(\frac{5}{4} - o(1)\right) \cdot e(Q) \cdot \left(\min\left\{|A_1|, |A_2|\right\}\right)^{4} \cdot p^6.
    \end{align*}    

    Moving to the second item of the lemma, we have
    \begin{align*}
        \Delta_p\left(\cFQL\right) &= \sum_{F' \in \cFQL} \Pr\left(F' \subseteq \Gnp\right) \sum_{\substack{F' \neq F'' \in \cFQL \\ F' \cap F'' \neq \emptyset}} \Pr\left(F'' \subseteq \Gnp \mid F' \subseteq \Gnp\right) \\
        &\le \left|\cFQL\right| \cdot p^{6} \cdot \max_{F' \in \cFQL} \sum_{\substack{F' \neq F'' \in \cFQL \\ F' \cap F'' \neq \emptyset}} p^{6 - e(F' \cap F'')}\\
        &\le \sqrt{C_{Low}} \cdot \mu_p\left(\cFQL[\ext(\Pi)]\right) \cdot \max_{F' \in \cFQL} \sum_{\substack{F' \neq F'' \in \cFQL \\ F' \cap F'' \neq \emptyset}} p^{6 - e(F' \cap F'')}.
    \end{align*}
    where the last inequality is true for large enough $C_{Low}$ since 
    \[
        \left|\cFQL\right| p^{6} = \Theta\left(e(Q) \cdot n^{4} p^{6}\right) = O\left(\mu_p\left(\cFQL[\ext(\Pi)]\right)\right).
    \]
    We will give an upper bound to the above sum for every $F' \in \cFQL$.    

    Fix $F' \in \cFQL$. For every $F'' \in \cFQL$, there exists $f'' \in Q$ such that $F'' \cup \{f''\} \cong F$. We have
    \begin{align}\label{align:sum-of-low-delta}
        \sum_{\substack{F' \neq F'' \in \cFQL \\ F' \cap F'' \neq \emptyset}} p^{6 - e(F' \cap F'')} = 
        \sum_{i=0}^{3} \sum_{\substack{f'' \in Q \\ |f'' \cap V(F')| = i}} \sum_{\substack{F'' \in \cFQL \\ F'' \cup \{f''\} \cong F \\ F' \cap F'' \neq \emptyset}} p^{6 - e(F' \cap F'')}.
    \end{align}
    We analyse the above sum for each case $i \in \{0, 1, 2, 3\}$, where $i$ denotes the number of vertices in the intersection $f'' \cap V(F')$. For the case $i=0$, we observe that no such copies $F''$ exist. Indeed, otherwise $F'$ would contain two vertex-disjoint edges that contradict the fact that every pair of edges in the Fano plane intersects.

    In the case $i = 1$, there are $O(1)$ ways to choose the intersection $\tilde{F} \coloneqq F' \cap F''$. Following the second sum in \eqref{align:sum-of-low-delta}, there are $O(\Delta_1(Q))$ possibilities for the edge $f'' \in Q$. To complete the copy of $F''$, we must select $5 - v(\tilde{F})$ additional vertices and $6 - e(\tilde{F})$ additional edges. Since $F$ is strictly $3$-balanced and $p \ge n^{-1/m_3(F)}$, then, by Claim \ref{claim:strictly-balanced}, we have $n^{v(\tilde{F})} p^{e(\tilde{F})} \ge n^3 p$, implying that $n^{5 - v(\tilde{F})} p^{6 - e(\tilde{F})} \le n^{2} p^{5}$. 
    Hence, the contribution of the $i=1$ case to the sum in \eqref{align:sum-of-low-delta} is
    \[
        O\left(\Delta_1(Q) \cdot n^{2} p^{5}\right) = O\left(\tilde{\epsilon} \cdot \frac{n^{4} p^{6}}{\log n}\right),
    \]
where we used that $\Delta_1(Q) \le \tilde{\epsilon} \frac{n^2 p}{\log n}$ which holds by the definition of $\cQ_1$ and $\cQ_2$.

    In the case $i=2$, there are $O(1)$ possibilities for the intersection $\tilde{F} \coloneqq F' \cap F''$. From the second sum in \eqref{align:sum-of-low-delta}, we have $O(\Delta_2(Q))$ choices for the edge $f'' \in Q$. To complete the copy $F''$, we must select $6 - v(\tilde{F})$ additional vertices and $6 - e(\tilde{F})$ additional edges. Note that we must have $v(\tilde{F}) > 3$, as otherwise the two vertices in $f'' \cap V(\tilde{F})$ would be contained in two distinct edges in $F''$, a contradiction to the fact that $F$ is linear (i.e., every pair of vertices is contained in at most one edge). Since $F$ is strictly $3$-balanced and $p \ge n^{-1/m_3(F)}$, then, by Claim \ref{claim:strictly-balanced}, we have $n^{v(\tilde{F})} p^{e(\tilde{F})} \ge n^{3 + \lambda} p$ for some $\lambda > 0$. Thus, 
    \[
        n^{6 - v(\tilde{F})} p^{6 - e(\tilde{F})} \le n^{3 - \lambda} p^{5}.
    \]
    Consequently, the contribution of the $i=2$ case to the sum in \eqref{align:sum-of-low-delta} is
    $O\left(\Delta_2(Q) \cdot n^{3 - \lambda} p^{5}\right) = O\left(n^{4 - \lambda} p^{6}\right),$
where we used that $\Delta_2(Q) = O(np)$ which holds by the definition of $\cQ_1$ and $\cQ_2$.

    For $i=3$, there are only $O(1)$ possibilities for the choice of $f'' \in Q$ in the second sum in \eqref{align:sum-of-low-delta}, as the vertex set of $f''$ is completely contained within $V(F')$. There are likewise $O(1)$ ways to choose the intersection $\tilde{F} \coloneqq F' \cap F''$. To complete the copy $F''$, we must select $7 - v(\tilde{F})$ additional vertices and $7 - e(\tilde{F})$ additional edges. Since $F''$ contains $f''$ and at least one other edge from $F'$, the intersection must satisfy $v(\tilde{F}) > 3$. Applying Claim \ref{claim:strictly-balanced} once more, the strictly $3$-balanced nature of $F$ ensures that $n^{v(\tilde{F})} p^{e(\tilde{F})} \ge n^{3 + \lambda} p$ for some $\lambda > 0$. Thus, 
    \[
        n^{7 - v(\tilde{F})} p^{7 - e(\tilde{F})} \le n^{4 - \lambda} p^{6}.
    \]
    Consequently, the contribution of the $i=3$ case to the sum in \eqref{align:sum-of-low-delta} is $O\left(n^{4 - \lambda} p^{6}\right)$.

    Combining the estimates for each case $i \in \{1,2,3\}$ yields the second item of the lemma. \qedhere

\end{proof}

\subsection{High Degree}\label{subsection:high degree}
For every $Q \in \cQ_3$, let $\cFQH$ be the set of all subgraphs $F' \setminus Q \subseteq K_n^{(3)}$ such that $F'$ is a copy of the Fano plane $F$ formed as follows: one vertex of $F'$ is a centre of a star in $Q$, and all three edges incident to this centre are edges of $Q$. Note that all the seven vertices of $F'$ are determined by the chosen centre and the three edges from $Q$ incident to it. By the properties of $\cQ_3$, these three edges do not contain any centres of the stars of $Q$ rather than $v$. Hence, every element of $\cFQH$ is a graph with four edges.

Recall that every $Q \in \cQ_3$ is $\br{2}$ coloured, and a $2$-tuple $\Pi = (A_1, A_2)$ of disjoint vertex sets is compatible with $Q$ if $V^i(Q) \subseteq A_i$ for every $i \in \br{2}$. Also, recall that $\cFQH[\ext(\Pi)]$ is the set of all copies $F' \in \cFQH$ that use only edges of $\ext(\Pi)$ (i.e. $F' \subseteq \ext(\Pi)$). Notably, for every $F' \in \cFQH[\ext(\Pi)]$, there exist three edges $f_1, f_2, f_3 \in Q$ (sharing a common centre of a star in $Q$) such that $F' \cup \{f_1, f_2, f_3\}$ is a $Q$-supported and $\Pi$-crossing copy of $F$.

For every $Q \in \cQ_3$, we have $d_Q = O(e(Q))$ (see \eqref{align:d_Q-def}). Thus, in contrast to the low-degree case, we may be slightly looser and require fewer copies of $F' \in \cFQH\left[\Gnp \cap \ext(\Pi)\right]$. We will then search for copies within a particular subset of $\cFQH[\ext(\Pi)]$ rather than the entire faimly. This, in turn, will allow for better control of the $\Delta$ term in the setting of Janson's inequality.

For every $Q \in \cQ_3$ and every $2$-tuple $\Pi = (A_1, A_2)$ of disjoint vertex sets that is compatible with $Q$, let $\cTFQH[\ext(\Pi)] \subseteq \cFQH[\ext(\Pi)]$ be all the collection of all copies $F' \in \cFQH[\ext(\Pi)]$ such that $|A_1 \cap V(F')| = 3$ and $|A_2 \cap V(F')| = 4$. Referring to Figure \ref{figure:copies of F}, a copy of $\cTFQH[\ext(\Pi)]$ corresponds to the one on the left; specifically, the edge $f$ and the edges consisted of (say) the blue vertex and blue sets in $A_2$ must be edges of $Q$.

As before, define
\begin{align*}
    \mu_p\left(\cTFQH[\ext(\Pi)] \right) &\coloneqq \sum_{F' \in \cTFQH[\ext(\Pi)]} \Pr\left(F' \subseteq \Gnp\right) \quad \text{and}\\
    \Delta_p\left(\cTFQH \right) &\coloneqq \sum_{\substack{F' \neq F'' \in \cTFQH  \\ F' \cap F'' \neq \emptyset}} \Pr\left(F' \cup F'' \subseteq \Gnp\right).
\end{align*}
The central lemma of this subsection establishes sharp bounds on the above quantities.
\needspace{6\baselineskip}
\begin{lemma}\label{lemma:mu-Delta-bound-high-degree}
    For every $p \gg n^{-2/3}$, $Q \in \cQ_3$, and a $2$-tuple $\Pi = (A_1, A_2)$ of disjoint sets of vertices of size at least $n/4$ each that is compatible with $Q$, the following holds. 
    \begin{enumerate}
        \item $\mu_p\left(\cTFQH[\ext(\Pi)]\right) = \Theta(k n^6 p^7) \gg k \cdot n^2 p \ge e(Q)$,
        \item $\frac{\mu_p\left(\cTFQH[\ext(\Pi)]\right)^2}{\Delta_p\left(\cTFQH\right)} = \Omega(n^3 p) \gg k \cdot n^2 p \ge e(Q)$.
    \end{enumerate}
\end{lemma}
\begin{proof}
    Let $k \in \mathbb{N}$ be the number of centres of the stars in $Q$.

    For the first item, we estimate the number of elements in $\cTFQH[\ext(\Pi)]$. As described at the beginning of this subsection, all seven vertices of such an element are chosen by a centre of a star in $Q$ and three additional edges of $Q$ incident to it. Recall that every centre has exactly $\eta n^2 p$ pairs of vertices, in each of the sets $A_1$ and $A_2$, that form an edge together with this centre. Consequently, there are $\Theta\left(k (\eta n^2 p)^3\right)$ ways to choose the vertices of $F'$. Then, we have $\Theta(1)$ choices for the four edges in $F'$. Therefore, using that $p \gg n^{-2/3}$,
    \begin{align}\label{align:mu-bound-high-degree}
        \mu_p\left(\cTFQH[\ext(\Pi)]\right) = \left|\cTFQH[\ext(\Pi)]\right| \cdot p^4 = \Theta\left(k \cdot (\eta n^2 p)^3 \cdot p^4\right) = \Theta\left(k n^6 p^7\right) \gg k \cdot n^2 p \ge e(Q).
    \end{align}

    Moving to the second item of the lemma, we have
    \begin{align*}
        \Delta_p\left(\cTFQH\right) &= \sum_{F' \in \cTFQH} \Pr\left(F' \subseteq \Gnp\right) \sum_{\substack{F' \neq F'' \in \cTFQH \\ F' \cap F'' \neq \emptyset}} \Pr\left(F'' \subseteq \Gnp \mid F' \subseteq \Gnp\right) \\
        &\le \left|\cTFQH\right| \cdot p^4 \cdot \max_{F' \in \cTFQH} \sum_{\substack{F' \neq F'' \in \cTFQH \\ F' \cap F'' \neq \emptyset}} p^{4-e(F' \cap F'')} \\
        &= O\left(\mu_p\left(\cTFQH[\ext(\Pi)]\right)\right) \cdot \max_{F' \in \cTFQH} \sum_{\substack{F' \neq F'' \in \cTFQH \\ F' \cap F'' \neq \emptyset}} p^{4-e(F' \cap F'')},
    \end{align*}
where the last equality is true since we may estemate $\left|\cTFQH\right| \cdot p^4$ in the exact same way as we estimated $\mu_p\left(\cTFQH[\ext(\Pi)]\right)$. Then we get
    \[
        \frac{\mu_p\left(\cTFQH[\ext(\Pi)]\right)^2}{\Delta_p\left(\cTFQH\right)} = \Omega\left( \frac{\mu_p\left(\cTFQH[\ext(\Pi)]\right)}{\max_{F' \in \cTFQH} \sum_{\substack{F' \neq F'' \in \cTFQH \\ F' \cap F'' \neq \emptyset}} p^{4-e(F' \cap F'')}} \right).
    \]
    By \eqref{align:mu-bound-high-degree}, we have $\mu_p\left(\cTFQH[\ext(\Pi)]\right) = \Theta\left(k n^6 p^7\right)$ and thus it suffices to show that 
    \begin{align}\label{align:high-degree-Delta}
        \sum_{\substack{F' \neq F'' \in \cTFQH \\ F' \cap F'' \neq \emptyset}} p^{4-e(F' \cap F'')} = O(kn^3 p^6)
    \end{align}
    for every $F' \in \cTFQH$ and the constant in the big O is independent of the choice of $F'$. Indeed, if this is true, then
    \[
        \frac{\mu_p\left(\cTFQH[\ext(\Pi)]\right)^2}{\Delta_p\left(\cTFQH[\ext(\Pi)]\right)} = \Omega\left(\frac{k n^6 p^7}{k n^3 p^6}\right) = \Omega(n^3 p) \gg k n^2 p \ge e(Q),
    \]
    where the first inequality is true since $k = o(n)$.

    Fix $F' \in \cTFQH$. Summing over all possible intersections $F' \cap F''$, we have
    \begin{align*}
        \sum_{\substack{F' \neq F'' \in \cTFQH \\ F' \cap F'' \neq \emptyset}} p^{4-e(F' \cap F'')} &= \sum_{H \subsetneq F'} \sum_{\substack{F' \neq F'' \in \cTFQH \\ F' \cap F'' \cong H}} p^{4-e(F' \cap F'')} \\
        &= \sum_{H \subsetneq F'} \left|\left\{F' \neq F'' \in \cTFQH \colon F' \cap F'' \cong H\right\}\right| \cdot p^{4 - e(H)}.
    \end{align*}    
    Fix $\emptyset \neq H \subsetneq F'$. We now upper bound the number of copies $F'' \in \cTFQH$ such that $F' \cap F'' \cong H$. To construct such an $F''$, we first select a centre $v$ of a star in $Q$ in at most $k$ ways. By the definition of $\cTFQH$, the edges of $F' \in \cTFQH$ do not contain any centre of the stars of $Q$ (see the beginning of this subsection), implying that $v \notin V(H)$. Since every pair of vertices in the Fano plane is contained in exactly one edge and there are no vertex-disjoint edges,  every additional vertex in $F''$ must form an edge together with $v$ and a vertex from $H$. Therefore, to complete the copy of $F''$, it is left to choose $6-v(H)$ additional vertices, each forming an edge in $Q$ with $v$ and a vertex from $V(H)$. Since $\Delta_2(Q) = O(np)$, there are $O\left((n p)^{6 - |V(H)|}\right)$ choices for these additional vertices. Thus,
    \begin{align*}
        \left|\left\{F' \neq F'' \in \cTFQH \colon F' \cap F'' \cong H\right\}\right| \cdot p^{4 - e(H)} &= O\left(k (np)^{6 - v(H)} \cdot p^{10 - v(H) - e(H)}\right)\\& = O\left(k n^{6-v(H)} p^{7-e(H)}\right) \\&= O\left(k n^3 p^6\right),
    \end{align*}
    where the second equality is true since $v(H) \ge 3$ (as otherwise, $e(H) = 0$) and the last equality is true by Claim \ref{claim:strictly-balanced}.

    As we have $\Theta(1)$ choices for the intersection $F' \cap F''$, the equality \eqref{align:high-degree-Delta} holds, implying the lemma. \qedhere

\end{proof}

\section{Rigidity}\label{section:rigidity}
In this section, we present a crucial property for graphs that is used extensively in the proof of Theorem \ref{thm:main-theorem}. This discussion follows the notations and claims established in \cite{DeMKah15Tur} and \cite{hoshen2023simonovits} with the adjusted changes for hypergraphs. Roughly speaking, the  discussed property in this section is that cuts with small deficit are unique up to a placement of small number of vertices.


Let $\cC$ be a family of cuts and let $G$ be a graph. Recall that $\text{maxcut}_{\cC}(G)$ is the collection of max-cuts in $G$ among the cuts in $\cC$. Moreover, for every cut $\Pi \in \cC$, recall that the deficit of $\Pi$ in $G$ among $\cC$ is
\[
    \deficit_{\cC}(\Pi; G) = b_\cC(G) - e(G \cap \ext(\Pi)).
\]
We denote by $\crit_{\cC}(G) \subseteq G$ the set of edges of $G$ that cross every max-cut of $G$. That is,
\[
    \crit_{\cC}(G) \coloneqq \left\{e \in G \colon e \in \ext(\Pi)\ \text{for all } \Pi \in \text{maxcut}_{\cC}(G)\right\}.
\]

Following DeMarco and Kahn \cite{DeMKah15Tur} and H., Samotij and Zhukovskii \cite{hoshen2024stabilitylargecutsrandom}, we define the following equivalence relation on $V(G)$. For every integer $d \ge 0$, we say that two distinct vertices $v, u \in V(G)$ are \textit{$(\cC, d)$-equivalent in $G$} if and only if they belong to the same part in every $\Pi \in \cC$ satisfying $\deficit_{\cC}(\Pi; G) \le d$. We call the equivalence classes the \textit{$(\cC, d)$-components} of $G$. The following definition, again from \cite{DeMKah15Tur} and \cite{hoshen2024stabilitylargecutsrandom} is key.

\begin{dfn}
    Given a family of cuts $\cC$, a graph $G$, an integer $d \ge 0$, and  $\alpha \in (0, 0.05)$, we say that $G$ is \textit{$(\cC, d, \alpha)$-rigid} if there are at least $(0.5-\alpha) \cdot \frac{n^2}{2}$ many $(\cC, d)$-equivalent pairs of vertices in $G$. 
\end{dfn}

The following claim states that if $\cC$ is a family of $\delta$-balanced cuts and a graph $G$ is $(\cC, d, \alpha)$-rigid, then exist two large vertex sets that consistently respect the partition structure of every cut of deficit at most $d$.
\begin{claim}\label{claim:rigidity-implies-core}
    Let $\alpha \in (0,0.05)$ and let $\delta > 0$ be sufficiently smaller than $\alpha$. Suppose that $\cC$ is a family of $\delta$-balanced cut. Then, for every integer $d \ge 0$, if $G$ is $(\cC, d, \alpha)$-rigid, then there are two $(\cC, d)$-components of $G$ of size at least $(0.5 - 4\alpha)n$ each.
\end{claim}
\begin{proof}
    Let $\lambda \in [0, 1]$ be the number satisfying that exactly $\lambda n$ vertices belong to a $(\cC, d)$-component of $G$ of size less than $(0.5 - 4\alpha)n$. Denote by $P$ the number of $(\cC, d)$-equivalent pairs in $G$. Since any cut in $\cC$ is $\delta$-balanced (that is, every part of such cut is of size at most $(0.5 + \delta)n$), we have
    \[
        (1-2\alpha) \cdot \frac{n^2}{2} \le 2P \le (1-\lambda)n \cdot (0.5 + \delta)n + \lambda n \cdot (0.5 - 4\alpha) n = \left(1 + 2\delta - \lambda(2\delta + 8\alpha)\right) \cdot \frac{n^2}{2},
    \]
    implying that
    $\lambda \le \frac{2\delta + 2\alpha}{2\delta + 8\alpha} \le \frac{3}{8}$,
    where the last inequality is true whenever $\delta$ is sufficiently smaller than $\alpha$.

    However, if there are less than two components of size at least $(0.5 - 4\alpha)n$, then
    \[
        (1-\lambda)n \le (0.5+\delta)n,
    \]
    implying that
    \[
        \lambda \ge 0.5 - \delta > \frac{3}{8},
    \]
    where the last inequality is true whenever $\delta$ is small enough, a contradiction.
\end{proof}

If the graph $G$ is $(\cC, d, \alpha)$-rigid, then we refer to the two equivalence classes provided by Claim \ref{claim:rigidity-implies-core} as the $(\cC, d)$-\textit{core} of $G$, denoted by $\core_{\cC, d}(G)$. Note that, as $\alpha \le 0.05$, this definition indeed does not depend on $\alpha$ and further, the $(\cC, d)$-core is unique. We say that the $(\cC, d)$-core $S_1, S_2$ is contained in a cut if $S_1$ is contained in one part of the cut and $S_2$ is contained in the other. Furthermore, a $\br{2}$-coloured graph $Q$ is said to be \textit{contained in the $(\cC, d)$-core $\{S_1, S_2\}$} if $V^1(Q) \subseteq S_i$ and $V^2(Q) \subseteq S_{3-i}$ for some $i \in \br{2}$.

The following lemma is a remarkable usage of Harris inequality used by DeMarco and Kahn in \cite{DeMKah15Tur}. It correlates the event of admitting a $(\cC, d)$-core $S$ and some monotone event determined by $\ext(S)$. More specifically, it provides a bound on the probability of the intersection of these two events. For the sake of completeness, we include the proof here as well.

\begin{lemma}\label{lemma:correlation-lemma}
    Let $\delta > 0$ and $\alpha > 0$ be sufficiently small. Let $Q$ be a $\br{2}$-coloured graph, let $\cC$ be a family of $\delta$-balanced cuts that are compatible with $Q$ and let $\xi \in [0, 1]$. Suppose that for every $2$-tuple $S = (S_1, S_2)$ of disjoint sets of vertices that is contained in some cut from $\cC$ and satisfies $|S_i| \ge (0.5 - 4\alpha)n$ for every $i \in \br{2}$, we have an event $F(S_1, S_2)$ that is decreasing, determined by $\ext(S)$ and satisfies
    \[
        \Pr\left(F(S_1, S_2) \mid Q \subseteq \Gnp\right) \le \xi.
    \]
    For every integer $d \ge 0$, let $\cR_d$ be the event that $\Gnp$ is $(\cC, d, \alpha)$-rigid, with $(\cC, d)$-core $S_1$ and $S_2$ labeled such that $V^i(Q) \subseteq S_i$ for every $i \in \br{2}$, and $F(S_1, S_2)$ holds. Then, $\Pr\left(\cR_d \mid Q \subseteq \Gnp\right) \le 2\xi$.
\end{lemma}
\begin{proof}
    For an ordered $2$-tuple of pairwise-disjoint vertex sets $S = (S_1, S_2)$ that is compatible with $Q$ and satisfies $\min_i |S_i| \ge (0.5-4\alpha)n$, denote by $E(\{S_1, S_2\})$ the event that $\Gnp$ is $(\cC, d, \alpha)$-rigid, with $(\cC, d, \alpha)$-core $\{S_1, S_2\}$.
    We claim that this event is increasing in $\ext(S)$. Indeed, if $G \in E(\{S_1, S_2\})$, then adding to $G$ an edge $e \in \ext(S)$ does not change the set of cuts of $G$ of deficit at most $d$ among $\cC$; in particular, the graph $G \cup e$ is $(\cC, d, \alpha)$-rigid and has the same $(\cC, d, \alpha)$-core as $G$.
    By Harris's inequality,
    \begin{multline*}
        \Pr\left(E(\{S_1, S_2\}) \cap F(S_1, S_2) \mid Q \subseteq \Gnp\right) 
        \le \Pr\left(E(\{S_1, S_2\}) \mid Q \subseteq \Gnp\right) \cdot \Pr\left(F(S_1, S_2) \mid Q \subseteq \Gnp\right).
    \end{multline*}
    Consequently,
    \begin{align*}
        \Pr(\cR_d \mid Q \subseteq \Gnp)
        &= \sum_{(S_1, S_2)} \Pr\big(E(\{S_1, S_2\}) \cap F(S_1, S_2) \mid Q \subseteq \Gnp\big) \\
        &\le\sum_{(S_1, S_2)} \Pr(E(\{S_1, S_2\}) \mid Q \subseteq \Gnp) \cdot \xi
        \le 2 \xi,
    \end{align*}
    where the sums are over all ordered tuples $(S_1, S_2)$ as above and the last inequality follows since the core is unique and its elements can be ordered in no more than two different ways.
\end{proof}
Notably, if $G_{n, p}$ typically admits a $(\cC, d)$-core $S$, then the lemma above allows one to establish typical monotone events determined by $\ext(S)$. 
The purpose of the following subsection is to prove that typically $G_{n, p}$ indeed admits a $(\cC, d)$-core. As we will see in Section \ref{section:0-statement}, we will be able to obtain much more rather than the direct application of this result combined with Lemma \ref{lemma:correlation-lemma}.

\subsection{Rigidity for random hypergraphs}\label{section:rigidity for hypergraphs}
As described in the introduction, one of the hardest challenges in proving Theorem \ref{thm:main-theorem} is the massive union bound over all possible cuts. In this subsection, we show that $\Gnp$ typically admits a $(\cC, d)$-core. Equiqqed with this key property, we will then be able to prove typical properties that hold simultaniously for every cut of small deficit. Throughout this subsection, we take $\cC$ to be the family of all cuts and accordinly omit the index $\cC$ from our notation.

The following theorem and its immediate corrolary are the cemtral results in this subsection. Their proofs follow \cite{hoshen2023simonovits} but they are provided here as well for the sake of completeness. For convenience, set $N \coloneqq \binom{n}{3}$.
\begin{thm}\label{thm:rigidity}
    There exists $C > 0$ such that, for every $\delta, \epsilon \in (0,1)$ and every nonnegative integers $n, m, d$ satisfying $1 \le m \le (1-\delta)N$, we have
    \[
        \Pr\left(G_{n, m}^{(3)} \text{ is not } (d, \epsilon)\text{-rigid}\right) \le \frac{C}{\epsilon}\left(\frac{d}{\delta} \cdot \sqrt{\frac{n}{m}} + \left(\frac{n}{m}\right)^{1/4}\right).
    \]
\end{thm}
Denote by $\Core_d(\alpha)$ the set of graphs that admit a $d$-core whose each component has size at least $0.5n - \alpha n$.
\begin{cor}\label{cor:core}
    There exists $C > 0$ such that, for every $\delta \in (0,1)$ and $0 < \alpha < \frac{1}{3}$, and every nonnegative integers $n, m, d$ satisfying $1 \le m \le (1-\delta)N$, we have
    \[
        \Pr\left(G_{n, m} \notin \Core_{d}(\alpha)\right) \le \frac{C}{\alpha}\left(\frac{d}{\delta} \cdot \sqrt{\frac{n}{m}} + \left(\frac{n}{m}\right)^{1/4}\right).
    \]
\end{cor}

For every graph $G \subseteq K_n^{(3)}$ and every integer $d \ge 0$, let $\eq_d^{(3)}(G)$ denote the set of edges $e \in K_n^{(3)}$ such that, for every cut $\Pi=(A_1, A_2)$ with $\deficit(\Pi; G) \le d$, the edge $e$ is contained entirely within one part of the cut (i.e, either $e \subseteq A_1$ or $e \subseteq A_2$). Crucially, $e \in \eq_d^{(3)}(G)$ if and only if $e \in \int(\Pi)$ for every cut $\Pi=(A_1, A_2)$ with $\deficit(\Pi; G) \le d$.
Note that the structure of $\eq_d^{(3)}(G)$ (as a graph) is a union of cliques. Moreover, we define $\eq_{-1}^{(3)}(G) = \int(\Pi_G)$ where $\Pi_G$ is some (canonically chosen) max-cut of $G$.

The following claim describes how $\eq_d(G)$ evolves as edges are added to or removed from $G$.
\begin{claim}\label{claim:dyncamics of core}
    The following holds for every integer $d \ge 0$ and $e \in K_n^{(3)}$. If $e \notin G \cup \eq_{d+1}(G)$, then $e \notin \eq_d(G \cup e)$.
\end{claim}
The following properties will be useful throughout the proof of the above claim and across Section \ref{section:0-statement} as well.
\begin{fact}\label{fact:equivalence between eq}
    The following statements are equivalent for all $e \in K_n^{(3)} \setminus G$:
    \begin{itemize}
        \item $e \notin \eq_0(G)$;
        \item $e \in \ext(\Pi)$ for some max-cut $\Pi$ of $G$;
        \item $e \in \ext(\Pi)$ for all max-cuts $\Pi$ of $G \cup \{e\}$ (that is, $e \in \crit(G \cup \{e\})$);
        \item $b(G \cup \{e\}) = b(G) + 1$.
    \end{itemize}
\end{fact}
\begin{proof}[Proof of Claim \ref{claim:dyncamics of core}]
    Suppose that $e \notin G \cup \eq_{d+1}(G)$. Let $\Pi$ be a cut of $G$ of smallest deficit $d_e \le d+1$ such that $e \in \ext(\Pi)$. Note that if $d_e=0$, then $\Pi$ is a max-cut of $G$ and thus, by Fact \ref{fact:equivalence between eq}, $e \in \ext(\Pi')$ for every max-cut $\Pi'$ of $G \cup \{e\}$. Thus, in particular, $e \notin \eq_{d}(G \cup \{e\})$. If $d_e > 0$, then $e$ does not cross any max-cut of $G$ so we have that $\deficit(\Pi; G\cup\{e\})=d_e-1 \le d$, implying that $e \notin \eq_{d}(G \cup \{e\})$.
\end{proof}

Our applications of Corollary \ref{cor:core} in Section \ref{section:0-statement} will crucially use the fact that cores are `stable' under small edge perturbations.  Our next lemma, and its corollaries, formalise this notion of stability.

\begin{lemma}
  \label{lemma:eq-d-nested}
  For every integer $d \ge 0$, graph $G \subseteq K_n^{(3)}$, and edge $e \in K_n^{3}$, we have
  \[
    \eq_{d+1} (G) \subseteq \eq_d(G \triangle e),
  \]
  where $G \triangle e$ describes with the set of edges equal to the symmetric difference of the set of edges of $G$ together with $\{e\}$.
\end{lemma}
\begin{proof}
  We prove the equivalent statement $K_n^{(3)} \setminus \eq_d(G \triangle e) \subseteq K_n^{(3)} \setminus \eq_{d+1}(G)$.
  Observe first that, for every fixed cut $\Pi$, the function\footnote{We write $\cP(K_n^{(3)})$ to denote the family of all subgraphs of $K_n^{(3)}$, i.e., the powerset of $K_n^{(3)}$, which we identify here with its set of edges.}
  \[
    \cP(K_n^{(3)}) \ni H \mapsto \deficit_H(\Pi) = b_r(H) - \big|H \cap \ext(\Pi)\big|
  \]
  is the difference of two nondecreasing, $1$-Lipshitz functions, and thus it is also $1$-Lipshitz.
  In particular, for every graph $G$ and edge $e \in K_n$, we have $\deficit_G(\Pi) \le \deficit_{G \triangle e}(\Pi) + 1$.
  Now, suppose that $f \in K_n^{(3)} \setminus \eq_d(G \triangle e)$, that is, $f \in \ext(\Pi)$ for some $\Pi$ with $\deficit_{G \triangle e}(\Pi) \le d$.  Since $\deficit_G(\Pi) \le \deficit_{G \triangle e}(\Pi) + 1 \le d+1$, we have $f \notin \eq_{d+1}(G)$.
\end{proof}

\begin{cor}
  \label{cor:eq-d-nested}
  For every integer $d \ge 0$ and graphs $G, T \subseteq K_n^{(3)}$, we have
  \[
    \eq_{d+e_T} (G) \subseteq \eq_d(G \triangle T).
  \]
\end{cor}

We say that a $d$-core of a graph $G$ is contained in the $d'$-core of a graph $G'$ if every element of a $d$-core of $G$ is contained in some element of the $d'$-core of $G'$. We denote this by $\core_{d}(G) \preceq \core_{d'}(G')$.
\begin{cor}
  \label{cor:core-d-nested}
  The following holds for every integer $d \ge 0$ and all graphs $G, T \subseteq K_n^{(3)}$.
  If $G$ has a $(d+e_T)$-core, then $G \triangle T$ has a $d$-core and $\core_{d+e_T}(G) \preceq \core_d(G \triangle T)$.
\end{cor}

The following lemma is the heart of the proof of Theorem \ref{thm:rigidity}. As noted at the beginning, its proof is borrowed from \cite{hoshen2024stabilitylargecutsrandom} with the minor adjustments to hypergraphs.
\begin{lemma}\label{lemma:transition-lemma}
    Let $G_1 \sim G_{n, m}^{(3)}$ and $G_2 \sim G_{n, m+1}^{(3)}$ for some $m \in \left\{0, \dots, N-1\right\}$. For every integer $d \ge -1$,
    \[
        \Ex\left| \eq_{d+1}^{(3)}(G_1) \setminus G_1\right| \ge \frac{N - m}{m + 1} \Ex\left|\eq_{d}^{(3)}(G_2) \cap G_2\right|.
    \]
\end{lemma}
\begin{proof}
    We may clearly couple $G_1$ and $G_2$ so that $G_2 = G_1 \cup \{e\}$ where $e$ is uniformly random edge of each $K_n^{(3)} \setminus G_1$ and $G_2$. By Claim \ref{claim:dyncamics of core}, and by Fact \ref{fact:equivalence between eq} for $d=-1$, we have that for every $d \ge 0$
    \begin{align*}
        \Ex\left| \left(K_n^{(3)} \setminus \eq_{d+1}^{(3)}(G_1)\right) \setminus G_1 \right| &= (N-m) \Pr\left(e \in K_n^{(3)} \setminus \eq_{d+1}^{(3)}(G_1)\right) \\
        &\le (N-m) \Pr\left(e \notin \eq_d^{(3)}(G_2)\right) = \frac{N-m}{m+1} \Ex\left| G_2 \setminus \eq_d^{(3)}(G_2) \right|.
    \end{align*}
    Further,
    \[
        \Ex\left| \left(K_n^{(3)} \setminus \eq_{d+1}^{(3)}(G_1)\right) \setminus G_1 \right| = N-m - \Ex\left| \eq_{d+1}^{(3)}(G_1) \setminus G_1 \right|
    \]
    and
    \[
        \Ex\left| G_2 \setminus \eq_d^{(3)}(G_2) \right| = m+1 - \Ex\left| G_2 \cap \eq_d^{(3)}(G_2) \right|,
    \]
    finishing the proof.
\end{proof}
We now use the above in order to derive the following cleaner statement.
\begin{cor}\label{cor:eq_d+1-eq_d}
    Let $G_1 \sim G_{n, m}^{(3)}$ and $G_2 \sim G_{n, m}^{(3)}$ for some $m \in \left\{0, \dots, N-1\right\}$. For every integer $d \ge -1$,
    \[
        \Ex\left|\eq_{d+1}^{(3)}(G_1)\right| \ge \Ex\left|\eq_{d+1}^{(3)}(G_2)\right| - \frac{n^{6.5}}{\sqrt{m+1}(N-m)}.
    \]    
\end{cor}
\begin{proof}
    By Corollary \ref{cor:equiv_classes_edges_count}, letting $p_1 = \frac{m}{N}$ and $p_2 = \frac{m+1}{N}$, we have
    \[
        \Ex\left| \eq_{d+1}^{(3)}(G_1) \setminus G_1\right| = \Ex\left| \eq_{d+1}^{(3)}(G_1)\right| - \Ex\left| \eq_{d+1}^{(3)}(G_1) \cap G_1\right| \le \Ex\left| \eq_{d+1}^{(3)}(G_1)\right|(1-p_1) + n^{1.5} \sqrt{n p_1} + n^3 e^{-n}
    \]
    and
    \[
        \Ex\left|\eq_{d+1}^{(3)}(G_2) \cap G_2\right| \ge \Ex\left|\eq_{d+1}^{(3)}(G_2)\right| p_2 - n^{1.5} \sqrt{n p_2} - n^3 e^{-n}.
    \]
    Hence, using Lemma \ref{lemma:transition-lemma},
    \begin{align*}
        \Ex\left|\eq_{d+1}^{(3)}(G_1)\right| &\ge \frac{1}{1-p_1} \left(\Ex\left|\eq_{d+1}^{(3)}(G_1) \setminus G_1\right| - 2 n^{1.5} \sqrt{n p_1}\right) \\
        &\ge \frac{1}{1-p_1} \cdot \frac{N-m}{m+1} \cdot \Ex\left|\eq_{d}^{(3)}(G_2) \cap G_2\right| - \frac{2 n^{1.5} \sqrt{n p_1}}{1-p_1} \\
        &\ge \frac{p_2}{1-p_1} \cdot \frac{N-m}{m+1} \left(\Ex\left|\eq_{d}^{(3)}(G_2)\right| - \frac{2n^{1.5} \sqrt{n p_2}}{p_2}\right) - \frac{2 n^{1.5} \sqrt{n p_1}}{1-p_1} \\
        &= \Ex\left|\eq_{d}^{(3)}(G_2)\right| -2n^2 \left(\frac{1}{\sqrt{p_2}} + \frac{\sqrt{p_1}}{1-p_1}\right).
    \end{align*}
    The claimed inequality follows as
  \[
    \frac{1}{\sqrt{p_2}} + \frac{\sqrt{p_1}}{1-p_1} \le \frac{1}{\sqrt{p_2}} + \frac{\sqrt{p_2}}{1-p_1} = \frac{1-p_1+p_2}{\sqrt{p_2} \cdot (1-p_1)} = \frac{\sqrt{N} \cdot (N+1)}{\sqrt{m+1} \cdot (N-m)}.\qedhere
  \]
\end{proof}
We may now conclude the proof of Theorem \ref{thm:rigidity}.
\begin{proof}[Proof of Theorem \ref{thm:rigidity}]
By inductive usage of Corollary \ref{cor:eq_d+1-eq_d}, we have
\begin{align*}
    \Ex\left|\eq_{d}^{(3)}(G_{n, m})\right| &\ge \Ex\left|\eq_{0}^{(3)}(G_{n, m+d+1})\right| - \sum_{i=0}^{d} \frac{n^{6.5}}{\sqrt{m+i+1}\left(N-m-i\right)} \\
    &\ge 2 \binom{n/2}{3} - \frac{(d+1) n^{6.5}}{\sqrt{m}\left(\delta N - d\right)} \ge 2 \binom{n/2}{3} - \frac{7(d+1)n^{3.5}}{\delta \sqrt{m}}.
\end{align*}

By Lemma \ref{lemma:max_cut_balanced}, with probability at least $1-e^{-n}$, we have, assuming that $d \le \sqrt{nm}$,
\[
    \left|\eq_{d}^{(3)}(G_{n, m})\right| \le 2\binom{\frac{n}{2} + 4 n (n/m)^{1/4}}{3} \le 2\binom{n/2}{3} + \beta n^3 (n/m)^{1/4},
\]
for some constant $\beta > 0$. Set 
\[
    P \coloneqq \Pr\left(\left|\eq_d^{(3)}(G_{n, m})\right| \le (1-\epsilon) \cdot 2 \binom{n/2}{3}\right).
\]
We have
\begin{align*}
    2 \binom{n/2}{3} - \frac{7(d+1)n^{3.5}}{\delta \sqrt{m}} &\le \Ex\left|\eq_{d}^{(3)}(G_{n, m})\right| \\&\le P(1-\epsilon) \cdot 2 \binom{n/2}{3} + (1-P)\left(2\binom{n/2}{3} + \beta n^3 (n/m)^{1/4}\right) + e^{-n} N,
\end{align*}
implying that
\begin{align*}
    P &\le \frac{1}{\epsilon \cdot 2\binom{n/2}{3} + \beta n^3 (n/m)^{1/4}} \left(\frac{7(d+1)n^{3.5}}{\delta \sqrt{m}} + \beta n^3 (n/m)^{1/4} + e^{-n} N\right) \\
    &\le \frac{C}{\epsilon}\left(\frac{d}{\delta} \cdot \sqrt{\frac{n}{m}} + \left(\frac{n}{m}\right)^{1/4}\right).
\end{align*}

\end{proof}

\section{Main Points for Lemma \ref{lemma:main-lemma}}
In this section, we prove Lemma \ref{lemma:main-lemma}. As mentioned before, the estimations obtained in Section \ref{section:fixed-cuts} are sufficient to prove Lemma \ref{lemma:main-lemma} for a fixed balanced cut $\Pi$ and all graphs $Q$ for which $\Pi \in \cC_Q$. Here, we show that actually the union bound over all possible cuts can be avoided, and instread the primary part lies in the selection of the graphs $Q$.

We divide this section into three pieces. First, in Section~\ref{subsection:setting-the-ground}, we introduce essential notation and prepare the ground for the proof. We state a central result, Lemma \ref{lemma:main-switching-lemma}, which is proved in Section~\ref{section:proof-of-switching-lemma}. This lemma employs a delicate switching argument to show that the union bound over all cuts in $\cC_Q$ is not costly for us. A version of this argument was first established by DeMarco and Kahn in \cite{DeMKah15Tur} and was later improved by the author and Samotij in \cite{hoshen2023simonovits}. 

As we shall see, this argument involves several parameters and, in Section~\ref{section:choosing-parameters-for-swithcing-lemma}, we select these parameters based on the values of $p$ and $Q$. Finally, we bring all the pieces together in Section~\ref{section:proof-of-main-lemma} to complete the proof of Lemma \ref{lemma:main-lemma}.

\subsection{Setting the ground}\label{subsection:setting-the-ground}
%

For every $Q \in \cQ$, a family $\cF_Q$ of subsets of $K_n^{(3)} \setminus Q$, and two integers $d$ and $x$, denote by $\cB_{Q, \cF_Q}(d, x)$ the set of graphs $G \supseteq Q$ that admit a cut $\Pi \in \cC_Q$ with $\deficit_{\cC_Q}(\Pi; G) \le d$ and 
\[
    \nu\left(\cF_Q\left[\Gnp \cap \ext\left(\Pi\right)\right]\right) \le x.
\]
Observe that the condition ~$\nu\left(\cF_Q\left[\Gnp \cap \ext\left(\Pi\right)\right]\right) > x$ implies the existence of more than $x$ edge-disjoint, $\Pi$-crossing elements $F' \in \cF_Q$ satisfying $F' \subseteq \Gnp$. Thus, to prove Lemma \ref{lemma:main-lemma}, it suffices to find a family $\cF_Q$ for each $Q \in \cQ$, consisting of subgraphs of the form $F' \setminus Q$ where $F'$ is a $Q$-supported copy of $F$, such that
\begin{align}\label{align:main-lemma-desired-probability}
    \Pr\left(\Gnp \in \bigcup_{Q \in \cQ} \cB_{Q, \cF_Q}(d_Q, d_Q)\right) = o(1).
\end{align}

Denote by $\cT$ the set of graphs $G \subseteq K_n^{(3)}$ that satisfy the following.
\begin{enumerate}
    \item $\Delta_1(G) \le 2n^2 p$ and $\Delta_2(G) \le 2np$.
    \item $e(G \cap \int(\Pi)) = e(\int(\Pi)) \cdot p \pm o(n^3 p)$ and $e(G \cap \ext(\Pi)) = e(\ext(\Pi)) \cdot p \pm o(n^3 p)$ for every $2$-tuple of vertex-disjoint sets of vertices $\Pi = (A_1, A_2)$.
\end{enumerate}
By standard concentration over the degree of vertices of $\Gnp$ and by Corollary \ref{cor:equiv_classes_edges_count} (noting that $n^2 \sqrt{p} \ll n^3 p$), we have the following claim.
\begin{claim}\label{claim:typical properties}
    $\Pr\left(\Gnp \in \cT\right) = 1 - o(1)$.
\end{claim}

The following lemma plays the main role in the proof of Lemma \ref{lemma:main-lemma}. It states that it is very unlikely that $\Gnp \in \bigcup_{Q \in \cQ} \cB_{Q, \cF_Q}(d_Q, d_Q) \cap \cT$ and moreover $\Gnp$ satisfies one more additional property that, as we will see later, is a typical property of $\Gnp$. Thus, using Claim \ref{claim:typical properties} as well, this implies that the event $\Gnp \in \bigcup_{Q \in \cQ} \cB_{Q, \cF_Q}(d_Q, d_Q)$ is rare. 

\begin{lemma}\label{lemma:main-switching-lemma}
    Let $Q \in \cQ$ and let $d, m, D$ be integers satisfying
    \[
        d, e(Q) \ll n^3 p \quad \text{and} \quad d \le |V(Q)| \cdot n^2 p.
    \]
    Let $\cF_Q$ by a family of subgraphs of $K_n^{(3)} \setminus Q$ and denote     
    \begin{align*}
        \xi \coloneqq \Pr\Bigg( &\Gnp \text{ is $(\cC_Q, 0, \alpha)$-rigid} \\
        &\text{and } \nu\left(\cF_Q\left[\Gnp \cap \ext\left(\core_{(\cC_Q, 0)}\left(\Gnp\right)\right)\right]\right) \le 5d + m \;\Bigg|\; Q \subseteq \Gnp\Bigg).
    \end{align*}  
    Denote by $\cD$ the set of graphs $G$ that satisfy $\min\left\{\binom{n}{3}, \left|\partial\cF_Q\left[G\right]\right|\right\} \le Dm/p$. Then, for any positive $K$, 
    \begin{align*}
        \Pr\left(\Gnp \in \cT \cap \cD \cap \cB_{Q, \cF_Q}(d, d) \right) \le p^{e(Q)} D^{d} \left(Z \cdot \xi + e^{-K \cdot d \cdot \log\left(\frac{n^3 p}{d}\right)}\right),
    \end{align*}
    where $Z$ depends only on $\alpha$.
\end{lemma}

\subsection{Choosing the parameters}\label{section:choosing-parameters-for-swithcing-lemma}
In this section, we choose the parameters for Lemma \ref{lemma:main-switching-lemma} required to establish Lemma \ref{lemma:main-lemma}. First, recall that $d_Q$ has already been defined for every $Q \in \cQ$ (see \eqref{align:d_Q-def}). We additionally define integers $m_Q, D_Q$, along with a family $\cF_Q$ of subgraphs of $K_n^{(3)} \setminus Q$ consisting of $Q$-supported copies of $F$. Eventually, Lemma \ref{lemma:main-switching-lemma} will be applied to each $Q$ using the parameters $d = d_Q, m = m_Q, D = D_Q, \cF_Q = \cF_Q$.

We begin by recalling several definitions that will be used throughtout this section. First, $\cF$ denotes the family of all copies of the Fano plane $F$ in $K_n^{(3)}$. For every $Q \in \cQ_1 \cup \cQ_2$, the family $\cFQL$ encapsulates all the copies $F' \setminus Q$ such that $F' \in \cF$ shares exactly one edge with $Q$. For any cut $\Pi = (A_1, A_2) \in \cC_Q$, we denoted by $\cFQL[\ext(\Pi)]$ the set of copies of $F' \in \cFQL$ such that $F' \subseteq \ext(\Pi)$. In other words, $\cFQL[\ext(\Pi)]$ is the collection of all $Q$-supported, $\Pi$-crossing copies of $F$ that contain exactly one edge of $Q$. 
Furthermore, for every $Q \in \cQ_3$, the set $\cFQH$ is the family of all subgraphs $F' \setminus Q$ such that $F' \in \cF$ contains the centre of a star in $Q$ along with three edges of $Q$ incident to that centre (recakk that by the definition of $\cQ_3$, these edges intersect only at the centre). Similarly, for every $\Pi = (A_1, A_2) \in \cC_Q$, we denote by $\cFQH[\ext(\Pi)]$ the set of elements $F' \in \cFQH$ such that $F' \subseteq \ext(\Pi)$.


The selection of parameters for the cases where $Q \in \cQ_2$ and $p$ is large is slightly more involved and thus, we first define the parameters for all the remaining cases. Let $C$ be a sufficiently large constant and $\kappa > 0$ be a sufficiently small constant. Recall the constant $\beta =\beta(n) \ll 1$ that was defined in Section \ref{section:constants}.
\begin{itemize}
    \item If $p \le C n^{-2/3} \left(\log n\right)^{1/6}$ and $Q \in \cQ_1 \cup \cQ_2$, then
    \[
        \cF_Q = \cFQL, \quad d_Q = \min\left\{\beta n^3 p,\sqrt{\eta} \cdot e(Q) \log n\right\},\quad m_Q = \kappa \cdot e(Q) n^4 p^6 \quad \text{and} \quad D_Q = \frac{2}{\kappa}.
    \]
    \item If $p > C n^{-2/3} \left(\log n\right)^{1/6}$ and $Q \in \cQ_1$, then 
    \[
        \cF_Q = \cFQL,\quad d_Q = m_Q = 30 e(Q) \quad \text{and} \quad D_Q = \frac{n^3 p}{d_Q}.
    \]    
    \item If $Q \in \cQ_3$, letting $k(Q)$ be the number of centres of the stars in $Q$, then
    \[
        \cF_Q = \cFQH,\quad d_Q = m_Q = \frac{33 \cdot e(Q)}{\eta} \quad \text{and} \quad D_Q = \frac{14 \cdot \min\left\{k(Q) \cdot n^6 p^7, n^3 p\right\}}{m_Q}.
    \]    
\end{itemize}

We are thus left with the case where $p > C n^{-2/3} \left(\log n\right)^{1/6}$ and $Q \in \cQ_2$. A key ingredient in the proof of Theorem \ref{thm:1-statement} is the estimations presented in Section \ref{section:fixed-cuts}, which provide strong bounds on the probability we cannot find many edge-disjoint, $Q$-supported and $\Pi$-crossing copies of $F$ in $\Gnp$. As the value of $p$ increases, however, the terms $\Delta_p(\cFQL)$ and $\Delta_p(\cFQH)$ from Lemmas \ref{lemma:mu-Delta-bounds-low-degree} and \ref{lemma:mu-Delta-bound-high-degree}, respectively, become problematic because the ratio between these terms and their respective means, and $\mu_p(\cFQL[\ext(\Pi)])$ and $\mu_p(\cFQH[\ext(\Pi)])$, increases. 

This is not a significant issue when $Q \in \cQ_1$, as we possess a sufficiently strong upper bound on $\Delta_1(Q)$, leading to favorable bounds on the aforementioned terms. Also, in the high-degree case, we are permitted to consider a slightly modified family of $Q$-supported copies of $F$, which enables us to circumvent this problem. However, in the case where $p > C n^{-2/3} \left(\log n\right)^{1/6}$ and $Q \in \cQ_2$, this issue kicks in and requires a delicate treatment. To address this, rather than examining the entire family of all possible $Q$-supported copies containing exactly one edge of $Q$ (the family $\cFQL$), we instead consider a carefully chosen random subfamily.

For every $q \in [0, 1]$, let $\cF_q \subseteq \cF$ be a random subset of $\cF$ obtained by including each element of $\cF$ independently with probability $q$. For every $\cF' \subseteq \cF$, we denote by $\cFQL(\cF')$ the set of all subgraphs of $F' \setminus Q$ where $F' \in \cF'$ contains shares exactly one edge with $Q$.

Recall from Section \ref{section:typical-props} that, for a family of graphs $\cG$ and an edge $e \in K_n^{(3)}$, we defined $\partial \cG$ and $\partial_e \cG$ to be the family of copies $A \setminus f$ for every $A \in \cG$ and $f \in A$, and the family of copies of $A \setminus e$ for every $e \in A \in \cG$, respectively. Furthermore, for a graph $G$, the family $\partial \partial_e \cG[G]$ denotes the set of copies of $A \setminus \{f, e\}$ for every $A \in \cG$ such that $A \setminus \{f, e\} \subseteq G$. Observe that, in the setting of Lemma \ref{lemma:main-switching-lemma}, it will be useful to upper bound $\partial \partial_e \cFQL\left[\Gnp\right]$ for every $e \in Q$.

\begin{lemma}\label{lemma:sparsification}
    There exists a large constant $\hat{C} > 0$ such that if $q \in [0, 1]$ satisfies
    \[
        q \ge \max_{Q \in \cQ_2} \hat{C} \frac{n^{-3}}{e(Q)}\quad \text{and} \quad n^4 p^5 q \ge \hat{C} \cdot \log n,
    \]
    then there exist $\cF_1, \cF_2, \dots, \cF_{n^3} \subseteq \cF$ satisfying the following.
    \begin{enumerate}
        \item Whp (under the measure of $\Gnp$), we have $\left| \partial \partial_e\cF_i\left[\Gnp\right] \right| \le \hat{C} n^4 p^5 q$ for every $e \in K_n^{(3)}$ and $i \in \br{n^3}$.   
        \item For every $Q \in \cQ_2$, there exists an $i \in \br{n^3}$ such that the following holds.
        \begin{enumerate}
            \item For every $2$-tuple $S = (S_1, S_2)$ of disjoint sets of vertices of size at least $n/4$ each with $V(Q) \subseteq S_1$, we have 
            \[
                \mu_p\left(\cFQL(\cF_i)[\ext(S)]\right) \ge \frac{q}{2} \cdot \mu_p\left(\cFQL[\ext(S)]\right).
            \]
            \item $\Delta_p\left(\cFQL(\cF_i)\right) \le 2 q^2 \cdot \Delta_p\left(\cFQL\right)$.
        \end{enumerate}
    \end{enumerate}
\end{lemma}
\begin{proof}
    For every $i \in \br{n^3}$, let $\cF_i$ be an independent sample of $\cF_q$.
    We begin by proving the first item. Fix $e \in K_n^{(3)}$ and $i \in \br{n^3}$. We have
    \begin{align*}
        \Pr\left(\left| \partial \partial_e\cF_i\left[\Gnp\right] \right| \ge \hat{C} n^4 p^5 q\right) &\le \Pr\left(\left| \partial \partial_e\cF\left[\Gnp\right] \right| \ge 2 n^4 p^5 \right) \\&+ \Pr\left(\left| \partial \partial_e \cF_i\left[\Gnp\right] \right| \ge \hat{C} n^4 p^5 q \mid \left|\partial \partial_e \cF\left[\Gnp\right] \right| \le 2 n^4 p^5 \right).
    \end{align*}
    By Claim \ref{claim:Fano minus edge}, the first term is at most $n^{-8}$. To bound the second term, we define the set
    \[
        Z_e \coloneqq \left\{F' \in \cF \colon e \in F' \text{ and } F' \setminus \{e, f\} \subseteq \Gnp \text{ for some } f \in F' \right\}.
    \]

    We first claim that $\frac{1}{7} \left|\partial \partial_e \cF_i\left[\Gnp\right] \right| \le |Z_e \cap \cF_i| \le 7 \left|\partial \partial_e \cF_i\left[\Gnp\right] \right|$ for every $i \in \br{n^3}$. 
    For the second inequality, note that, for every $F' \in Z_e \cap \cF_i$, there are at most $e(F)$ choices of $f \in F'$ such that $F' \setminus \{e, f\} \subseteq \Gnp$. Each such choice yields an element $F' \setminus \{e, f\} \in \partial \partial_e \cF_i[\Gnp]$.
    For the second inequality, observe that for every $F' \in \partial \partial_e \cF_i [\Gnp]$, there are at most $e(F)$ choices for an edge $f$ such that $F' \cup \{e, f\} \in \cF$. By the definition of $Z_e$, any such completion $F' \cup \{e,f\}$ must belong to $Z_e \cap \cF_i$. Since each $F' \in Z_e \cap \cF_i$ can be formed in most $e(F)$ ways in this manner, the second inequality follows.

    Given the size of $Z_e$, we have $|Z_e \cap \cF_i| \sim \text{Bin}(|Z_e|, q)$. Hence,
    \begin{align*}
        &\Pr\left(\left| \partial \partial_e \cF_i\left[\Gnp\right] \right| \ge \hat{C} n^4 p^5 q \mid \left|\partial \partial_e \cF\left[\Gnp\right] \right| \le 2 n^4 p^5 \right) \\&\le \Pr\left(\left| Z_e \cap \cF_i \right| \ge \frac{\hat{C} n^4 p^5 q}{7}  \mid \left|\partial \partial_e \cF\left[\Gnp\right] \right| \le 2 n^4 p^5 \right)
        \\&\le \Pr\left(\text{Bin}(7 \cdot 2 n^4 p^5, q) \ge \frac{\hat{C} n^4 p^5 q}{7}\right) \ll n^{-10},
    \end{align*}
    where the first inequality is true since $\frac{1}{7} \left|\partial \partial_e \cF_i\left[\Gnp\right] \right| \le |Z_e \cap \cF_i|$, the second inequality is true since $|Z_e \cap \cF_i| \le 7 \left|\partial \partial_e \cF_i\left[\Gnp\right] \right|$ and the last inequality is true provided that $\hat{C}$ is large enough and the assumption $n^4 p^5 q \ge \hat{C} \log n$. Therefore, by the union bound over at most $n^3$ choices for $e \in K_n^{(3)}$ and at most $n^3$ choices for $i \in \br{n^3}$, the first item of the lemma holds.
    
    Moving on to the second item of the lemma, fix $Q \in \cQ_2$. We shall bound from above the probability that there exists no $i \in \br{n^3}$ for which the second item holds with respect to $\cF_i$. Fix $i \in \br{n^3}$. For every $2$-tuple $S = (S_1, S_2)$ of disjoint sets of vertices of size at least $n/4$ each that is compatible with $Q$, observe that, for every $F' \in \cFQL[\ext(S)]$, there is a unique choice of $F'' \in \cF$ and $e \in Q$ such that $F'' = F \cup \{e\}$. Also, for every $F'' \in \cF$ that contains exactly one edge of $Q$, there exists a unique $F' \in \cFQL$ such that $F' = F'' \setminus Q$. Consequently, the random variable representing the size of the sparsified family follows a binomial distribution:
    \[
        \left|\cFQL(\cF_i)[\ext(S)]\right| \sim \text{Bin}\left(\left|\cFQL[\ext(S)]\right|, q\right).
    \]
    We note that $\left|\cFQL[\ext(S)]\right| \ge e(Q) \cdot \frac{n^4}{1000} \eqqcolon m$. Since the inequality $\mu_p\left(\cFQL(\cF_i)[\ext(S)]\right) \le \frac{q}{2} \cdot \mu_p\left(\cFQL[\ext(S)]\right)$ is equivalent to $\left|\cFQL(\cF_i)[\ext(S)]\right| \le \frac{q}{2} \left| \cFQL[\ext(S)] \right|$, we have
    \begin{align*}
        \Pr\left(\mu_p\left(\cFQL(\cF_i)[\ext(S)]\right) \le \frac{q}{2} \cdot \mu_p\left(\cFQL[\ext(S)]\right)\right) &\le \max_{m \le t \le n^7}\Pr\left(\text{Bin}(t, q) \le tq / 2\right) \\&\le \max_{m \le t \le n^7} e^{-tq / 12} \le e^{-mq / 12},
    \end{align*}
    where the second inequality is true by Chernoff inequality. By the union bound, the probability that there exists such a $2$-tuple $S$ for which the second item fails, that is
    \[
        \mu_p\left(\cFQL(\cF_i)[\ext(S)]\right) \le \frac{q}{2} \cdot \mu_p\left(\cFQL[\ext(S)]\right),
    \]
    is at most
    \[
        3^n \cdot e^{-mq/12} \ll 1,
    \]
    where the last inequality is true whenever $\hat{C}$ is large enough and $q \ge \hat{C} \cdot \frac{n^{-3}}{e(Q)}$ (implying that $mq \ge \hat{C} \cdot n$).

    Furthermore, 
    \[
        \Delta_p\left(\cFQL(\cF_i)\right) = \sum_{F_1 \in \cFQL} \mathbf{1}_{F_1 \in \cF_i} \Pr\left(F_1 \subseteq \Gnp\right) \sum_{\substack{F_1 \neq F_2 \in \cFQL\\ F_1 \cap F_2 \neq \emptyset}} \mathbf{1}_{F_2 \in \cF_i} \Pr\left(F_2 \subseteq \Gnp \mid F_1 \subseteq \Gnp\right),
    \]
    and thus
    \[
        \mathbb{E}\left[\Delta_p\left(\cFQL(\cF_i)\right)\right] = q^2 \cdot \Delta_p\left(\cFQL\right).
    \]
    Then, by Markov's inequality, 
    \[
        \Pr\left(\Delta_p\left(\cFQL(\cF_i)\right) \ge 2 q^2 \cdot \Delta_p\left(\cFQL\right)\right) \le \frac{1}{2}.
    \]

    We conclude that the probability that $\cF_i$ is bad for a given $Q$ is at most $\frac{2}{3}$ (meaning it fails to satisfy the second item of the lemma). Since the families $\cF_1, \dots, \cF_{n^3}$ are chosen independently, the probability that all $n^3$ families are bad for a fixed $Q$ is at most $\left(\frac{2}{3}\right)^{n^3}$. By the union bound over all $Q$, the probability that the second item of the lemma is not satisfied is at most
    \[
        2^{\binom{n}{3}} \cdot \left(\frac{2}{3}\right)^{n^3} \ll 1. \qedhere
    \]    
\end{proof}

We now define the parameters for the case where $p > C n^{-2/3} \left(\log n\right)^{1/6}$ and $Q \in \cQ_2$. Let $\hat{C}$ be the constant guaranteed by Lemma \ref{lemma:sparsification} and set $q = \min\left\{1, \frac{\hat{C} \log n}{\kappa n^4 p^6}\right\}$. Recalling that $e(Q) \ge \frac{n^2 p}{\log^3 n}$ for every $Q \in \cQ_2$ by the definition of $\cQ_2$, we observe that the assumptions of Lemma \ref{lemma:sparsification} are satisfied. Indeed, we have
\[
    q = \min\left\{1, \frac{\hat{C} \log n}{\kappa n^4 p^6}\right\} \ge \frac{\log^3 n}{n^5 p} \ge \max_{Q \in \cQ_2} \frac{1}{n^3 \cdot e(Q)} \quad \text{and} \quad q = \min\left\{1, \frac{\hat{C} \log n}{\kappa n^4 p^6}\right\} \ge \frac{\hat{C} \log n}{ n^4 p^5}.
\]
Let $\cF_1, \dots, \cF_{n^3} \subseteq \cF$ be the families guaranteed by Lemma \ref{lemma:sparsification}. For every $Q \in \cQ_2$, let $i = i(Q) \in \br{n^3}$ be the index guaranteed by the second consequence of Lemma \ref{lemma:sparsification} and set $\tilde{\cF}_Q$ to be the set of all copies $F' \setminus Q$ where $F' \in \cF_i$ contains exactly one edge of $Q$. The choice of the parameters in this case is the following.
\[
    \cF_Q = \tilde{\cF_Q},\quad d_Q = \sqrt{\eta} \cdot e(Q) \log n,\quad m_Q = \kappa \cdot e(Q) n^4 p^6 \cdot q \quad \text{and} \quad D_Q = \frac{\hat{C}}{\kappa}.
\]

The following table summarises the parameter selection for each case.
\begin{table}[h!]
\centering
\renewcommand{\arraystretch}{1.5}
\begin{tabular}{|c|c|c|c|c|}
\hline
\textbf{} & \textbf{$d_Q$} & \textbf{$m_Q$} & \textbf{$D_Q$} & \textbf{$\cF_Q$} \\
\hline
\makecell{$Q \in \cQ_1 \cup \cQ_2$ and \\ $p \le C n^{-2/3} \left(\log n\right)^{1/6}$} & $\beta n^3 p \wedge \sqrt{\eta} \cdot e(Q) \log n$ & $\kappa \cdot e(Q) n^4 p^6$ & $\frac{2}{\kappa}$ & $\cFQL$ \\
\hline
\makecell{$Q \in \cQ_1$ and \\ $p > C n^{-2/3} \left(\log n\right)^{1/6}$} & $30 e(Q)$ & $30 e(Q)$ & $\frac{n^3 p}{d_Q}$ & $\cFQL$ \\
\hline
$Q \in \cQ_3$ & $\frac{33 \cdot e(Q)}{\eta}$ & $\frac{33 \cdot e(Q)}{\eta}$ & $\frac{14 \cdot M(k(Q))}{m_Q}$ & $\cFQH$ \\
\hline
\makecell{$Q \in \cQ_2$ and \\ $p > C n^{-2/3} \left(\log n\right)^{1/6}$} & $\beta n^3 p \wedge \sqrt{\eta} \cdot e(Q) \log n$ & $\kappa \cdot e(Q) n^4 p^6 \cdot q$ & $\frac{\hat{C}}{\kappa}$ & \makecell{Random sample \\ of $\cFQL$} \\
\hline
\end{tabular}
\caption{A summary of the chosen parameters for Lemma \ref{lemma:main-switching-lemma}. (For brevity, $a \wedge b$ denotes $\min\{a,b\}$ and $M(k) \coloneqq k \cdot n^6 p^7 \wedge n^3 p$)}
\label{table:parameters for Q}
\end{table}

\subsection{The proof for Lemma \ref{lemma:main-lemma}}\label{section:proof-of-main-lemma}
As described in Section \ref{subsection:setting-the-ground}, establishing \eqref{align:main-lemma-desired-probability} is sufficient to complete the proof of Lemma \ref{lemma:main-lemma}. This is achieved by applying Lemma \ref{lemma:main-switching-lemma} to each $Q$ using the parameters specified in Section \ref{section:choosing-parameters-for-swithcing-lemma}.

 For every $Q \in \cQ$, denote by $\cD_Q$ the set of graphs $G$ that satisfy 
\begin{align}\label{align:D_Q inequality}
    \min\left\{\binom{n}{3}, \left|\partial\cF_Q\left[\Gnp\right]\right|\right\} \le D_Q \cdot m_Q/p.
\end{align}
Recall that $\cQ\left[\Gnp\right]$ denotes the collection of all graphs $Q \in \cQ$ such that $Q \subseteq \Gnp$. Further, since $\cB_{Q, \cF_Q}(d_Q, d_Q)$ only contains graphs $G$ for which $Q \subseteq G$, we have
\begin{align*}
    \Pr\left(\bigcup_{Q \in \cQ} \Gnp \in \cB_{Q, \cF_Q}(d_Q, d_Q)\right) &\le \Pr\left(\Gnp \notin \cT\right) + \Pr\left(\Gnp \notin \bigcup_{Q \in \cQ\left[\Gnp\right]} \cD_Q^c \right) \\&+ \Pr\left(\Gnp \in \bigcup_{Q \in \cQ} \cT \cap \cD_Q \cap \cB_{Q, \cF_Q}(d_Q, d_Q)\right).
\end{align*}
By Claim \ref{claim:typical properties}, we have that $\Pr\left(\Gnp \notin \cT\right) = o(1)$. The following claim states that the second term is also $o(1)$.
\begin{claim}
    $\Pr\left(\Gnp \in \bigcup_{Q \in \cQ\left[\Gnp\right]} \cD_Q^c \right) = o(1).$
\end{claim}
\begin{proof}
    We will show that whp \eqref{align:D_Q inequality} holds for every $Q \in \cQ\left[\Gnp\right]$.

    Assume first that $p \le C n^{-2/3} \left(\log n\right)^{1/6}$ and $Q \in \cQ_1 \cup \cQ_2$. In this case, $D_Q \cdot m_Q / p = 2 e(Q) \cdot n^4 p^5$. Every element of $\partial \cF_Q$ is a copy of $F$ minus two edges (an edge from $Q$ and another edge). Hence, every such copy belongs to $\partial \partial_e \cF$ for some $e \in Q$. We thus have
    $\partial \cF_Q\left[\Gnp\right] \subseteq \bigcup_{e \in Q} \partial \partial_e \cF\left[\Gnp\right]$.
    By the first property of Claim \ref{claim:Fano minus edge} and a union bound over all $e \in K_n^{(3)}$, whp
    \[
        \left|\partial \cF_Q\left[\Gnp\right]\right| \le \left|\bigcup_{e \in Q} \partial \partial_e \cF\left[\Gnp\right]\right| \le 2 e(Q) n^4 p^5 = D_Q \cdot m_Q / p,
    \]
    establishing \eqref{align:D_Q inequality} in this case.
    
    Assume now that $p > C n^{-2/3} \left(\log n\right)^{1/6}$. If $Q \in \cQ_1$, then $D_Q \cdot m_Q / p = n^3 > \binom{n}{3}$. So we may assume that $Q \in \cQ_2$. In this case, $D_Q \cdot m_Q / p = e(Q) \cdot \hat{C} n^4 p^5 q$. Let $\cF_{i(Q)}$ be the random subset of $\cFQL$ chosen for $Q$. Similarly to before, we have $\partial \cF_Q\left[\Gnp\right] \subseteq \bigcup_{e \in Q} \partial \partial_e \cF_{i(Q)}\left[\Gnp\right]$. Hence, by the first property of Lemma \ref{lemma:sparsification}, whp
    \[
        \left|\partial \cF_Q\left[\Gnp\right]\right| \le \left|\bigcup_{e \in Q} \partial \partial_e \cF_{i(Q)}\left[\Gnp\right]\right| \le e(Q) \cdot \hat{C} n^4 p^5 q = D_Q \cdot m_Q / p.
    \]

    Lastly, assume that $Q \in \cQ_3\left[\Gnp\right]$ and let $k(Q)$ denote the number of centres of the stars in $Q$. In this case, $D_Q \cdot m_Q / p = 14 \cdot \min\left\{k(Q) n^6 p^6, n^3\right\}$. Let $C(Q)$ denote the set of the centres of the stars of $Q$. Every element of $\partial \cF_Q$ is a copy of $F$ minus four edges (three edges from $Q$ and another edge). Since $Q \subseteq \Gnp$, every such copy is a subgraph of some element in $\partial \cF_v\left[\Gnp\right]$ for some $v \in C(Q)$, where $\cF_v$ is the family of all copies from $\cF$ containing $v$. Furthermore, every element in $\partial \cF_v$ contains at most $e(F)=7$ subgraphs which are an element in $\partial \cF_Q$. Then, we have 
    $\left|\partial \cF_Q\left[\Gnp\right]\right| \le 7 \left|\bigcup_{v \in C(Q)} \partial \cF_v\left[\Gnp\right]\right|$.
    By the second property of Claim \ref{claim:Fano minus edge} and a union bound over all $v \in \br{n}$, whp
    \[
        \left|\partial \cF_Q\left[\Gnp\right]\right| \le 7\left|\bigcup_{v \in C(Q)} \partial  \cF_v\left[\Gnp\right]\right| \le 14 k(Q) n^6 p^6,
    \]
    which implies that
    \[
        \min\left\{\binom{n}{3}, \left|\cF_Q\left[\Gnp\right]\right|\right\} \le 14 \min\left\{k(Q) n^6 p^6, n^3\right\} = D_Q \cdot m_Q / p. \qedhere
    \]
\end{proof}

Getting ready for the application of Lemma \ref{lemma:main-switching-lemma}, notice that it is straightforward to check that $d_Q = o(n^3 p)$ for every $Q \in \cQ$ as $e(Q) = o(n^3 p)$ and $\beta = o(1)$. Furthermore, if $Q \in \cQ_1 \cup \cQ_2$, then
    \[
        d_Q \le \sqrt{\eta} e(Q) \log n \le \sqrt{\eta} |V(Q)| \cdot \Delta_1(Q) \log n \le |V(Q)| n^2 p,
    \]
    where the last inequality is true since $\Delta_1(Q) \le \tilde{\epsilon} \frac{n^2 p}{\log n}$. Moreover, if $Q \in \cQ_3$, then
    \[
        d_Q = \frac{33e(Q)}{\eta} \le \frac{|V(Q)| \cdot \eta n^2 p}{\eta} \le |V(Q)| n^2 p.
    \]
By Lemma \ref{lemma:main-switching-lemma} (applied with the parameters chosen in Section \ref{section:choosing-parameters-for-swithcing-lemma}), for any large constant $K$, there exists a large constant $Z = Z(\alpha)$ such that
\begin{align}\label{align:main-lemma-probability}
    \Pr\left(\Gnp \in \bigcup_{Q \in \cQ} \cT \cap \cD_Q \cap \cB_{Q, \cF_Q}(d_Q, d_Q)\right) \le \sum_{Q \in \cQ} p^{e(Q)} D_Q^{d_Q}\left(Z^{d_Q} \cdot \xi_Q + e^{-K d_Q \log(n^3 p / d_Q)}\right),
\end{align}
where
\begin{align*}
    \xi_Q \coloneqq \Pr\Bigg( &\Gnp \text{ is $(\cC_Q, 0, \alpha)$-rigid} \\
    &\text{and } \nu\left(\cF_Q\left[\Gnp \cap \ext\left(\core_{(\cC_Q, 0)}\left(\Gnp\right)\right)\right]\right) \le 5d_Q + m_Q \;\Bigg|\; Q \subseteq \Gnp \Bigg).
\end{align*}

The following two claims help us upper bound \eqref{align:main-lemma-probability}.
\begin{claim}\label{claim:estimation-of-D_Q K}
    We have
    \begin{enumerate}[topsep=0pt, parsep=0pt]
        \item $D_Q^{d_Q} e^{-K d_Q \log(n^3 p / d_Q)} \le e^{-(1+0.5\epsilon) e(Q) \log\left(\frac{n^3 p}{e(Q)}\right)}$ for every $Q \in \cQ_1 \cup \cQ_2$.
        \item $D_Q^{d_Q} e^{-K d_Q \log(n^3 p / d_Q)} \le e^{-K \cdot e(Q)}$ for every $Q \in \cQ_3$.
    \end{enumerate}
\end{claim}

\begin{claim}\label{claim:estimation-of-xi_Q}
    We have
    \begin{enumerate}[topsep=0pt, parsep=0pt]
        \item $(D_Q Z)^{d_Q} \cdot \xi_Q \le e^{-(1+0.5\epsilon) e(Q) \log\left(\frac{n^3 p}{e(Q)}\right)}$ for every $Q \in \cQ_1 \cup \cQ_2$.
        \item $(D_Q Z)^{d_Q} \cdot \xi_Q \le e^{-K \cdot e(Q)}$ for every $Q \in \cQ_3$.
    \end{enumerate}
\end{claim}

Before proving the above claims, we first show how they imply Lemma \ref{lemma:main-lemma}, that is we show that \eqref{align:main-lemma-desired-probability} holds. We have
\begin{align*}
    \Pr\left(\Gnp \in \bigcup_{Q \in \cQ_1 \cup \cQ_2} \cT \cap \cD_Q \cap \cB_Q\right) &\le \sum_{Q \in \cQ_1 \cup \cQ_2} p^{e(Q)} D_Q^{d_Q}\left(Z^{d_Q} \xi_Q + e^{-K d_Q \log(n^3 p / d_Q)}\right) \\
    &\le \sum_{Q \in \cQ_1 \cup \cQ_2} p^{e(Q)} \cdot 2 e^{-(1+0.5\epsilon) e(Q) \log\left(\frac{n^3 p}{e(Q)}\right)}.
\end{align*}
Since $e(Q) \le \frac{n^3 p}{\log n}$ for every $Q \in \cQ_1 \cup \cQ_2$, the last term above is at most
\begin{align*}    
    2\sum_{m=1}^{n^3 p/\log n} \binom{\binom{n}{3}}{m} p^m \cdot  e^{-(1+0.5\epsilon)m \log\left(\frac{n^3 p}{m}\right)} 
    &\le 2\sum_{m=1}^{n^3 p/\log n} \left(\frac{e\binom{n}{3}p}{m}\right)^{m} e^{-(1+0.5\epsilon)m \log\left(\frac{n^3 p}{m}\right)} \\
    &\le 2\sum_{m=1}^{n^3 p/\log n} e^{m - 0.5 \epsilon m \log\left(n^3 p / m\right)},
\end{align*}
where the second inequality follows from Claims \ref{claim:estimation-of-D_Q K} and \ref{claim:estimation-of-xi_Q}. Therefore,
\begin{align*}
    \Pr\left(\Gnp \in \bigcup_{Q \in \cQ_1 \cup \cQ_2} \cT \cap \cD_Q \cap \cB_Q\right) &\le 2\sum_{m=1}^{n^3 p / \log n} e^{m - 0.5 \epsilon m \log\log n} \\&\le 2\sum_{m=1}^{n^3 p / \log n} \left(e^{-0.1 \epsilon \log \log n}\right)^m = o(1).
\end{align*}

Let us analyse the union bound over $\cQ_3$, 
\begin{align*}
    \Pr\left(\Gnp \in \bigcup_{Q \in \cQ_3} \cT \cap \cD_Q \cap \cB_Q\right) &\le \sum_{Q \in \cQ_3} p^{e(Q)} D_Q^{d_Q}\left(Z^{d_Q} \cdot \xi_Q + e^{-K d_Q \log(n^3 p / d_Q)}\right) \\
    &\le \sum_{Q \in \cQ_3} p^{e(Q)} 2e^{-K e(Q)}.
\end{align*}
Observe that the number of $Q \in \cQ_3$ with exactly $k$ centres is at most $\binom{n}{k} \left(\sum_{m \ge 0} \binom{n^2}{m} p^m\right)^k$. Indeed, we have at most $\binom{n}{k}$ choices for the centres of the stars in $Q$. Then, for each centre, we should choose pairs of vertices for which the centre will form an edge with in $Q$. So if a centre should have $m$ edges incident to it in $Q$, then we have at most $\binom{n^2}{m}$ choices for these edges. Further, by the definition of $\cQ_3$, there are no edges containing more than one centre. Thus, whenever $K$ is sufficiently large,
\begin{align*}
    \Pr\left(\Gnp \in \bigcup_{Q \in \cQ_3} \cT \cap \cD_Q \cap \cB_Q\right) &\le 2\sum_{k=1}^n \binom{n}{k} \left(\sum_{m \ge 0} \binom{n^2}{m} p^m\right)^k e^{-2k n^2 p} \\
    &= 2\sum_{k=1}^n \binom{n}{k} \left(1 + p\right)^{n^2 k} e^{-2k n^2 p} 
    \le 2\sum_{k=1}^n \binom{n}{k} e^{-k n^2 p} \\
    &= 2\left((1 + e^{-n^2 p})^n - 1\right)
    \le 2\left(e^{n e^{-n^2 p}} - 1\right)
    \\&\le 2\left(2n e^{-n^2 p}\right) = o(1),
\end{align*}
where the last inequality is true since $n e^{-n^2 p} \in [0, 1]$ and $1 + x \ge e^{x/2}$ for every $x \in [0, 1]$.

Let us prove Claims \ref{claim:estimation-of-D_Q K} and \ref{claim:estimation-of-xi_Q}.
\begin{proof}[Proof of Claim \ref{claim:estimation-of-D_Q K}]
    Assume first that $Q \in \cQ_2$ (and no assumption on $p$), or $Q \in \cQ_1$ and $p \le C n^{-2/3} \left(\log n\right)^{1/6}$. Recalling that $\beta \ll 1$, and $d_Q \le \beta n^3 p$ in this case, we have
    \[
        \log(n^3 p / d_Q) \ge \log\left(\frac{1}{\beta}\right).
    \]
    Then,
    \begin{align*}
        D_Q^{d_Q} e^{-K d_Q \log(n^3 p / d_Q)} &\le \left(\frac{\hat{C}}{\kappa}\right)^{d_Q} e^{-K d_Q \log(n^3 p / d_Q)} \le e^{-0.5 K d_Q \log(n^3 p / d_Q)} \\&\le e^{-0.5K e(Q) \log(n^3 p / e(Q))},
    \end{align*}
    where the second inequality is true since $\log\left(\frac{\hat{C}}{\kappa}\right) \ll \log(1/\beta)$ because $\beta = o(1)$. The last inequality is true since the function $f(x) = x \log (a/x)$ is increasing in $(0, a/e)$ and $e(Q) < d_Q < n^3 p / e$.

    Assume that $p > C n^{-2/3} \left(\log n\right)^{1/6}$ and $Q \in \cQ_1$. Then, 
    \begin{align*}
        D_Q^{d_Q} e^{-K d_Q \log(n^3 p / d_Q)} &= \left(\frac{n^3 p}{d_Q}\right)^{d_Q} e^{-K \cdot d_Q\cdot  \log(n^3 p / d_Q)} = e^{-(K-1) d_Q \log(n^3 p / d_Q)} \\&\le e^{-(K-1) e(Q) \log(n^3 p / e(Q))}.
    \end{align*}

    Lastly, assume that $Q \in \cQ_3$. Then, since $D_Q \le 14 n^3 p / d_Q$,
    \[
        D_Q^{d_Q} e^{-K d_Q \log(n^3 p / d_Q)} \le 14^{d_Q} e^{-(K-1) d_Q \log(n^3 p / d_Q)} \le e^{-0.5 K \cdot d_Q} \le e^{-K \cdot e(Q)},
    \]
    where the second inequality is true since $d_Q = o(n^3 p)$ and the last inequality is true since $d_Q \ge 2 e(Q)$.
\end{proof}
\begin{proof}[Proof of Claim \ref{claim:estimation-of-xi_Q}]
    For every $Q \in \cQ$, let $\cS_Q$ denote the collection of all $2$-tuples $(S_1, S_2)$ of disjoint sets of vertices of size at least $(1-\alpha)n/2$ that are compatible with $Q$. For every $S \in \cS_Q$, the event $\nu\left(\cF_Q\left[\Gnp \cap \ext\left(S\right)\right]\right) \le 5d_Q + m_Q$ is decreasing and determined by $\ext(S)$. Thus, by Lemma \ref{lemma:correlation-lemma}, 
    \[
        \xi_Q \le 2\max_{(S_1, S_2) \in \cS_Q} \Pr\left(\nu\left(\cF_Q\left[\Gnp \cap \ext(S)\right]\right) \le 5d_Q + m_Q\right).
    \]

    Let $Q \in \cQ_1 \cup \cQ_2$ and assume that $p \le C \cdot n^{-2/3} \left(\log n\right)^{1/6}$. By Lemma \ref{lemma:mu-Delta-bounds-low-degree} (recalling that $\cF_Q = \cFQL$), for every $S \in \cS_Q$,
    \[
        \mu_p\left(\cF_Q\left[\ext(S)\right]\right) \ge \left(\frac{5}{4} - \Theta(\alpha)\right) \cdot e(Q) \cdot \frac{n^4 p^6}{2^4} \ge \left(1 - \Theta(\alpha) + \epsilon\right) \cdot \frac{7}{3} \cdot e(Q) \log n \ge (1+0.9\epsilon) \frac{7}{3} \log n,
    \]
    where the second inequality is true by \eqref{align:p definition} and the last inequality is true whenever $\alpha$ is sufficiently smaller than $\epsilon$. In addition, by Lemma \ref{lemma:mu-Delta-bounds-low-degree}, noting that $\Delta_p(\cF_Q[\ext(S)]) \le \Delta_p(\cF_Q)$
    \[
        \Delta_p\left(\cF_Q\left[\ext(S)\right]\right) = O\left(C_{Low} \cdot \tilde{\epsilon} \cdot \frac{n^4 p^6}{\log n} \cdot \mu_p\left(\cF_Q\left[\ext(S)\right]\right)\right) = O\left(C_{Low} \cdot \tilde{\epsilon} \cdot C^6 \cdot \mu_p\left(\cF_Q\left[\ext(S)\right]\right)\right).
    \]
    We have $5d_Q + m_Q = O\left((\sqrt{\eta} + \kappa) e(Q) n^4 p^6\right) = O\left((\sqrt{\eta} + \kappa) \mu_p\left(\cF_Q\left[\ext(S)\right]\right)\right)$. Thus, provided that $\eta$ and $\kappa$ are sufficiently small, we have, by Corollary \ref{cor:Janson-matchings},
    \begin{align*}
        \xi_Q &\le e^{-\left(1-\Theta\left(\left(\sqrt{\eta}+\kappa\right)^{0.5}\right)\right)\mu_p\left(\cFQL\left[\ext(S)\right]\right) + \Delta_p\left(\cFQL\left[\ext(S)\right]\right)} \\
        &\le e^{-\left(1-\Theta\left(\left(\sqrt{\eta}+\kappa\right)^{0.5}\right) - O\left(C_{Low} \cdot \tilde{\epsilon} \cdot C^6\right)\right)\mu_p\left(\cFQL\left[\ext(S)\right]\right)} \\
        &\le e^{-\left(1-\Theta\left(\left(\sqrt{\eta}+\kappa\right)^{0.5}\right) - O\left(C_{Low} \cdot \tilde{\epsilon} \cdot C^6\right)\right) \cdot \left(1 + 0.9\epsilon\right) \cdot \frac{7}{3} e(Q) \log n}
        \le e^{-(1 + 0.8\epsilon) \cdot \frac{7}{3} \cdot e(Q) \log n},
    \end{align*}
    where the last inequality is true whenever $\left(\sqrt{\eta}+\kappa\right)^{0.5}$ and $C_{Low} \cdot \tilde{\epsilon} \cdot C^6$ are sufficiently smaller than $\epsilon$ (which is possible since the constants $\eta$ and $\kappa$ are chosen after $\epsilon$, and the constant $\tilde{\epsilon}$ is chosen after $C_{Low}, C$ and $\epsilon$). Then,
    \[
        (D_Q Z)^{d_Q} \cdot \xi_Q \le \left(\frac{2Z}{\kappa}\right)^{\sqrt{\eta} e(Q) \log n} \cdot e^{-(1 + 0.8\epsilon) \cdot \frac{7}{3} \cdot e(Q) \log n} \le e^{-(1 + 0.7\epsilon) \cdot \frac{7}{3} \cdot e(Q) \log n},
    \]
    where the last inequality is true whenever $\sqrt{\eta} \log\left(2Z/\kappa\right)$ is sufficiently smaller than $\epsilon$ (which is allowed since $Z$ depends only on $\alpha$, and $\eta$ is chosen after $\alpha, \kappa$ and $\epsilon$). Lastly, since $p \le C n^{-2/3} \left(\log n\right)^{1/6}$, we have
    \[
        \log\left(\frac{n^3 p}{e(Q)}\right) \le \log\left(n^3 p\right) \le \log\left(C n^{7/3} \left(\log n\right)^{1/6}\right) = \left(\frac{7}{3} + o(1)\right) \log n.
    \]
    and thus
    \[
        (D_Q Z)^{d_Q} \cdot \xi_Q \le  e^{-(1 + 0.7\epsilon) \cdot \frac{7}{3} \cdot e(Q) \log n} \le e^{-(1 + 0.5\epsilon) e(Q) \log\left(\frac{n^3 p }{e(Q)}\right)}.
    \]

    Assume now that $p > C \cdot n^{-2/3} \left(\log n\right)^{1/6}$ and $Q \in \cQ_2$. Set set $q = \min\left\{1, \frac{\hat{C} \log n}{\kappa n^4 p^6}\right\}$. Lemma \ref{lemma:sparsification} states that there exist $\cF_1, \dots, \cF_{n^3} \subseteq \cF$ such that each one of them satisfies the first property of the lemma and, for $Q \in \cQ_2$, there exists $i \in \br{n^3}$ such that the other properties hold for $\cF_Q(\cF_i)$. Thus,
    \begin{align*}
        \mu_p\left(\cFQL(\cF_i)\left[\ext(S)\right]\right) &\ge 0.5 q \mu_p\left(\cFQL\left[\ext(S)\right]\right) \ge 0.5 q \left(\frac{5}{4} - \Theta(\alpha)\right) \cdot e(Q) \cdot \frac{n^4 p^6}{2^4} \\&= \Omega\left(\min\left\{\frac{\hat{C}}{\kappa} \cdot e(Q) \log n, C^6 e(Q) \log n\right\}\right)
    \end{align*}
    and
    \begin{align*}
        \Delta_p\left(\cFQL(\cF_i)\right) &\le 2 q^2 \cdot \Delta_p\left(\cFQL\right) \le 2q^2 \cdot \frac{C_{Low} \cdot \tilde{\epsilon} n^4 p^6}{\log n} \cdot \mu_p\left(\cFQL\left[\ext(S)\right]\right) \\&\le \frac{2 \hat{C} \cdot C_{Low} \cdot \tilde{\epsilon}}{\kappa} \cdot \mu_p\left(\cFQL\left[\ext(S)\right]\right) q \le 0.1 \cdot q \mu_p\left(\cFQL\left[\ext(S)\right]\right) \\
        &\le 0.2 \mu_p\left(\cFQL(\cF_i)\left[\ext(S)\right]\right),
    \end{align*}
    where the first inequality is true by the second item in Lemma \ref{lemma:sparsification} (the second subitem), the second inequality is true by Lemma \ref{lemma:mu-Delta-bounds-low-degree}, the third inequality is true since $q \le \frac{\hat{C} \log n}{\kappa n^4 p^6}$, the fourth inequality is true whenever $\frac{2 \hat{C} \cdot C_{Low} \cdot \tilde{\epsilon}}{\kappa}$ is sufficiently small (which is possible since $\tilde{\epsilon}$ is chosen after the constants $\hat{C}, C_{Low}$ and $\kappa$), and the last inequality is true by the second item of Lemma \ref{lemma:sparsification} (the first subitem).

    Moreover, $5d_Q + m_Q = (5 \sqrt{\eta} + \hat{C}) e(Q) \log n < 0.01 \mu_p\left(\cFQL\left[\ext(S)\right]\right)$ whenever $\kappa$ and $\eta$ are sufficiently small and $C$ is sufficiently larger than $\hat{C}$. By Corollary \ref{cor:Janson-matchings}.
    \[
        \xi_Q \le e^{-0.5 \mu_p\left(\cFQL(\cF_i)\left[\ext(S)\right]\right) + 2\Delta_p\left(\cF_Q\left[\ext(S)\right]\right)} \le e^{-0.1 \mu_p\left(\cFQL(\cF_i)\left[\ext(S)\right]\right)} \le e^{-\min\{C, \hat{C}\} \cdot e(Q) \log n},
    \]
    Then,
    \[
        (D_Q Z)^{d_Q} \cdot \xi_Q \le \left(\frac{\hat{C} \cdot Z}{\kappa}\right)^{\sqrt{\eta} e(Q) \log n} \cdot e^{-\min\{C, \hat{C}\} e(Q) \log n} \le e^{-(1+0.5\epsilon) e(Q) \log\left(\frac{n^3 p}{e(Q)}\right)},
    \]    
    where the second inequality is true whenever $C$ and $\hat{C}$ are sufficiently large and $\sqrt{\eta} \log\left(\frac{\hat{C} \cdot Z}{\kappa}\right)$ is sufficiently small (which is possible since $\eta$ is chosen after all the other relevant constants).

    Assume now that $p > C \cdot n^{-2/3} \left(\log n\right)^{1/6}$ and $Q \in \cQ_1$. In this case, $\cF_Q = \cFQL$. By Lemma \ref{lemma:mu-Delta-bounds-low-degree}, for every $S \in \cS_Q$,
    \[
        \mu_p\left(\cF_Q\left[\ext(S)\right]\right) \ge \left(\frac{5}{4} - \Theta(\alpha)\right) \cdot e(Q) \cdot \frac{n^4 p^6}{2^4} = \Omega\left(C^6 e(Q) \log n\right)
    \]
    and
    \[
        \frac{\mu_p\left(\cF_Q\left[\ext(S)\right]\right)^2}{\Delta_p\left(\cF_Q\left[\ext(S)\right]\right)} \ge \frac{\left(\frac{5}{4} - \Theta(\alpha)\right) \cdot e(Q) \cdot \frac{n^4 p^6}{2^4}}{\frac{C_{Low} \cdot \tilde{\epsilon} \cdot n^4 p^6}{\log n}} = \Omega\left(\frac{1}{C_{Low} \cdot \tilde{\epsilon}} \cdot e(Q) \log n\right).
    \]
    Moreover, $5d_Q + m_Q = O(e(Q)) \le 0.0001 \min\left\{\mu_p\left(\cF_Q\left[\ext(S)\right]\right), \frac{\mu_p\left(\cF_Q(\Pi)\right)^2}{\Delta_p\left(\cF_Q\left(\Pi\right)\right)}\right\} \eqqcolon \Lambda $. Hence, provided that $C$ is large enough and $\frac{1}{C_{Low} \cdot \tilde{\epsilon}}$ is sufficiently larger than $C$, we have, by Corollary \ref{cor:extended-Janson-matchings},
    \[
        \xi_Q \le e^{-0.1 \Lambda} \le e^{-C \cdot e(Q) \log n}.
    \]
    Then,
    \[
        (D_Q Z)^{d_Q} \cdot \xi_Q \le \left(n^4\right)^{30e(Q)} \cdot e^{-C \cdot e(Q) \log n} \le e^{-(1+0.5\epsilon) e(Q) \log\left(\frac{n^3 p}{e(Q)}\right)},
    \]
    where the last inequality is true whenever $C$ is sufficiently large.

    Assume now that $Q \in \cQ_3$ and let $k$ be the number of centres of the stars in $Q$. In this case, $\cF_Q = \cFQH$. By Lemma \ref{lemma:mu-Delta-bound-high-degree}, for every $S \in \cS_Q$,
    \[
        \mu_p\left(\cF_Q\left[\ext(S)\right]\right) = \Theta( k n^6 p^7) \gg k n^2 p \ge e(Q) \quad \text{and} \quad \frac{\mu_p\left(\cF_Q\left[\ext(S)\right]\right)^2}{\Delta_p\left(\cF_Q\left[\ext(S)\right]\right)} = \Omega(n^3 p) \gg e(Q).
    \]
    Moreover, $5d_Q + m_Q = O(e(Q)/\eta) \le 0.0001 \min\left\{\mu_p\left(\cF_Q\left[\ext(S)\right]\right), \frac{\mu_p\left(\cF_Q\left[\ext(S)\right]\right)^2}{\Delta_p\left(\cF_Q\left[\ext(S)\right]\right)}\right\} \eqqcolon \Lambda $. Then, there exists $c > 0$ such that $\Lambda \ge c \cdot \min\{k n^6 p^7, n^3 p\}$.
    Hence, 
    \[
        \xi_Q \le e^{-0.1\Lambda} \le e^{-0.1 c \cdot \min\{k n^6 p^7, n^3 p\}}.
    \]
    Further,
    \[
        (D_Q Z)^{d_Q} = e^{\frac{33e(Q)}{\eta} \log\left(14 Z \eta \cdot \frac{\min\{kn^6 p^7, n^3p\}}{33e(Q)}\right)}.
    \]
    Observe that $\min\{k n^6 p^7, n^3 p\} \gg e(Q)$ implies $\min\{k n^6 p^7, n^3 p\} \gg e(Q) \log\left(\frac{\min\{k n^6 p^7, n^3 p\}}{e(Q)}\right)$. Thus,
    \[
        (D_Q Z)^{d_Q} \cdot \xi_Q \le e^{-0.01 c \cdot \min\{k n^6 p^7, n^3 p\}} \le e^{-K \cdot e(Q)},
    \]
    for any constant $K > 0$.
\end{proof}

\section{Proof of Lemma \ref{lemma:main-switching-lemma}}\label{section:proof-of-switching-lemma}
Denote by $\cG_Q$ the set of graphs $G \in \cT \cap \cD \cap \cB_{Q, \cF_Q}(d, d)$.
Pick $G_0 \sim \Gnp$ conditioned on $G_0 \in \cG_Q$ and set $R_0$ to be the empty graph. Recall that $\cC_Q$ is the set of all $\delta$-balanced cuts $\Pi = (A_1, A_2)$ that are compatible with $Q$ (that is, $V^i(Q) \subseteq A_i$ for every $i \in \br{2}$). Let $\Pi = (A_1, A_2) \in \cC_Q$ be the cut indicating that $G_0 \in \cB_{Q, \cF_Q}(d, d)$. Then, $\Pi$ satisfies the following.
\begin{enumerate}
    \item $\deficit_{\cC_Q}(\Pi; G_0) \le d$.
    \item $\nu\left(\cF_Q\left[G_0 \cap \ext\left(\Pi\right)\right]\right) \le d$.
\end{enumerate}
For convenience, we omit the subscript of $\cC_Q$ in the definitions of the deficit, max-cuts and critical edges. In this section, we always mean that the family of cuts is $\cC_Q$.

Let $\gamma = \gamma(\alpha) > 0$ be a sufficiently small constant. Let $K$ and $\Gamma$ be sufficiently large constants and set
\[
    L \coloneqq \Gamma K \cdot d \cdot \log\left(\frac{n^3 p}{d}\right).
\]
For every $0 \le i \le L - 1$, do the following.
\begin{enumerate}[label=(\alph*)]
    \item\label{item:alg-type-a} Let $U_i$ be the set of edges $e \in \int(\Pi)$ such that there exists $\omega \in \cF_Q$ satisfying
    \[
        e \in \omega \quad \text{and} \quad \omega \subseteq G_i \cap \crit(G_i \cup R_i).
    \]
    If $|U_i| \ge m$, then set $G_{i+1} \coloneqq G_i \setminus e$ for an edge $e \in U_i$ chosen u.a.r. Otherwise,

    \item\label{item:alg-type-b} If $e(\crit(G_i \cup R_i) \cap \int(\Pi)) \ge \gamma n^3 p$, then set $G_{i+1} \coloneqq G_i \setminus e$ for an edge 
    \[
        e \in \left(G_i \cap \crit(G_i \cup R_i) \cap \int(\Pi)\right) \setminus Q
    \]       
    chosen u.a.r. Otherwise,

    \item\label{item:alg-type-c} If $G_i \cup R_i$ is not $(\cC_Q, 0, \alpha)$-rigid, then set $G_{i+1} \coloneqq G_i \setminus e$ for an edge 
    \[
        e \in (G_i \cap \int(\Pi)) \setminus (\crit(G_i \cup R_i) \cup Q)
    \] 
    chosen u.a.r. Otherwise,

    \item\label{item:alg-type-d} If $G_i \cup R_i$ is $(\cC_Q, 0, \alpha)$-rigid but $Q$ is not contained in the $(\cC_Q, 0)$-core, then, letting $\{S_1, S_2\}$ be the $(\cC_Q, 0)$-core of $G_i \cup R_i$, we have that $V^1(Q)$ and $S_1$ are in the same part in some but not all max-cuts of $G_i \cup R_i$ among $\cC_Q$. Set $R_{i+1} \coloneqq R_i \cup e$ for an edge $e \in \ext(\Pi) \setminus (G_i \cup R_i)$ which consists of one vertex from $V^1(Q)$ and two vertices from $S_1$, chosen u.a.r. Otherwise,

    \item Set $T = i$ and stop.    
\end{enumerate}
If the algorithm did not stop in the middle, then set $T = L$. Note that if $T < L$, then $G_T \cup R_T$ is $(\cC_Q, 0, \alpha)$-rigid and $Q$ is contained in the $(\cC_Q, 0)$-core. Although it is not clear at this point that there are possible edges to choose from in items \ref{item:alg-type-c} and \ref{item:alg-type-d}, Claim \ref{claim:algorithm-bounds-forwards} will show that there are actually many choices.

For every $0 \le i \le T - 1$, we say that the $i$-th step in the process is of type \ref{item:alg-type-a}, \ref{item:alg-type-b}, \ref{item:alg-type-c} or \ref{item:alg-type-d}, if the step going from $G_i \cup R_i$ to $G_{i+1} \cup R_{i+1}$ is of type \ref{item:alg-type-a}, \ref{item:alg-type-b}, \ref{item:alg-type-c} or \ref{item:alg-type-d}, respectively. 
Let us state a few properties of the above process.

\begin{claim}\label{claim:alg-properties} 
    \begin{enumerate}
        \item The number of steps of types \ref{item:alg-type-a} and \ref{item:alg-type-b} (in total) is at most $d$.
        
        \item If a step is of type \ref{item:alg-type-d}, then the next step cannot be of type \ref{item:alg-type-c}. Moreover, there cannot be more than one step of type \ref{item:alg-type-d} in a row.
        
        \item There are at most $d$ steps of type \ref{item:alg-type-d}.
    \end{enumerate}
\end{claim}
\begin{proof}
    To follow the proof of the claim, we should understand how the deficit of $\Pi$ changes in each step.

    We first show that the deficit is decreased in steps of type \ref{item:alg-type-a} and \ref{item:alg-type-b}, and remains unchanged in the other steps. If the $i$-th step is of type \ref{item:alg-type-a} or \ref{item:alg-type-b}, then we remove an edge $e \in \crit(G_i \cup R_i)$ and thus $b(G_{i+1} \cup R_{i+1}) = b(G_i \cup R_i) - 1$. Also, we have $e \in \int(\Pi)$. Hence, $\deficit(\Pi; G_{i+1} \cup R_{i+1}) = \deficit(\Pi; G_i \cup R_i) - 1$. If the $i$-th step is of type \ref{item:alg-type-c}, then $e \notin \crit(G_i \cup R_i)$ and thus $b(G_{i+1} \cup F_{i+1}) = b(G_i \cup R_i)$. Moreover, $e \in \int(\Pi)$. Therefore, $\deficit(\Pi; G_{i+1} \cup F_{i+1}) = \deficit(\Pi; G_i \cup R_i)$. If the $i$-th step is of type \ref{item:alg-type-d}, then $Q$ is not contained in the $(\cC_Q, 0, \alpha)$-core $\{S_1, S_2\}$ of $G_i \cup R_i$. In this step, we add an edge that crosses $V^1(Q)$ and $S_1$ and we have that $V^1(Q)$ and $S_1$ are in the same part in some but not all max-cuts of $G_i \cup R_i$. Thus, $b(G_{i+1} \cup R_{i+1}) = b(G_i \cup R_i) + 1$. Furthermore, we have $e \in \ext(\Pi)$ and thus we have again $\deficit(\Pi; G_{i+1} \cup R_{i+1}) = \deficit(\Pi; G_i \cup R_i)$. Therefore, since the deficit cannot be negative and $\deficit(\Pi; G_0 \cup R_0) \le d$, there are at most $d$ steps of types \ref{item:alg-type-a} and \ref{item:alg-type-b}.

    We continue to the second property of the claim. If the $i$-th step of \ref{item:alg-type-d}, then $G_i \cup R_i$ is $(\cC_Q, 0, \alpha)$-rigid. As explained above, we add an edge $e$ that crosses some max-cut of $G_i \cup R_i$ and intersects only $V^1(Q)$ and $S_1$. Therefore, the set of max-cuts of $G_{i+1} \cup R_{i+1}$ is exactly the set of max-cuts of $G_i \cup R_i$ for which $V^1(Q)$ and $S_1$ are in different parts. In particular, the set of max-cuts of $G_{i+1} \cup R_{i+1}$ is contained in the set of max-cuts of $G_i \cup R_i$. This implies that $G_{i+1} \cup R_{i+1}$ is $(\cC_Q, 0, \alpha)$-rigid as well and $Q$ is now contained in the $(\cC_Q, 0)$-core. Hence, the next step cannot be of type \ref{item:alg-type-c}. Indeed, letting $\{S'_1, S'_2\}$ be the $(\cC_Q, 0)$-core of $G_{i+1} \cup F_{i+1}$, we have $S_j \subseteq S'_j$ for every $j \in \br{2}$. Thus, in every max-cut of $G_{i+1} \cup R_{i+1}$, $V^1(Q)$ and $S'_2$ are in the same part. Thus, after one step of this type, $Q$ is contained in the $(\cC_Q, 0)$-core of $G_{i+1} \cup R_{i+1}$ and thus the next step cannot be again of type \ref{item:alg-type-d}.
\end{proof}

The next claim gives a lower bound to the number of possibilities going forward in the above algorithm. That is, given $G_i \cup F_i$, we give a lower bound to the number of choices of $G_{i+1} \cup R_{i+1}$.

\begin{claim}\label{claim:algorithm-bounds-forwards}
    The algorithm satisfies the following for every $0 \le i \le T-1$.
    \begin{enumerate}
        \item If the $i$-th step is of type \ref{item:alg-type-a}, then, given $G_i \cup R_i$, there are at least $m$ possibilities for $G_{i+1} \cup R_{i+1}$.

        \item If the $i$-th step is of type \ref{item:alg-type-b}, then, given $G_i \cup R_i$, there are at least $0.5\gamma n^3 p$ possibilities for $G_{i+1} \cup R_{i+1}$.

        \item If the $i$-th step is of type \ref{item:alg-type-c}, then, given $G_i \cup R_i$, there are at least $(1-0.25\alpha) \cdot \frac{1}{4} \cdot \binom{n}{3}p$ possibilities for $G_{i+1} \cup R_{i+1}$.

        \item If the $i$-th step is of type \ref{item:alg-type-d}, then, given $G_i \cup R_i$, there are at least $|V^1(Q)| \cdot \frac{n^2}{20}$ possibilities for $G_{i+1} \cup R_{i+1}$.
    \end{enumerate}
\end{claim}
\begin{proof}
    The lower bound of the number of steps in \ref{item:alg-type-a} is due to the definition of the algorithm.

    We begin by showing the lower bound for the second item. Let $G_i \cup R_i$ satisfies the condition of type \ref{item:alg-type-b} while the condition of type \ref{item:alg-type-a} is not satisfied. Then, we have
    \begin{align*}
        e\left(\left(G_i \cap \crit(G_i \cup R_i) \cap \int(\Pi)\right) \setminus Q\right) \ge e(\crit(G_i \cup R_i) \cap \int(\Pi)) - e(R_i) - e(Q).
    \end{align*}
    Since $G_i \cup R_i$ satisfies the condition of type \ref{item:alg-type-b}, we have $e(\crit(G_i \cup R_i) \cap \int(\Pi)) \ge \gamma n^3 p$. Moreover, we have $e(R_i) = O(d) = o(n^3 p)$ and $e(Q) = o(n^3 p)$. Hence,
    \begin{align*}
        e\left(\left(G_i \cap \crit(G_i \cup R_i) \cap \int(\Pi)\right) \setminus Q\right) \ge 0.5 \gamma n^3 p.
    \end{align*}

    We move to the third item. Assume that $G_i \cup R_i$ satisfies the condition of type \ref{item:alg-type-c} and does not satisfy the conditions of the previous items. We have
    \begin{align*}
        e\left((G_i \cap \int(\Pi)) \setminus (\crit(G_i \cup R_i) \cup Q)\right) \ge e\left(G_i \cap \int(\Pi)\right) - e\left(\crit(G_i \cup R_i) \cap \int(\Pi)\right) - e(Q).
    \end{align*}
    For the first term, 
    \begin{align*}
        e(G_i \cap \int(\Pi)) \ge e(G_0 \cap \int(\Pi)) - i \ge e(\int(\Pi)) \cdot p - o(n^3 p) - i,
    \end{align*}
    where the first inequality is true since $G_i \subseteq G_0$ and $e(G_0 \setminus G_i) \le i$, and the last inequality is true since $G_0 \in \cT$. Moreover, since $G_i \cup R_i$ does not satisfy the condition of type \ref{item:alg-type-b}, we have
    \begin{align*}
        e(\crit(G_i \cup R_i) \cap \int(\Pi)) < \gamma n^3 p.
    \end{align*}
    Therefore,
    \begin{align*}
        e\left((G_i \cap \int(\Pi)) \setminus (\crit(G_i \cup R_i) \cup Q)\right) &\ge e(\int(\Pi)) \cdot p - o(n^3 p) - i - \gamma n^3 p - e(Q) \\
        &\ge e(\int(\Pi)) \cdot p - 2\gamma n^3 p,
    \end{align*}
    where the last inequality is true since $e(Q) = o(n^3 p)$ and $i \le L = o(n^3 p)$. Lastly, recall that $\Pi \in \cC_Q$ is $\delta$-balanced, and thus
    \begin{align*}
        e\left((G_i \cap \int(\Pi)) \setminus (\crit(G_i \cup R_i) \cup Q)\right) \ge \left(1 - 0.25\alpha\right) \cdot \frac{1}{4} \cdot \binom{n}{3} p,
    \end{align*}
    where the last inequality is true whenever $\gamma$ and $\delta$ are sufficiently smaller than $\alpha$.

    Lastly, we show the fourth item. Assume that $G_i \cup R_i$ satisfies the condition of type \ref{item:alg-type-d} and does not satisfy the conditions of the previous items. 
    Let $\{S_1, S_2\}$ be the $(\cC_Q, 0)$-core of $G_i \cup R_i$. Since $Q$ is not contained in the core,  there exists some max-cut of $G_i \cup R_i$ such that $S_1$ and $V^1(Q)$ are in the same part and some max-cut of $G_i \cup R_i$ such that $S_2$ and $V^1(Q)$ are in the same part. Recall that $\Pi = (A_1, A_2)$ is compatible with $Q$ and thus $V^1(Q) \subseteq A_1$. As $G_i \cup R_i$ does not satisfy the condition of Step \ref{item:alg-type-b}, we must have $\min\{|S_1 \cap A_1|, |S_2 \cap A_1|\} \le 2 \gamma^{1/3} n$. Indeed, otherwise, since $G_0 \in \cT$, letting $\Pi'$ be a $2$-tuple of vertex disjoint sets $S_1 \cap A_1$ and $S_2 \cap A_1$, we have
    \begin{align*}
        e\left((G_{i}\cup R_i)[\ext(\Pi')]\right) &\ge e\left(G_0[\ext(\Pi')]\right) - i \ge 4\gamma n^3 p - o(n^3 p) - i \ge \gamma n^3 p,
    \end{align*}
    a contradiction to the assumption that $G_i \cup R_i$ does not satisfy the condition of \ref{item:alg-type-b} since 
    \[
        (G_{i}\cup R_i)[\ext(\Pi')] \subseteq \int(\Pi) \cap \crit(G_i \cup F_i).
    \]
    
    Thus, we may assume WLOG that $|S_1 \cap A_1| \le 2 \gamma^{1/3} n$ which implies that $|S_1 \setminus A_1| \ge n/3$ whenever $\alpha$ is small enough and $\gamma^{1/3}$ is sufficiently smaller than $\alpha$. Hence, the number of edges $e \in K_n^{(3)}$ with one endpoint in $V^1(Q)$ and another two endpoints in $S_1 \setminus A_1$ is at least $|V^1(Q)| \cdot n^2 / 18$. Note that each such edge $e$ belongs to $\ext(\Pi)$. Also, in Step \ref{item:alg-type-d}, we also require that $e \notin G_i \cup R_i$. Since $G_0 \in \cT$, we have $O\left(|V^1(Q)| n^2 p\right)$ edges in $G_i \cup R_i$ consisting of one vertex in $V^1(Q)$ and another two vertices in $S_1 \setminus A_1$. Hence, since $p = o(1)$, the number of choices of $e$ in this case is at least $|V^1(Q)| n^2 / 20$.
\end{proof}
The following claim gives an upper bound to the number of possibilities going backwards. That is, given $G_{i+1} \cup R_{i+1}$, we give an upper bound to the number of choices of $G_i \cup R_i$. For convenience, set $N \coloneqq \binom{n}{3}$. Trivially, the number of choices of $G_i \cup R_i$ given $G_{i+1} \cup R_{i+1}$ is at most $N$.
\begin{claim}\label{claim:algorithm-bounds-backwards}
    The algorithm satisfies the following for every $0 \le i \le T-1$.
    \begin{enumerate}
        \item If the $i$-th step is of type \ref{item:alg-type-a}, then, given $G_{i+1} \cup R_{i+1}$, there are at most $D \cdot m/p$ possibilities for $G_{i} \cup R_{i}$.

        \item If the $i$-th and the $(i+1)$-th steps are of type \ref{item:alg-type-c}, then, given $G_{i+1} \cup R_{i+1}$ and $G_{i+2} \cup R_{i+2}$, there are at most $(1-0.5\alpha)\cdot \frac{1}{4} \cdot N$ possibilities for $G_{i} \cup R_{i}$.
        
        \item If the $i$-th step is of type \ref{item:alg-type-d}, then, given $G_{i+1} \cup R_{i+1}$, there are at most $3 |V(Q)| n^2 p$ possibilities for $G_{i} \cup R_{i}$.
    \end{enumerate}
\end{claim}
\begin{proof}
    For the first item, note that $R_{i} = R_{i+1}$ and $G_{i} = G_{i+1} \cup \{e\}$ for some edge $e \notin G_{i+1}$ such that $e$ creates a copy of some $F \in \cF_Q$ in $G_{i+1}$. Thus, there are at most $\min\{N, |\partial\cF_Q[G_{i+1}]|\}$ such edges. Since $G_{i+1} \subseteq G_{0}$, we have $\partial\cF_Q[G_{i+1}] \subseteq \partial\cF_Q[G_{0}]$. Thus, recalling that $G_0 \in \cG_Q$ and thus $G_0 \in \cD$, there are at most $D \cdot m / p$ such edges.

    We continue to the second item. In this case, $R_{i} = R_{i+1}$ and $G_{i} = G_{i+1} \cup \{e\}$ for some $e \notin G_{i+1} \cup R_{i+1}$ for which, in particular, $e \notin \crit(G_i \cup R_i)$ and thus $b(G_i \cup R_i) = b(G_{i+1} \cup R_{i+1})$. Thus, $e$ must be contained in some $(\cC_Q, 0)$-component of $G_{i+1} \cup R_{i+1}$. Since the $(i+1)$-th step is of type \ref{item:alg-type-c} as well, then $G_{i+1} \cup R_{i+1}$ is not $(\cC_Q, 0, \alpha)$-rigid. Therefore, there are at most $(1-0.5\alpha)\cdot \frac{1}{4} \cdot N$ choices for $e$.
    
    For the last item, assume the $i$-th step is of type \ref{item:alg-type-d}. Then, we need to remove an edge $e_i \in G_{i+1} \cup R_{i+1}$ which is incident to $V(Q)$. Note that $G_{i+1} \cup R_{i+1}$ is a subgraph of some graph satisfying $\cT$ union at most $d$ edges (since $G_0 \cup R_0 \in \cT$ and $e(R_t) \le d$ by Claim \ref{claim:alg-properties}). Since $d \le |V(Q)| \cdot n^2 p$, there are at most $3 |V(Q)| n^2 p$ edges incident to $V(Q)$ to choose in this step.
\end{proof}

Let us now derive Lemma \ref{lemma:main-switching-lemma}. Denote by $T_a, T_b, T_c$ and $T_d$ the random variables indicating the number of steps of types \ref{item:alg-type-a}, \ref{item:alg-type-b}, \ref{item:alg-type-c} and \ref{item:alg-type-d} in the algorithm, respectively. 
\begin{claim}\label{claim:number-of-steps}
    There exist $t_a, t_b, t_c, t_d$ such that
    \[
        \Pr((T_a, T_b, T_c, T_d) = (t_a, t_b, t_c, t_d)) \ge \frac{1}{2(t_a+t_b+t_c+t_d)^6}.
    \]
\end{claim}
\begin{proof}
    Suppose towards contradiction that the claim is false. Then, letting $W_{t_a, t_b, t_c, t_d}$ be the event that $(T_a, T_b, T_c, T_d) = (t_a, t_b, t_c, t_d)$, we have
    \begin{align*}
        1 &= \sum_{t_a, t_b, t_c, t_d} \Pr(W_{t_a, t_b, t_c, t_d}) = \sum_{t=1}^{L} \sum_{\substack{t_a, t_b, t_c, t_d \\ t_a + t_b + t_c + t_d = t}} \Pr(W_{t_a, t_b, t_c, t_d}) 
        \le \sum_{t=1}^{L} t^4 \cdot \frac{1}{2t^6} \le \sum_{t=1}^{\infty} 0.5 t^{-2} < 1,
    \end{align*}
    a contradiction.
\end{proof}

Let $t_a, t_b, t_c, t_d$ be integers guaranteed by Claim \ref{claim:number-of-steps}. For convenience, set $t \coloneqq t_a + t_b + t_c + t_d$. Denote by $W$ the event that $(T_a, T_b, T_c, T_d) = (t_a, t_b, t_c, t_d)$. Then, we have
\begin{align}\label{align:W-probability-lower-bound}
    \Pr(W) \ge 0.5 t^{-6}.
\end{align}
We will bound $\Pr(W)$ from above as well.

For every sequence $\Sigma = (G_i, R_i)_{i=0}^{T}$, denote by $\cA(\Sigma)$ the event that the algorithm produced the sequence $\Sigma$. Denote by $\cS$ the set of sequences $\Sigma = (G_i, R_i)_{i=0}^{T}$ which can be produced by the algorithm and the numbers of steps of types \ref{item:alg-type-a}, \ref{item:alg-type-b}, \ref{item:alg-type-c} and \ref{item:alg-type-d} in $\Sigma$ are $t_a, t_b, t_c$ and $t_d$, respectively. Then,
\begin{align*}
    \Pr(W) = \sum_{\Sigma = (G_i, R_i)_{i=0}^{t}\in \cS} \Pr(\cA(\Sigma)) = \sum_{\Sigma = (G_i, R_i)_{i=0}^{t}\in \cS} \Pr(\bG_0 = G_0) \cdot \Pr(\cA(\Sigma) \mid \bG_0 = G_0).
\end{align*}
Note that, for every $\Sigma = (G_i, R_i)_{i=0}^{t} \in \cS$, we have that $e(G_t \cup R_t) = e(G_0) - t + 2t_d$. Thus,
\begin{align*}
    \Pr(\bG_0 = G_0) &= \frac{\Pr\left(\Gnp = G_0\right)}{\Pr\left(\Gnp \in \cG_Q\right)} = \frac{p^{e(G_0)} (1-p)^{N-e(G_0)}}{\Pr\left(\Gnp \in \cG_Q\right)} 
    = \frac{\Pr\left(\Gnp = G_t \cup R_t\right)}{\Pr\left(\Gnp \in \cG_Q\right)} \cdot \left(\frac{p}{1-p}\right)^{t - 2t_d}.
\end{align*}
Hence, we get
\begin{align}\label{align:W-probability}
    \Pr(W) = \sum_{\Sigma = (G_i, R_i)_{i=0}^{t}\in \cS} \frac{\Pr\left(\Gnp = G_t \cup R_t\right)}{\Pr\left(\Gnp \in \cG_Q\right)} \cdot \left(\frac{p}{1-p}\right)^{t - 2t_d} \cdot \Pr(\cA(\Sigma) \mid \bG_0 = G_0).
\end{align}

Claim \ref{claim:algorithm-bounds-forwards} gives an upper bound to the probability to choose the correct next graph for each type of step. Therefore, we have
\begin{align}\label{align:algorithm-probability-bound}
    \Pr(\cA(\Sigma) \mid \bG_0 = G_0) \le \left(\frac{1}{m_Q}\right)^{t_a} \left(\frac{2}{\gamma n^3 p}\right)^{t_b} \left(\frac{4}{(1-0.25\alpha) N p}\right)^{t_c} \left(\frac{20}{|V(Q)| n^2}\right)^{t_d}.
\end{align}

Set $t_{\bar{c}} \coloneqq t - t_c$ which indicates the total number of steps not of type \ref{item:alg-type-c}. Denote by $\cE_t$ all the graphs $G$ for which there exists $\Sigma = (G_i, R_i)_{i=0}^{t} \in \cS$ with $G_t \cup R_t = G$. 
\begin{claim}\label{claim:size-of-sum-bound}
    For every $G \in \cE_t$, the number of sequences $\Sigma = (G_i, R_i)_{i=0}^{t} \in \cS$ with $G_t \cup F_t = G$ is at most
    \[
        \binom{t}{t_{\bar{c}}} \cdot 3^{t_{\bar{c}}} \cdot \left(\frac{D \cdot m_Q}{p}\right)^{t_a} \cdot N^{t_b + t_{\bar{c}}} \cdot \left(\frac{(1-0.5\alpha) N}{4}\right)^{t_c - t_{\bar{c}}} \cdot \left(3|V(Q)| \cdot n^2 p\right)^{t_d}.
    \]
\end{claim}
\begin{proof}
    Fix $G \in \cE_t$. Every sequence $\Sigma = (G_i, R_i)_{i=0}^{t} \in \cS$ with $G_t \cup R_t = G$ may be constructed by the following steps.
    \begin{enumerate}
        \item Choose the types of the $t$ steps.        
        \item Choose the sequence $\Sigma = (G_i, R_i)_{i=0}^{t} \in \cS$ with $G_t \cup R_t = G$ such that, for every $i \in \{0, \dots, t-1\}$, then $(G_i, R_i)$ falls into the condition of a step of the chosen type.
    \end{enumerate}

    There are at most $\binom{t}{t_{c}} = \binom{t}{t_{\bar{c}}}$ options to choose the steps of type \ref{item:alg-type-c} and then at most $3^{t_{\bar{c}}}$ options to choose the types of the rest of the steps. Fix $i \in \br{t-1}$ and assume we have already chosen $G_j \cup R_j$, for every integer $i+1 \le j \le t$. We will bound the number of choices of $G_i \cup R_i$ from above. Let us stress here that we upper bound the number of choices of the union $G_i \cup R_i$ and not the ordered tuple $(G_i, R_i)$. Having said that, this is enough for determining all the tuples since we know that $R_0 = \emptyset$ and we may only delete and add edges from $G_i$ and $R_i$, respectively.

    Assume that the $i$-th step is not of type \ref{item:alg-type-d}. Then, we need to choose to add an edge $e_i \notin G_{i+1} \cup R_{i+1}$. If the $i$-th step is of type \ref{item:alg-type-a}, then, by Claim \ref{claim:algorithm-bounds-backwards}, the number of choices is at most $D \cdot m / p$. If the $i$-th and the $(i+1)$-th step is of type \ref{item:alg-type-c}, then the number of choices is at most $(1-0.5\alpha) \cdot \frac{1}{4} \cdot N$ for $G_i \cup R_i$. Further, there are at least $t_c - t_{\bar{c}}$ such pair of indices $i$. For the other cases, we will bound the number of choices by the trivial bound $N$.

    Notice that whenever we determined $G_i \cup R_i$, for every $i \in \br{t}$, then the sequence $(G_i, R_i)_{i=0}^{t}$ is determined as well. Indeed, we know that $R_0 = \emptyset$ and thus $(G_0, R_0)$ is determined. Next, for every $i \in \{0, \dots, t-1\}$ and given $(G_i, R_i)$, the type of the $i$-th step and the edge $e_i$, we know the next tuple $(G_{i+1}, R_{i+1})$.

    Hence, the number of possible sequences $\Sigma = (G_i, R_i)_{i=0}^{t} \in \cS$ with $G_t \cup R_t = G$ is at most
    \[
        \binom{t}{t_c} \cdot 3^{t_{\bar{c}}} \cdot \left(\frac{D \cdot m}{p}\right)^{t_a} \cdot N^{t_b + t_{\bar{c}}} \cdot \left(\frac{(1-0.5\alpha) N}{4}\right)^{t_c - t_{\bar{c}}} \cdot \left(3|V(Q)| \cdot n^2 p\right)^{t_d}.       \qedhere 
    \]
\end{proof}

By \eqref{align:algorithm-probability-bound}, and by Claim \ref{claim:size-of-sum-bound}, we get from \eqref{align:W-probability} that $\Pr(W)$ is at most
\begin{align*}    
    \frac{\Pr\left(\Gnp \in \cE_t\right)}{\Pr\left(\Gnp \in \cG_Q\right)} \cdot \left(\frac{p}{1-p}\right)^{t - 2t_d} \cdot \binom{t}{t_c} 3^{t_{\bar{c}}} \cdot \left(\frac{D}{p}\right)^{t_a} \cdot \left(\frac{1}{\gamma p}\right)^{t_b} \cdot \left(\frac{1-0.5\alpha}{(1-0.25\alpha)p}\right)^{t_c} \cdot (60p)^{t_d}.
\end{align*}
Since $t_a + t_b + t_c - t_d = t - 2t_d$ and $t_a \le d$, we have
\begin{align}\label{align:final W bound}
    \nonumber\Pr(W) &\le \frac{\Pr\left(\Gnp \in \cE_t\right)}{\Pr\left(\Gnp \in \cG_Q\right)} \cdot \binom{t}{t_c} 3^{t_{\bar{c}}} \cdot \left(\frac{D}{1-p}\right)^{t_a} \\&\nonumber \quad \cdot \left(\frac{1}{\gamma (1-p)}\right)^{t_b} \cdot \left(\frac{1-0.5\alpha}{(1-0.25\alpha)(1-p)}\right)^{t_c} \cdot \left(60(1-p)\right)^{t_d}\\
    &\le D^{d} \cdot \left(\frac{C_{*}t}{t_{\bar{c}}}\right)^{t_{\bar{c}}} \cdot \left(\frac{1-0.5\alpha}{1-0.3\alpha}\right)^{t} \cdot \frac{\Pr\left(\Gnp \in \cE_t\right)}{\Pr\left(\Gnp \in \cG_Q\right)},
\end{align}
where $C_{*} = C_{*}(p_0, \gamma, \alpha)$ is a sufficiently large constant. The last inequality is true since $p=o(1)$ and, by Claim \ref{claim:alg-properties}, $t_d \le d$.
By \eqref{align:W-probability-lower-bound} and \eqref{align:final W bound},
\begin{align}
    \nonumber\Pr\left(\Gnp \in \cG_Q\right) &\le 2 t^{6} D^{d} \left[\left(\frac{C_{*}t}{t_{\bar{c}}}\right)^{t_{\bar{c}}} \cdot \left(\frac{1-0.5\alpha}{1-0.3\alpha}\right)^{t}\right] \cdot \Pr\left(\Gnp \in \cE_t\right) \\
    \nonumber&= 2 t^{6} D^{d} \left[\left(\frac{C_{*}t}{t_{\bar{c}}}\right)^{t_{\bar{c}}} \cdot \left(\frac{1-0.5\alpha}{1-0.3\alpha}\right)^{t}\right] \cdot \Pr\left(Q \subseteq \Gnp\right) \Pr\left(\Gnp \in \cE_t \mid Q \subseteq \Gnp\right) \\
    &\le p^{e(Q)} \cdot 2 t^{6} D^{d} \left[\left(\frac{C_{*}t}{t_{\bar{c}}}\right)^{t_{\bar{c}}} \cdot \left(\frac{1-0.5\alpha}{1-0.3\alpha}\right)^{t}\right] \cdot \Pr\left(\Gnp \in \cE_t \mid Q \subseteq \Gnp\right).
\end{align}

Set $\tau \coloneqq \frac{t_{\bar{c}}}{t}$. Then,
\begin{align}\label{align:T-cap-B_Q-upper-bound}
    \nonumber\Pr\left(\Gnp \in \cG_Q\right) &\le p^{e(Q)} \cdot 2 t^{6} D^{d} e^{\tau t \cdot \log(C_{*}/\tau) - t \log\left(\frac{1-0.3\alpha}{1-0.5\alpha}\right)} \cdot \Pr\left(\Gnp \in \cE_t \mid Q \subseteq \Gnp\right) \\
    &\le p^{e(Q)} \cdot 2 t^{6} D^{d} e^{\tau t \cdot \log(C_{*}/\tau) - t \log\left(1+0.1\alpha\right)} \cdot \Pr\left(\Gnp \in \cE_t \mid Q \subseteq \Gnp\right),
\end{align}
where the last inequality is true since $\frac{1-0.3\alpha}{1-0.5\alpha} > 1 + 0.1\alpha$ whenever $\alpha$ is sufficiently small.

\begin{claim}
    If $t = L$, then
    \[
        \Pr\left(\Gnp \in \cG_Q\right) \le p^{e(Q)} D^{d} e^{-K d \log(n^3 p / d)}.
    \]
\end{claim}
\begin{proof}
    By Claim \ref{claim:alg-properties}, $\tau t = t_{\bar{c}} \le 4d$ and thus $\tau \le \frac{4d}{L} \le \frac{4}{\Gamma K \log(n^3 p / d)}$. Then,
    \begin{align*}
        e^{\tau t \cdot \log(C_{*}/\tau) - t \log\left(1+0.1\alpha\right)} 
        &= e^{-\Gamma K d \log(n^3 p / d) \left(\log\left(1+0.1\alpha\right)  -\tau \log(C_{*}/\tau)\right)} \\
        &\le e^{-\sqrt{\Gamma} K d \log(n^3 p / d)},
    \end{align*}        
    where the last inequality is true whenever $K$ and $\Gamma = \Gamma(\alpha, C_{*})$ are sufficiently large and that $x \log\left(\frac{C^*}{x}\right)$ tends to $0$ as $x$ tends to $0$. By \eqref{align:T-cap-B_Q-upper-bound}, whenever $\Gamma$ is sufficiently large, and using the trivial bound $\Pr\left(\Gnp \in \cE_t \mid Q \subseteq \Gnp\right) \le 1$,
    \begin{align*}
        \Pr\left(\Gnp \in \cG_Q\right) \le p^{e(Q)} 2 L^{6} D^{d} e^{-\sqrt{\Gamma} K d \log(n^3 p / d)} \le p^{e(Q)} D^{d} e^{-K d \log(n^3 p / d)}.
    \end{align*}
\end{proof}
Recall the definition of $\xi$ in the statement of Lemma \ref{lemma:main-switching-lemma}.
\begin{claim}
    If $t < L$, then
    \[
        \Pr\left(\Gnp \in \cG_Q\right) \le p^{e(Q)} D^{d} Z^{d} \cdot \xi.
    \]
\end{claim}
\begin{proof}
    Since the function $x \log(C_{*}/x)$ tends to $0$ as $x$ tends to $0$, and decreasing in an interval close enough to $0$, there exists $\tau_0 = \tau_0(C_{*}, \alpha)$ such that $\tau \cdot \log(C_{*} / \tau) \le 0.5\log\left(1+0.1\alpha\right)$ for every $\tau \le 
    \tau_0$.     
    Therefore,
    \begin{align*}
        e^{\tau t \cdot \log(C_{*}/\tau) - t \log\left(1+0.1\alpha\right)} \le e^{\tau t \cdot \log(C_{*}/\tau_0) - 0.5\log\left(1+0.1\alpha\right) t} \le \left(\frac{C_{*}}{\tau_0}\right)^{4d} e^{- 0.5 \log\left(1+0.1\alpha\right) t}.
    \end{align*}
    Hence, by \eqref{align:T-cap-B_Q-upper-bound},
    \begin{align*}
        \Pr\left(\Gnp \in \cG_Q\right) &\le p^{e(Q)} \cdot 2 t^{6} D^{d} \left(\frac{C_{*}}{\tau_0}\right)^{4d} e^{- 0.5 \log\left(1+0.1\alpha\right) t} \cdot \Pr\left(\Gnp \in \cE_t \mid Q \subseteq \Gnp\right).
    \end{align*}
    Since $t^{6} e^{- 0.5 \log\left(1+0.1\alpha\right) t}$ is bounded from above by a constant that depends only on $\alpha$, we have
    \begin{align*}
        \Pr\left(\Gnp \in \cG_Q\right) &\le p^{e(Q)} D^{d} Z^{d} \cdot \Pr\left(\Gnp \in \cE_t \mid Q \subseteq \Gnp\right),
    \end{align*}
    for a large enough constant $Z = Z(C_{*}, \tau_0)$. Since $C_{*}$ depends on $\alpha$ and $\gamma$, $\gamma$ depends only on $\alpha$, and $\tau_0$ depends on $C_*$ and $\alpha$, then $Z$ depends only on $\alpha$.

    Lastly, in order to show that $\Pr\left(\Gnp \in \cE_t \mid Q \subseteq \Gnp\right) \le \xi$, it suffices to show that every graph in $\cE_t$ is $(\cC_Q, 0, \alpha)$-rigid and $\nu\left(\cF_Q\left[(G_t \cup R_t) \cap \ext\left(\core_0(G_t \cup R_t)\right)\right]\right) \le 5d + m$. Let $G \in \cE_t$ and let $(G_i, R_i)_{i=0}^{t} \in \cS$ satisfying $G_t \cup R_t = G$. Since $t < L$, we have that $G_t \cup R_t$ is $(\cC_Q, 0, \alpha)$-rigid. Assume towards contradiction that 
    \[
        \nu\left(\cF_Q\left[(G_t \cup R_t) \cap \ext\left(\core_0(G_t \cup R_t)\right)\right]\right) > 5d + m.
    \] 

    Let $M_1, \dots, M_{5d + m}$ be edge-disjoint elements from $\cF_Q\left[(G_t \cup R_t) \cap \ext\left(\core_0(G_t \cup R_t)\right)\right]$. Since $e(R_t) \le d$ (recalling that $t_d \le d$), then there are at most $d$ indices $i \in \br{5d + m}$ such that $M_i$ uses an edge from $R_t$. Moreover, we have
    \[
        G_t \cap \ext\left(\core_0(G_t \cup R_t)\right) \subseteq G_t \cap \crit(G_t \cup R_t).
    \]
    Since $G_t \cup R_t$ does not satisfy the condition of step of type \ref{item:alg-type-a}, we must have that there are at most $m$ indices $i \in \br{5d + m}$ such that $M_i \subseteq G_t$ and $M_i \cap \int(\Pi) \neq \emptyset$. Hence, we got that there are at least $3d$ indices $i \in \br{5d + m}$ such that 
    \[
        M_i \subseteq G_t \cap \ext\left(\Pi\right) \subseteq G_0 \cap \ext\left(\Pi\right),
    \]
    which implies that $\nu\left(\cF_Q\left[G_0 \cap \ext\left(\Pi\right)\right]\right) \ge 3d$, a contradiction to the assumption that $G_0 \in \cG_Q$.

    


\end{proof}

\section{Proof of Theorem \normalfont{\ref{thm:0-statement}}}\label{section:0-statement}
In this section, we prove Theorem \ref{thm:0-statement}. As previously noted, the majority of the framework was established in \cite{hoshen2024stabilitylargecutsrandom}. To complete the proof of Theorem \ref{thm:0-statement}, we should address the following remaining components:
\begin{itemize}
    \item We must demonstrate that $\Gnp$ typically possesses a core, a result we establish in Corollary \ref{cor:core}.
    \item We provide precise estimates for the concentration of the number of copies of $F \setminus \{e\}$ (the graph $F$ minus an edge) within a fixed cut in $\Gnp$.
\end{itemize}
To this end, we mainly focus on the points mentioned above and omit the proofs for certain statements that have already been established in \cite{hoshen2024stabilitylargecutsrandom}.

Let $\bG \sim \Gnp$.
First, suppose that $\frac{1}{n^2} \ll p \ll n^{-2/3}$.
It follows from the definition of $3$-density and Markov's inequality that whp $\bG$ contains only $o(n^3 p)$ copies of $F$;
consequently, whp $\bG$ contains an $F$-free subgraph with $(1-o(1))\binom{n}{3}p$ edges.
On the other hand, letting $\cC$ be the collection of all cuts, Corollary \ref{cor:equiv_classes_edges_count} implies that whp the largest size of a cut in $\bG$ is at most
\[
  \max_{\Pi \in \cC} |\ext(\Pi)| \cdot p + 3n^2\sqrt{p} = \left(\frac{3}{4} + o(1)\right) \binom{n}{3}p.
\]
Therefore, whp every largest $F$-free subgraph of $\bG$ is not bipartite.
We may thus assume from now on that
\[
  \eps n^{-2/3} \le p \le (1-\eps) \Theta_F \cdot n^{-1/m_3(F)} \cdot (\log n)^{1/(e(F)-1)} = (1-\eps) \Theta_F \cdot n^{-2/3} \cdot (\log n)^{1/6}
\]
for some positive constant $\eps$.

It is clearly sufficient to prove that, with high probability, there exists a max-cut $\Pi$ of $\bG$ and an edge $e \in \int(\Pi) \cap \bG$ such that the subgraph $(\ext(\Pi) \cap \bG) \cup e$ is $F$-free.  We shall prove a stronger statement.
For every collection $\Sigma$ of two disjoint set of vertices, define $\ext^*(\Sigma) \coloneqq K_n^{(3)} \setminus \int(\Sigma)$. Note that every max-cut $\Pi$ of $\bG$ satisfies
\[
  \ext(\Pi) = K_n^{(3)} \setminus \int(\Pi) \subseteq K_n^{(3)} \setminus \int(\core_0(\bG)) = \ext^*(\core_0(\bG)),
\]
provided that $\bG$ has a $0$-core. We remark again that we omit the subindex of the collection of cuts in the definitions of cores, deficit, and related notions, implicitly treating them with respect to $\cC$, the collection of all cuts.

The aforementioned stronger statement that implies Theorem \ref{thm:0-statement} is that whp $\bG$ has a $0$-core with minimum part size $n/2-o(n)$ and there is an edge $e \in \int(\core_0(\bG)) \cap \bG$ such that $\partial_e\cF[\ext^*(\core_0(\bG)) \cap \bG]$ is empty;  note that this implies that $(\ext(\Pi) \cap \bG) \cup e$ is $F$-free for every max-cut $\Pi$ of $\bG$.

To formalise this, set $\alpha \coloneqq (1/\log n)^2$ and define, for each $e \in K_n^{(3)}$, the event
\[
  \cY_e \coloneqq \big\{G \in \Core_0(\alpha) : e \in \int(\core_0(G)) \cap G \wedge \partial_e\cF[\ext^*(\core_0(G)) \cap G] = \emptyset\big\}.
\]
Our goal is to prove that whp $\bG \in \cY_e$ for some $e \in K_n^{(3)}$.
Denoting by $Z$ the number of $e \in K_n^{(3)}$ satisfying $\bG \in \cY_e$, it will be enough to show that
\begin{align}\label{eq:0-statement-goal}
  \Ex[Z^2] \le (1+o(1)) \cdot \Ex[Z]^2.
\end{align}
Indeed, if \eqref{eq:0-statement-goal} holds, then, by the Paley--Zygmund inequality,
\[
  \Pr(\text{$\bG \in \cY_e$ for some $e \in K_n^{(3)}$}) = \Pr(Z \neq 0) \ge \frac{\Ex[Z]^2}{\Ex[Z^2]} = 1-o(1).
\]

\subsection*{Proof outline}

In order to establish \eqref{eq:0-statement-goal}, we separately prove a lower bound on $\Ex[Z]$ and an upper bound on $\Ex[Z^2]$.  We obtain a lower bound on $\Ex[Z]$ using a delicate switching argument that (roughly speaking) goes as follows.  We first choose $d \gg \log n$ so that whp $\bG \in \Core_{2d}(\alpha)$;
this is possible thanks to Corollary \ref{cor:core}.
Fix $e \in K_n^{(3)}$ and assume that $e \in \int(\core_{2d}(\bG)) \cap \bG$.  Our upper-bound assumption on $p$ and the fact that $F$ is strictly $3$-balanced imply that whp $\partial_e\cF[\bG]$ is a matching of size $O(\log n)$, which in turn allows us to analyse the following `resampling' process:  Remove from $\bG$ all the edges of $\bigcup \partial_e \cF[\bG]$, denote the resulting graph by $\bGs$, and consider the conditional distribution of $\bG$ given $\bGs$.  The fact that $\partial_e\cF[\bG]$ is a matching with $o(d)$ edges allows us to infer that $\bGs \in \Core_{d}(\alpha)$ and to essentially couple the conditional distribution of $\partial_e\cF[\bG]$ given $\bGs$ with a $p^{e(F)-1}$-random subset of $\partial_e \cF$, giving
\[
  \Pr\big(\partial_e\cF[\ext^*(\core_d(\bGs)) \cap \bG] = \emptyset \mid \bGs\big) \ge (1-o(1)) \cdot \left(1-p^{e(F)-1}\right)^{|\partial_e\cF[\ext^*(\core_d(\bGs))]|}.
\]
Recall the definitions of the constant $\pi_F$ in \eqref{align:pi_F} and the graph and the graph $K_2^+(a)$ defined just prior to it. Since $\core_d(\bGs)$ is a collection of two pairwise-disjoint sets of size at least $n/2 - \alpha n$ each, we have
\[
  |\partial_e \cF[\ext^*(\core_d(\bGs))]| \le (1+O(\alpha)) \cdot N\left(F, K_2^+\left(\frac{n}{2}\right)\right) = (1+O(\alpha)) \cdot \pi_F \cdot (n/2)^{v(F)-3}.
\]
Further, by Corollary \ref{cor:core-d-nested}, since $\bG$ and $\bGs$ differ in $o(d)$ edges, we have $\core_0(\bG) \succeq \core_d(\bGs)$ (that is, the $d$-core of $\bGs$ is contained in the $0$-core of $\bG$, see the definition in Section \ref{section:rigidity for hypergraphs}). Thus, $\ext^*(\core_0(\bG)) \subseteq \ext^*(\core_d(\bGs))$; consequently, since $n^{v(F)-3} p^{e(F)-1} = O(\log n)$ and $\alpha \ll 1/\log n$,
\begin{multline*}
  \Pr\big(\partial_e\cF[\ext^*(\core_0(\bG)) \cap \bG] = \emptyset \mid \bG \in \Core_{2d}(\alpha) \wedge e \in \int(\core_{2d}(\bG)) \cap \bG\big) \\
  \ge \exp\left(-\pi_F \cdot (n/2)^{v(F)-3} \cdot p^{e(F)-1}-o(1)\right).
\end{multline*}
A lower bound on $\Ex[Z]$ now follows by multiplying the above inequality by the probability of the event in the conditioning and summing the result over all $e \in K_n^{(3)}$.

In order to prove an upper bound on $\Ex[Z^2]$, we use Lemma \ref{lemma:correlation-lemma} which allows us to bound, for every pair $e, f$ of edges of $K_n^{(3)}$, the conditional probability
\[
  \Pr\big((\partial_e\cF \cup \partial_f\cF)[\ext(\core_0(\bG)) \cap \bG] = \emptyset \mid \bG \in \Core_0(\alpha) \wedge e,f \in \int(\core_0(\bG)) \cap \bG\big)
\]
from above by the (unconditional) probability of the same event with $\core_0(\bG)$ replaced by a fixed collection of two pairwise-disjoint sets of at least $n/2 - \alpha n$ vertices each.  The latter probability can be easily shown, using Janson's inequality, to be at most $\exp\big(-2\pi_F \cdot (n/2)^{v(F)-3} \cdot p^{e(F)-1} + o(1)\big)$.  An upper bound on $\Ex[Z^2]$ is then deduced in a straightforward manner by summing the above estimate over all pairs $e, f$.

\subsection*{Organisation}

The remainder of this section is organised as follows.  In Section \ref{sec:preliminaries}, we prove several useful estimates concerning the hypergraph $\partial_e\cF$, and derive estimates on the moments of $|\int(\core_d(\bG)) \cap \bG| \cdot \1_{\bG \in \Core_d(\alpha)}$ from Corollary \ref{cor:core}.  In the remaining two sections, we prove the lower bound on $\Ex[Z]$ and the upper bound on $\Ex[Z^2]$.

\subsection{Preliminaries}
\label{sec:preliminaries}

We start with an estimate on the sizes of subgraphs of $\partial_e \cF$ induced by graphs that are close to a complete, balanced, bipartite graph.
\begin{lemma}
  \label{lemma:copies-of-H-ext}
  Let $\alpha$ be a nonnegative real and suppose that $\Sigma$ is a family of $2$ pairwise-disjoint subsets of $\br{n}$ such that $|X| \ge n/2 - \alpha n$ for each $X \in \Sigma$.
  There is a constant $C_F$ that depends only on $F$ such that, for every $e \in \int(\Sigma)$,
  \begin{align*}
    |\partial_e \cF [\ext(\Sigma)]| & \ge \big(\pi_F - C_F\alpha \big) \cdot (n/2)^{v(F)-3} - C_F n^{v(F)-4}, \\
    |\partial_e \cF [\ext^*(\Sigma)]| & \le \big(\pi_F + C_F\alpha\big) \cdot (n/2)^{v(F)-3}.
  \end{align*}
\end{lemma}
\begin{proof}
  Since $\ext(\Sigma) \supseteq K_2(n/2-\alpha n)$ by our assumption on $\cC$, when $e \in \int(\Sigma)$, we have
  \[
    |\partial_e \cF [\ext(\Sigma)]| \ge N\big(F, K_2^+(n/2-\alpha n)\big) \ge \pi_F \cdot (n/2-\alpha n)^{v(F)-3} - O(n^{v(F)-4}),
  \]
  which implies the first inequality.  (The reason why we may write such explicit error term is that $N(F, K_2^+(m))$ is a polynomial of degree $v(F)-3$ in $m$.)
  Further, since\footnote{We write $G_1 \vee G_2$ for the graph obtained from the disjoint union of $G_1$ and $G_2$ by adding all edges intersecting both $V(G_1)$ and $V(G_2)$.}
  \[
    \ext^*(\Sigma) \subseteq K_2(n/2-\alpha n) \vee K_{2\alpha n}^{(3)},
  \]
  every copy of $F$ minus an edge in $\ext^*(\Sigma)$ that is not fully contained in $K_2(n/2-\alpha n)$ must have at least one vertex in $K_{2\alpha n}^{(3)}$.  Consequently, there is a constant $C_F'$ that depends only on $F$ such that, for each $e \in \int(\Sigma)$, 
  \[
    |\partial_e\cF[\ext^*(\Sigma)]| \le N\big(F, K_2^+(n/2)\big) + C_F' \cdot \alpha n \cdot n^{v(F)-4} \le \pi_F \cdot (n/2)^{v(F)-3} + C_F' \cdot \alpha n^{v(F)-3},
  \]
  which implies the second inequality.
\end{proof}

Our second lemma supplies an upper bound on $\Delta_p(\partial_e \cF \cup \partial_f \cF)$, and thus also on $\Delta_p(\partial_e \cF)$. This lemma is a direct consequence of Lemma \ref{lemma:mu-Delta-bounds-low-degree} applied with a graph $Q = \{e, f\}$ (or $Q = \{e\}$).

\begin{lemma}
  \label{lemma:Delta-p-partial-H}
  For every pair of distinct edges $e, f \in K_n^{(3)}$ and all $p \ge \eps n^{-1/m_3(F)}$,
  \[
    \Delta_p(\partial_e\cF \cup \partial_f \cF) \le Cn^{-\lambda} \cdot \left(n^{v(F)-3} p^{e(F)-1}\right)^2
  \]
  for some positive $\lambda = \lambda(F)$ and $C = C(F,\eps)$.
\end{lemma}
We finish with an estimate on the moments of $\left|\int\left(\core_d\left(\Gnp\right)\right) \cap \Gnp\right| \cdot \1_{\Gnp \in \Core_d(\alpha)}$.
\begin{lemma}
  \label{lemma:edges-int-core}
  Suppose that $\alpha, p \in (0,1/2)$ and a nonnegative integer $d$ satisfy
  \[
    \alpha \ll 1
    \qquad
    \text{and}
    \qquad
    n^2 p \gg \max\left\{\alpha^{-4}, (d/\alpha)^2\right\}.
  \]
  Then, for all fixed $k \ge 0$, the random graph $\bG \sim \Gnp$ satisfies
  \[
    \Ex\left[\left|\int(\core_d(\bG)) \cap \bG\right|^k \cdot \1_{\bG \in \Core_d(\alpha)}\right] = (1+o(1)) \cdot \left(\frac{n^3p}{3! \cdot 2^2}\right)^k.
  \]
\end{lemma}
\begin{proof}
  Set $m \coloneqq \binom{n}{3}p$ and note that our assumption that $n^2 p \gg 1$ guarantees that whp $|\bG| \in [m/2, 3m/2]$.
  In particular, the assumed asymptotic relations between $\alpha$, $p$, and $d$ allow us to conclude from Corollary \ref{cor:core} that whp $\bG \in \Core_d(\alpha)$.  Since for every $G \in \Core_d(\alpha)$, the graph $\int(\core_d(G))$ is a disjoint union of two complete graphs of order at least $n/2 - \alpha n$ each, we have
  \[
    \left|\int(\core_{d}(G))\right| = \frac{n^3}{3! \cdot 2^2} \pm O(\alpha n^3)
  \]
  and further, by the Chernoff bound and the union bound over the at most $3^n$ possible graphs $\int(\core_{d}(G))$,
  \[
    \Pr\left(\forall G \in \Core_d(\alpha) \;\; \left||\int(\core_{d}(G)) \cap \bG| - \frac{n^3 p}{3! \cdot 2^2}\right| = O\left(\alpha + (n^2p)^{-1/2}\right) \cdot n^3 p \right) \ge 1 - e^{-n}.
  \]
  The assertion of the lemma follows, as $\alpha + (n^2p)^{-1/2} \ll 1$ by our assumptions.
\end{proof}

\subsection{Proof of the lower bound on $\Ex[Z]$}
\label{sec:proof-lower-bound-Z}

The following lemma abstracts the essence of the `resampling' procedure that we described in the proof outline presented above.  Given a hypergraph $\cG$, we denote by $\cI(\cG)$ the family of its independent sets. Since the proof is identical to that of \cite[Lemma 4.6]{hoshen2024stabilitylargecutsrandom}, we omit it here.

\begin{lemma}
  \label{lemma:matching-resample}
  Suppose that $\cG$ is a $k$-uniform hypergraph on $V$, let $t$ be a nonnegative integer, and define
  \[
    \cM \coloneqq \{R \subseteq V : \text{$\cG[R]$ is a matching with $\le t$ edges}\}.
  \]
  Suppose further that $p \in (0,1)$, let $\bR \sim V_p$, and let $\bRs \coloneqq \bR \setminus \bigcup \cG[\bR]$.  Then, for all $\cA \subseteq \cI(\cG)$ and $A \colon \cA \to \cP(V)$, letting
  \[
    \frac{q}{1-q} \coloneqq \left(\frac{p}{1-p}\right)^k
    \qquad
    \text{and}
    \qquad
    a \coloneqq \max\big\{|\cG[A(R^*)]| : R^* \in \cA\big\},
  \]
  we have
  \[
    \Pr\big(\bR \in \cM \wedge \bRs \in \cA \wedge \cG[A(\bRs) \cap \bR] = \emptyset\big) \ge
    (1-q)^{a} \cdot \Pr(\bR \in \cM \wedge \bRs \in \cA).
  \]
\end{lemma}

\begin{cor}
  \label{cor:Pr-Ye-lower}
  Suppose that $p \le C n^{-1/m_3(F)} (\log n)^{1/(e(F)-1)}$ for some constant $C$ and let $\bG \sim \Gnp$.
  For every $e \in K_n^{(3)}$ and all $\alpha > 0$, letting $d \coloneqq \left\lceil \log n \right\rceil^2$, we have
  \begin{multline*}
    \Pr(\cY_e) \ge \exp\left(- \big(\pi_F + O(\alpha+p)\big) \cdot (n/2)^{v(F)-3}p^{e(F)-1}\right) \\
    \cdot \left(\Pr\big(\bG \in \Core_{2d}(\alpha) \wedge e \in \int(\core_{2d}(\bG)) \cap \bG\big) - o(p)\right).
  \end{multline*}
\end{cor}
\begin{proof}
  Let $\omega \coloneqq \NN \to \RR$ be an arbitrary function satisfying $1 \ll \omega(n) \ll \log n$ and let
  \[
    t \coloneqq \omega(n) \cdot \Ex|\partial_e\cF[\bG]|.
  \]
  It follows from our upper-bound assumption on $p$ that, for some constants $C_1 = C_1(F)$ and $C_2 = C_2(F,C)$, we have
  \begin{align}
    \label{eq:t-upper}
    t \le \omega(n) \cdot C_1n^{v(F)-3}p^{e(F)-1} \le C_2 \omega(n) \log n \le d/e(F),
  \end{align}
  provided that $n$ is sufficiently large.
  As in Lemma \ref{lemma:matching-resample}, let
  \[
    \cM \coloneqq \big\{G \subseteq K_n^{(3)} : \partial_e\cF[G] \text{ is a matching with $\le t$ edges}\big\}.
  \]
  Further, let $\bGs \coloneqq \bG \setminus \bigcup \partial_e \cF[\bG]$ and define
  \[
    \cA \coloneqq \big\{G^* \subseteq K_n^{(3)} : G^* \in \Core_{d}(\alpha) \wedge e \in \int(\core_d(G^*)) \cap G^* \wedge \partial_e\cF[G^*] = \emptyset\big\}
  \]
  and the function $A \colon \cA \to \cP(K_n^{(3)})$ by $A(G^*) \coloneqq \ext^*(\core_d(G^*))$ for every $G^* \in \cA$.  By Lemma \ref{lemma:copies-of-H-ext}, for some constant $C_F$ that depends only on $F$,
  \[
    \max_{G^* \in \cA}\big|\partial_e\cF[A(G^*)]\big| \le \big(\pi_F + C_F\alpha\big) \cdot (n/2)^{v(F)-3} \eqqcolon a.
  \]
  Further, since $1-x \ge \exp(-x/(1-x))$ for all $x \in (0,1)$, Lemma \ref{lemma:matching-resample} applied to the $(e(F)-1)$-uniform hypergraph $\partial_e \cF$ with vertex set $K_n^{(3)}$ yields
  \begin{align}\label{eq:no-H-on-e-lower-bound}
    \Pr\big(\bG \in \cM \wedge \bGs \in \cA \wedge \partial_e\cF[A(\bGs) \cap \bG] = \emptyset\big) 
    \ge \exp\left(-\left(\frac{p}{1-p}\right)^{e(F)-1} \cdot a\right) \cdot \Pr\big(\bG \in \cM \wedge \bGs \in \cA\big).
  \end{align}
  We now show that \eqref{eq:no-H-on-e-lower-bound} implies the assertion of the corollary.
  
  First, since $(1-x)^{1-e(F)} = 1 + (e(F)-1)x + O(x^2)$ as $x \to 0$, we have
  \[
    \exp\left(-\left(\frac{p}{1-p}\right)^{e(F)-1} \cdot a\right) \ge \exp\left(- \big(\pi_F + C_F\alpha\big) \cdot (1+e(F) \cdot p) \cdot (n/2)^{v(F)-3} p^{e(F)-1}\right).
  \]
  Second, observe that $\bG \in \cM$ implies that
  \begin{equation}
    \label{eq:G-Gs-on-cM}
    e(\bGs) = e(\bG) - (e(F)-1) \cdot |\partial_e\cF[\bG]| \ge e(\bG) - e(F) \cdot t \ge e(\bG) - d,
  \end{equation}
  where the last inequality follows from \eqref{eq:t-upper}.
  Consequently, Corollary \ref{cor:core-d-nested} implies that the event $\bG \in \Core_{2d}(\alpha) \cap \cM$ is contained in the event $\bGs \in \Core_{d}(\alpha)$ and $\int(\core_{2d}(\bG)) \subseteq \int(\core_d(\bGs))$.  We thus have
  \begin{multline*}
    \Pr\big(\bG \in \cM \wedge \bGs \in \cA\big) \ge \Pr\big(\bG \in \Core_{2d}(\alpha) \cap \cM \wedge e \in \int(\core_{2d}(\bG)) \cap \bG\big) \\
    \ge \Pr\big(\bG \in \Core_{2d}(\alpha) \wedge e \in \int(\core_{2d}(\bG)) \cap \bG\big) - \Pr(\bG \notin \cM \wedge e \in \bG).
  \end{multline*}
  Since the events $e \in \bG$ and $\bG \notin \cM$ are independent, we further have
  \[
    \Pr(\bG \notin \cM \wedge e \in \bG) = p \cdot \Pr(\bG \notin \cM) \le p \cdot \left( \Pr(|\partial_e \cF[\bG]| > t) + \Pr\big(\Delta(\partial_e\cF[\bG]) \ge 2\big)\right).
  \]
  The first probability in the right-hand side is at most $1/\omega(n) = o(1)$, by Markov's inequality and the definition of $t$, whereas the second probability can be bounded using Lemma \ref{lemma:Delta-p-partial-H} as follows:
  \[
    \begin{split}
      \Pr\big(\Delta(\partial_e\cF[\bG]) \ge 2\big) & \le \Ex|\{(K,K') \in (\partial_e\cF[\bG])^2 : K \neq K', K \cap K' \neq \emptyset\}| \\
      & = 2\Delta_p(\partial_e \cF) \le O\left(n^{-\lambda} \cdot \left(n^{v(F)-3}p^{e(F)-1}\right)^2\right)
    \end{split}
  \]
  for some positive $\lambda = \lambda(F)$;  since $n^{v(F)-3}p^{e_H-1} = O(\log n)$ under our upper-bound assumption on $p$, we may conclude that $\Pr\big(\Delta(\partial_e\cF[\bG]) \ge 2\big) = o(1)$.
  
  Finally, since $\bG \in \cM$ implies that $e(\bG) \le e(\bGs) + d$, see~\eqref{eq:G-Gs-on-cM}, Corollary \ref{cor:core-d-nested} implies that the event  $\bG \in \cM \wedge \bGs \in \Core_{d}(\alpha)$ is contained in the event that $\bG \in \Core_{0}(\alpha)$ and $\int(\core_d(\bGs)) \subseteq \int(\core_0(\bG))$ (equivalently, that $\ext^*(\core_0(\bG)) \subseteq \ext^*(\core_d(\bGs)) = A(\bGs)$).  Further, since $\bGs \in \cA$ implies that $e \in \int(\core_d(\bGs)) \cap \bGs$, we conclude that
  \[
    \Pr(\cY_e) \ge \Pr\big(\bG \in \cM \wedge \bGs \in \cA \wedge \partial_e\cF[A(\bGs) \cap \bG] = \emptyset\big).
  \]
  The assertion of the lemma follows by combining the above inequality with~\eqref{eq:no-H-on-e-lower-bound} and the lower bounds on the two terms in the right-hand side of~\eqref{eq:no-H-on-e-lower-bound}.
\end{proof}

We are finally ready to complete the derivation of the lower bound on $\Ex[Z]$.
Since we have $n^{v(F)-3}p^{e(F)-1} = O(\log n)$, then
\[
  \big(\pi_F + O(\alpha+p)\big) \cdot (n/2)^{v(F)-3}p^{e(F)-1} = \pi_F \cdot (n/2)^{v(F)-3} p^{e(F)-1} + o(1).
\]
Consequently, we may deduce from Corollary \ref{cor:Pr-Ye-lower} that
\begin{multline*}
  \Ex[Z] = \sum_{e \in K_n} \Pr(\cY_e) \ge \exp\left(-\pi_F \cdot (n/2)^{v(F)-3} p^{e(F)-1} - o(1)\right) \\
  \cdot \left(\Ex\left[\left|\int(\core_{2d}(\bG)) \cap \bG \right| \cdot \1_{\bG \in \Core_{2d}(\alpha)}\right]-o(n^3 p)\right),
\end{multline*}
where $d \coloneqq \left\lceil \log n\right\rceil^2$.
Finally, Lemma \ref{lemma:edges-int-core} allows us to conclude that
\begin{equation}
  \label{eq:Ex-Z-lower}
  \Ex[Z] \ge (1+o(1)) \cdot \exp\left(-\pi_F \cdot (n/2)^{v(F)-3} p^{e(F)-1}\right) \cdot \frac{n^3 p}{3! \cdot 2^2}.
\end{equation}

\subsection{Proof of the upper bound on $\Ex[Z^2]$}
\label{sec:proof-upper-bound-Z}

Given distinct edges $e, f \in K_n^{(3)}$ and a family $\Sigma$ of $2$ pairwise-disjoint subsets of $\br{n}$, define
\begin{align*}
  I_{e,f}(\Sigma) & \coloneqq \big\{G \subseteq K_n^{(3)} : e,f \in \int(\Sigma) \cap G\big\}, \\
  E_{e,f}(\Sigma) & \coloneqq \big\{G \subseteq K_n^{(3)} : \left(\partial_e\cF \cup \partial_f\cF\right)[\ext(\Sigma) \cap G] = \emptyset\big\}.
\end{align*}
and note that, for every graph $G \subseteq K_n^{(3)}$,
\[
  G \in \cY_e \cap \cY_f \quad \Longrightarrow \quad G \in \Core_0(\alpha) \cap I_{e,f}(\core_0(G)) \cap E_{e,f}(\core_0(G)),
\]
as $\ext(\Sigma) \subseteq \ext^*(\Sigma)$ for every family $\Sigma$.  We may thus conclude that
\begin{equation}
  \label{eq:Ex-Z-squared}
  \Ex[Z^2] \le \Ex[Z] + \sum_{\substack{e, f \in K_n^{(3)} \\ e \neq f}} \Pr\big(\bG \in \Core_0(\alpha) \cap I_{e,f}(\core_0(\bG)) \cap E_{e,f}(\core_0(\bG))\big).
\end{equation}
Our next lemma, which is a variant of Lemma \ref{lemma:correlation-lemma}, will allow us to bound from above the probabilities in the right-hand side of~\eqref{eq:Ex-Z-squared}. This lemma is analouge to \cite[Lemma 4.8]{hoshen2023simonovits} and thus we omit its proof.

\begin{lemma}
  \label{lemma:correlation-argument}
  Let $\alpha$ be a nonnegative real and let $\fC$ be the collection of all $2$-element families $\Sigma$ of pairwise-disjoint subsets of $\br{n}$ satisfying $|X| \ge n/2-\alpha n$ for all $X \in \Sigma$.
  Suppose that, for each $\Sigma \in \fC$, we have an event $I(\Sigma)$ that is determined by $\int(\Sigma)$ and an event $E(\Sigma)$ that is determined by $\ext(\Sigma)$ and decreasing, and satisfies $\Pr(\bG \in E(\Sigma)) \le \xi$.  Then,
  \[
    \Pr\big(\bG \in \Core_{0}(\alpha) \cap I(\core_0(\bG)) \cap E(\core_0(\bG))\big) \le \xi \cdot \Pr\big(\bG \in \Core_{0}(\alpha) \cap I(\core_0(\bG))\big).
  \]
\end{lemma}

Returning to~\eqref{eq:Ex-Z-squared}, since clearly $I_{e,f}(\Sigma)$ is determined by $\int(\Sigma)$ whereas $E_{e,f}(\Sigma)$ is determined by $\ext(\Sigma)$ and decreasing, Lemma \ref{lemma:correlation-argument} implies that
\begin{equation}
  \label{eq:Ex-Z-2-upper}
  \begin{split}
    \Ex[Z^2] & \le \Ex[Z] + \xi \cdot \sum_{\substack{e,f \in K_n^{(3)} \\ e \neq f}} \Pr\big(\bG \in \Core_{0}(\alpha) \cap I_{e,f}(\core_0(\bG))\big) \\
    & \le \Ex[Z] + \xi \cdot \Ex\left[\left|\int(\core_0(\bG)) \cap \bG\right|^2 \cdot \1_{\bG \in \Core_{0}(\alpha)}\right],
  \end{split}
\end{equation}
where (writing $\fC$ for the collection of all families $\Sigma$ of two disjoint subsets of $\br{n}$ satisfying $|X| \ge n/2 - \alpha n$ for all $X \in \Sigma$)
\[
  \xi \coloneqq 
  \max_{\Sigma \in \fC} \Pr\big(\bG \in E_{e,f}(\Sigma)\big)
  = \max_{\Sigma \in \fC} \Pr\big((\partial_e\cF \cup \partial_f\cF)[\ext(\Sigma) \cap \bG] = \emptyset\big).
\]
It follows from Lemma \ref{lemma:copies-of-H-ext} that, for every $\Sigma \in \fC$,
\[
  \begin{split}
    \mu_p\big((\partial_e \cF \cup \partial_f \cF)[\ext(\Sigma)]\big)
    & = \mu_p(\partial_e \cF[\ext(\Sigma)]) + \mu_p(\partial_f \cF[\ext(\Sigma)]) \\
    & \ge 2 \big(\pi_F - O(\alpha + n^{-1})\big) \cdot (n/2)^{v(F)-3} p^{e(F)-1} \\
    & \ge 2\pi_F \cdot (n/2)^{v(F)-3} p^{e(F)-1} - o(1),
  \end{split}
\]
where we used that $n^{v(F)-3} p^{e(F)-1} = O(\log n)$ whereas $\alpha \ll 1/\log n$.
On the other hand, Lemma \ref{lemma:Delta-p-partial-H} gives that, for some $\lambda = \lambda(F) > 0$ and $C = C(F, \eps)$,
\[
  \begin{split}
    \max_{\cC \in \fC} \Delta_p\big((\partial_e \cF \cup \partial_f \cF)[\ext(\cC)]\big)
    & \le \Delta_p(\partial_e \cF \cup \partial_f \cF) \\
    & \le Cn^{-\lambda} \cdot \left(n^{v(F)-3} p^{e(F)-1}\right)^2 = o(1).
  \end{split}
\]
Applying Janson's inequality (\ref{theorem:Janson}), we conclude that
\[
  \xi \le \exp\left(-2\pi_H \cdot (n/2)^{v(F)-3} p^{e(F)-1} + o(1)\right).
\]
Substituting this estimate into~\eqref{eq:Ex-Z-2-upper} and using Lemma \ref{lemma:edges-int-core}, we obtain
\[
  \Ex[Z^2] \le \Ex[Z] + (1+o(1)) \cdot \left(\exp\left(-\pi_F \cdot (n/2)^{v(F)-3} p^{e(F)-1}\right) \cdot \frac{n^3 p}{3! \cdot 2^2}\right)^2.
\]
Finally, recalling~\eqref{eq:Ex-Z-lower}, in order to get the desired conclusion that $\Ex[Z^2] \le (1+o(1)) \cdot \Ex[Z]^2$, it is enough to argue that
\[
  \exp\left(-\pi_F \cdot (n/2)^{v(F)-3} p^{e(F)-1}\right) \cdot \frac{n^3 p}{3! \cdot 2^2} \gg 1.
\]
To see that this is the case, we should first recall equation \eqref{align:pi_F equality}. Note that our upper-bound assumption on $p$ and the assumption that $F$ is $3$-balanced, i.e., $m_3(F) = (e(F)-1)/(v(F)-3)$, gives
\[
  \pi_F \cdot (n/2)^{v(F)-3} p^{e(F)-1} \le \pi_F \cdot 2^{3-v(F)} \cdot (\Theta_F-\eps)^{e(F)-1} \cdot \log n \le \left(3-\frac{1}{m_3(F)}-\frac{\eps}{\Theta_F}\right) \cdot \log n,
\]
whereas our lower-bound assumption on $p$ is that $n^3p \ge \eps n^{3-1/m_3(H)}$.

\section*{Acknowledgments}
The author would like to thank Wojciech Samotij and Maksim Zhukovskii for carefully reading late versions of this paper and providing helpful suggestions that improved its presentation.

\bibliographystyle{amsplain}
\bibliography{bibliography}
\end{document}